\documentclass{amsart}

\usepackage{amssymb}
\usepackage{amsmath}
\usepackage{amsthm}
\usepackage{amsfonts}

\usepackage{a4wide}
\usepackage{longtable}
\usepackage{multirow}
\usepackage{setspace}

\newtheorem{lemma}{Lemma}[section]
\newtheorem{proposition}{Proposition}[section]
\newtheorem{theorem}{Theorem}[section]
\newtheorem{corollary}{Corollary}[section]

\theoremstyle{definition}
\newtheorem{definition}{Definition}[section]
\newtheorem{remark}{Remark}[section]
\newtheorem{example}{Example}[section]
\newtheorem{problem}{Problem}[section]
\newtheorem{conjecture}{Conjecture}[section]

\numberwithin{equation}{section}

\begin{document}


\title[Differential principal factors of pure metacyclic fields]{Differential principal factors and Polya property \\ of pure metacyclic fields}

\author{Daniel C. Mayer}
\address{Naglergasse 53\\8010 Graz\\Austria}
\email{algebraic.number.theory@algebra.at}
\urladdr{http://www.algebra.at}

\thanks{Research supported by the Austrian Science Fund (FWF): projects J 0497-PHY and P 26008-N25}

\subjclass[2000]{Primary 11R16, 11R20, 13F20, secondary 11R27, 11R29}
\keywords{Pure cubic fields, pure quintic fields, pure metacyclic fields, Polya fields, Ostrowski ideals,
strongly ambiguous classes, differential principal factorization types, class groups}

\date{December 05, 2018}



\begin{abstract}
Barrucand and Cohn's theory of principal factorizations
in pure cubic fields \(\mathbb{Q}(\sqrt[3]{D})\)
and their Galois closures \(\mathbb{Q}(\zeta_3,\sqrt[3]{D})\)
with \(3\) types
is generalized to pure quintic fields \(L=\mathbb{Q}(\sqrt[5]{D})\)
and pure metacyclic fields \(N=\mathbb{Q}(\zeta_5,\sqrt[5]{D})\)
with \(13\) possible types.
The classification is based on the Galois cohomology of the unit group \(U_N\),
viewed as a module over the automorphism group \(\mathrm{Gal}(N/K)\)
of \(N\) over the cyclotomic field \(K=\mathbb{Q}(\zeta_5)\),
by making use of theorems by Hasse and Iwasawa on the Herbrand quotient
of the unit norm index \((U_K:N_{N/K}(U_N))\)
by the number \(\#(\mathcal{P}_{N/K}/\mathcal{P}_K)\) of primitive ambiguous principal ideals,
which can be interpreted as principal factors of the different \(\mathfrak{D}_{N/K}\).
The precise structure of the group of differential principal factors
is determined with the aid of kernels of norm homomorphisms
and central orthogonal idempotents.
A connection with integral representation theory is established via
class number relations by Parry and Walter
involving the index of subfield units \((U_N:U_0)\).
Generalizing criteria for the Polya property
of Galois closures \(\mathbb{Q}(\zeta_3,\sqrt[3]{D})\)
of pure cubic fields \(\mathbb{Q}(\sqrt[3]{D})\) by Leriche and Zantema,
we prove that pure metacyclic fields \(N=\mathbb{Q}(\zeta_5,\sqrt[5]{D})\)
of only \(1\) type cannot be Polya fields.
All theoretical results are underpinned by extensive
numerical verifications of the \(13\) possible types
and their statistical distribution in the range
\(2\le D<10^3\) of \(900\) normalized radicands.
\end{abstract}

\maketitle


\section{Introduction and Main Theorems}
\label{s:Intro}
\noindent
Let \(F\) be an algebraic number field.
For each prime number \(p\in\mathbb{P}\) and each integer exponent \(f\ge 1\),
let the \textit{Ostrowski ideal}
\cite{Os}
of \(F\) for the prime power \(p^f\) be defined as
the product of all prime ideals of \(F\) with norm \(p^f\), that is
\begin{equation}
\label{eqn:Ostrowski} 
\mathfrak{b}_{p^f}(F):=\prod\,\lbrace\ \mathfrak{p}\in\mathbb{P}_F\mid N_{F/\mathbb{Q}}(\mathfrak{p})=p^f\ \rbrace.
\end{equation}
According to Zantema
\cite{Za},
\(F\) is called a \textit{Polya field}
\cite{Po},
if the Ostrowski ideals of \(F\) for arbitrary prime powers are principal, that is
\begin{equation}
\label{eqn:Polya} 
(\forall\,p\in\mathbb{P})\,(\forall\,f\in\mathbb{N})\,(\exists\,\Gamma\in F)\,\mathfrak{b}_{p^f}(F)=\Gamma\mathcal{O}_F.
\end{equation}

The aim of the present paper is to provide a
necessary and sufficient condition for the Polya property and a
classification of all
pure metacyclic fields \(N=\mathbb{Q}(\zeta_5,\sqrt[5]{D})\),
with \(\zeta_5=\exp(2\pi\sqrt{-1}/5)\) and \(D\ge 2\) a fifth power free integer,
into precisely \(13\) exhaustive and mutually exclusive types,
according to the Galois cohomology of the unit group \(U_N\),
viewed as a module over the automorphism group \(G=\mathrm{Gal}(N/K)\)
of \(N\) over the cyclotomic field \(K=\mathbb{Q}(\zeta_5)\).
Each type is characterized uniquely by the \(0\)-th cohomology, i.e. the unit norm index,
\begin{equation}
\label{eqn:UNI}
\#(H^0(G,U_N))=(U_K:N_{N/K}(U_N))=5^U \text{ with } 0\le U\le 2
\end{equation}
as the primary classification invariant,
and by the triplet \((A,I,R)\) of \(\mathbb{F}_5\)-dimensions
of the \(5\)-elementary abelian components
\((\mathcal{P}_{L/\mathbb{Q}}:\mathcal{P}_\mathbb{Q})=5^A\),
\(\#\left((\mathcal{P}_{M/K^+}/\mathcal{P}_{K^+})\cap\ker(N_{M/L})\right)=5^I\), and
\(\#\left((\mathcal{P}_{N/K}/\mathcal{P}_K)\cap\ker(N_{N/M})\right)=5^R\)
in the direct product decomposition
\begin{equation}
\label{eqn:AIR}
\begin{aligned}
H^1(G,U_N) & \simeq\mathcal{P}_{N/K}/\mathcal{P}_K \\
& \simeq
(\mathcal{P}_{L/\mathbb{Q}}/\mathcal{P}_\mathbb{Q}) \times
\left((\mathcal{P}_{M/K^+}/\mathcal{P}_{K^+})\cap\ker(N_{M/L})\right) \times
\left((\mathcal{P}_{N/K}/\mathcal{P}_K)\cap\ker(N_{N/M})\right)
\end{aligned}
\end{equation}
of the \(1\)-st cohomology, i.e. the group of primitive ambiguous principal ideals of \(N/K\),
as the secondary classification invariant,
where \(L=\mathbb{Q}(\sqrt[5]{D})\) denotes the pure quintic subfield of \(N\)
and \(M=\mathbb{Q}(\sqrt{5},\sqrt[5]{D})\) denotes the maximal real subfield of \(N\),
which contains \(K^+=\mathbb{Q}(\sqrt{5})\).

Since the Polya property can be proved independently of the classification,
we immediately state and prove our first Main Theorem concerning \textit{Polya fields},
reproved later in Theorem
\ref{thm:MainPolya}.


\begin{theorem}
\label{thm:MainQuinticPolya}
A pure metacyclic field \(N=\mathbb{Q}(\zeta_5,\sqrt[5]{D})\)
of absolute degree \(\lbrack N:\mathbb{Q}\rbrack=20\)
with \(5\)-th power free radicand \(D\in\mathbb{Z}\), \(D\ge 2\),
is a Polya field if and only if
\begin{equation}
\label{eqn:MainQuinticPolya}
(\forall\,p\in\mathbb{P},\ p \text{ ramified in } N/K)\,(\exists\,\alpha\in L)\,N_{L/\mathbb{Q}}(\alpha)=p.
\end{equation}
\end{theorem}

\begin{remark}
\label{rmk:MainQuinticPolya}
It will turn out that only pure metacyclic fields \(N\)
having the single differential principal factorization type \(\alpha_3\) in Theorem
\ref{thm:MainQuinticDisplay} 
cannot possess the Polya property.
For fields \(N\) of the other \(12\) types we shall give conditions
in terms of the prime factorization of the class field theoretic conductor of the Kummer extension \(N/K\)
which are necessary and sufficient for the Polya property of \(N\).
\end{remark}

\begin{proof}
Since the absolute extension \(N/\mathbb{Q}\) is Galois,
it suffices to consider Ostrowski ideals of \(N\) for powers of primes which ramify in \(N/\mathbb{Q}\)
\cite{Za}.
We reduce the principal ideal conditions for Ostrowski ideals of the normal field \(N\)
to principal ideal conditions for Ostrowski ideals of the non-Galois field \(L\).
Then we use the equivalence
\(\mathfrak{b}_p(L)=\alpha\mathcal{O}_L\) \(\Longleftrightarrow\) \(N_{L/\mathbb{Q}}(\alpha)=p\)
for \(\alpha\in L\) and \(p\in\mathbb{P}\) ramified in \(L/\mathbb{Q}\).

Firstly, if \(q\equiv\pm 2\,(\mathrm{mod}\,5)\), then
\(q\mathcal{O}_L=\mathcal{Q}^5\), \(N(\mathcal{Q})=q\), \(\mathfrak{b}_{q}(L)=\mathcal{Q}\), and
\(q\mathcal{O}_N=\mathfrak{Q}^5\), \(N(\mathfrak{Q})=q^4\), \(\mathfrak{b}_{q^4}(N)=\mathfrak{Q}\), and
\begin{equation}
\label{eqn:NonSplitIdl}
\mathfrak{b}_{q}(L)\mathcal{O}_N=\mathcal{Q}\mathcal{O}_N=\mathfrak{Q}=\mathfrak{b}_{q^4}(N).
\end{equation}
Thus, on the one hand, for \(x\in L\),
\begin{equation}
\label{eqn:NonSplitUp}
\mathfrak{b}_{q}(L)=x\mathcal{O}_L\ \Longrightarrow\ \mathfrak{b}_{q^4}(N)=x\mathcal{O}_N,
\end{equation}
and on the other hand, for \(\Gamma\in N\),
\begin{equation}
\label{eqn:NonSplitDown}
\mathfrak{b}_{q^4}(N)=\Gamma\mathcal{O}_N\ \Longrightarrow\ 
N_{N/L}(\Gamma)\mathcal{O}_L=N_{N/L}(\mathfrak{Q})=\mathcal{Q}^4=\mathfrak{b}_{q}(L)^4.
\end{equation}
If \(\mathcal{Q}\) were not principal, then
the class \(\mathcal{Q}\mathcal{P}_L\in\mathrm{Cl}(L)\)
would have order \(2\) or \(4\) in \(\mathrm{Cl}_2(L)\)
and would capitulate in \(N\),
\(T_{N/L}(\mathcal{Q}\mathcal{P}_L)=(\mathcal{Q}\mathcal{O}_N)\mathcal{P}_N=\mathfrak{Q}\mathcal{P}_N=1\),
which is impossible
\cite{Wa1}.

Secondly, if \(\ell\equiv -1\,(\mathrm{mod}\,5)\), then
\(\ell\mathcal{O}_L=\mathcal{L}^5\), \(N(\mathcal{L})=\ell\), \(\mathfrak{b}_{\ell}(L)=\mathcal{L}\), and
\(\ell\mathcal{O}_N=\mathfrak{L}_1^5\mathfrak{L}_2^5\), \(N(\mathfrak{L}_i)=\ell^2\), \(\mathfrak{b}_{\ell^2}(N)=\mathfrak{L}_1\mathfrak{L}_2\), and
\begin{equation}
\label{eqn:2SplitIdl}
\mathfrak{b}_{\ell}(L)\mathcal{O}_N=\mathcal{L}\mathcal{O}_N=\mathfrak{L}_1\mathfrak{L}_2=\mathfrak{b}_{\ell^2}(N).
\end{equation}
Thus, on the one hand, for \(x\in L\),
\begin{equation}
\label{eqn:2SplitUp}
\mathfrak{b}_{\ell}(L)=x\mathcal{O}_L\ \Longrightarrow\ \mathfrak{b}_{\ell^2}(N)=x\mathcal{O}_N,
\end{equation}
and on the other hand, for \(\Gamma\in N\),
\begin{equation}
\label{eqn:2SplitDown}
\mathfrak{b}_{\ell^2}(N)=\Gamma\mathcal{O}_N\ \Longrightarrow\ 
N_{N/L}(\Gamma)\mathcal{O}_L=N_{N/L}(\mathfrak{L}_1\mathfrak{L}_2)=\mathcal{L}^4=\mathfrak{b}_{\ell}(L)^4.
\end{equation}
If \(\mathcal{L}\) were not principal, then
the class \(\mathcal{L}\mathcal{P}_L\in\mathrm{Cl}(L)\)
would have order \(2\) or \(4\) in \(\mathrm{Cl}_2(L)\)
and would capitulate in \(N\),
\(T_{N/L}(\mathcal{L}\mathcal{P}_L)=(\mathcal{L}\mathcal{O}_N)\mathcal{P}_N=(\mathfrak{L}_1\mathfrak{L}_2)\mathcal{P}_N=1\),
which is impossible
\cite{Wa1}.

Next, if \(\ell\equiv +1\,(\mathrm{mod}\,5)\), then
\(\ell\mathcal{O}_L=\mathcal{L}^5\), \(N(\mathcal{L})=\ell\), \(\mathfrak{b}_{\ell}(L)=\mathcal{L}\), and
\(\ell\mathcal{O}_N=\mathfrak{L}_1^5\cdots\mathfrak{L}_4^5\), \(N(\mathfrak{L}_i)=\ell\), \(\mathfrak{b}_{\ell}(N)=\mathfrak{L}_1\cdots\mathfrak{L}_4\), and
\begin{equation}
\label{eqn:4SplitIdl}
\mathfrak{b}_{\ell}(L)\mathcal{O}_N=\mathcal{L}\mathcal{O}_N=\mathfrak{L}_1\cdots\mathfrak{L}_4=\mathfrak{b}_{\ell}(N).
\end{equation}
Thus, on the one hand, for \(x\in L\),
\begin{equation}
\label{eqn:4SplitUp}
\mathfrak{b}_{\ell}(L)=x\mathcal{O}_L\ \Longrightarrow\ \mathfrak{b}_{\ell}(N)=x\mathcal{O}_N,
\end{equation}
and on the other hand, for \(\Gamma\in N\),
\begin{equation}
\label{eqn:4SplitDown}
\mathfrak{b}_{\ell}(N)=\Gamma\mathcal{O}_N\ \Longrightarrow\ 
N_{N/L}(\Gamma)\mathcal{O}_L=N_{N/L}(\mathfrak{L}_1\cdots\mathfrak{L}_4)=\mathcal{L}^4=\mathfrak{b}_{\ell}(L)^4.
\end{equation}
If \(\mathcal{L}\) were not principal, then
the class \(\mathcal{L}\mathcal{P}_L\in\mathrm{Cl}(L)\)
would have order \(2\) or \(4\) in \(\mathrm{Cl}_2(L)\)
and would capitulate in \(N\),
\(T_{N/L}(\mathcal{L}\mathcal{P}_L)=(\mathcal{L}\mathcal{O}_N)\mathcal{P}_N=(\mathfrak{L}_1\cdots\mathfrak{L}_4)\mathcal{P}_N=1\),
which is impossible
\cite{Wa1}.

Finally, the investigation of the special prime \(5\)
must be divided in two parts.
Generally,
\(5\mathcal{O}_K=\mathfrak{p}^4\), \(N(\mathfrak{p})=5\), \(\mathfrak{b}_{5}(K)=\mathfrak{p}\),
and, since \(K\) is a Polya field
\cite{Za},
\(\mathfrak{b}_{5}(K)=\xi\mathcal{O}_K\) for \(\xi\in K\).

If \(5\) divides the fourth power of the conductor of \(N/K\), then
\(5\mathcal{O}_L=\mathcal{P}^5\), \(N(\mathcal{P})=5\), \(\mathfrak{b}_{5}(L)=\mathcal{P}\), and
\(5\mathcal{O}_N=\mathfrak{P}^{20}\), \(N(\mathfrak{P})=5\), \(\mathfrak{b}_{5}(N)=\mathfrak{P}\), and
\begin{equation}
\label{eqn:5IdlK}
\xi\mathcal{O}_N=\mathfrak{b}_{5}(K)\mathcal{O}_N=\mathfrak{p}\mathcal{O}_N=\mathfrak{P}^5=\mathfrak{b}_{5}(N)^5,
\end{equation}
whence \(\mathfrak{b}_{5}(N)\) is principal \(\Longleftrightarrow\) \(\mathfrak{b}_{5}(N)^4\) is principal in \(N\).
Furthermore,
\begin{equation}
\label{eqn:5IdlL}
\mathfrak{b}_{5}(L)\mathcal{O}_N=\mathcal{P}\mathcal{O}_N=\mathfrak{P}^4=\mathfrak{b}_{5}(N)^4.
\end{equation}
Thus on the one hand, for \(x\in L\),
\begin{equation}
\label{eqn:5Up}
\mathfrak{b}_{5}(L)=x\mathcal{O}_L\ \Longrightarrow\ \mathfrak{b}_{5}(N)^4=x\mathcal{O}_N\ \Longrightarrow\ \mathfrak{b}_{5}(N) \text{ is principal},
\end{equation}
and on the other hand, for \(\Gamma\in N\),
\begin{equation}
\label{eqn:5Down}
\mathfrak{b}_{5}(N)^4=\Gamma\mathcal{O}_N\ \Longrightarrow\ 
N_{N/L}(\Gamma)\mathcal{O}_L=N_{N/L}(\mathfrak{P}^4)=\mathcal{P}^4=\mathfrak{b}_{5}(L)^4,
\end{equation}
and thus
\(\mathfrak{b}_{5}(L)=\mathcal{P}=\mathcal{P}^5/\mathcal{P}^4=5\mathcal{O}_L/N_{N/L}(\Gamma)\mathcal{O}_L=(5/N_{N/L}(\Gamma))\mathcal{O}_L\)
is principal.

If \(5\) does not divide the fourth power of the conductor of \(N/K\), then
\(5\mathcal{O}_L=\mathcal{P}_1\mathcal{P}_2^4\), \(N(\mathcal{P}_i)=5\), \(\mathfrak{b}_{5}(L)=\mathcal{P}_1\mathcal{P}_2\), and
\(5\mathcal{O}_N=\mathfrak{P}_1^4\cdots\mathfrak{P}_5^4\), \(N(\mathfrak{P}_i)=5\), \(\mathfrak{b}_{5}(N)=\mathfrak{P}_1\cdots\mathfrak{P}_5\).
Since \(\mathfrak{p}^4\mathcal{O}_N=5\mathcal{O}_N=\mathfrak{P}_1^4\cdots\mathfrak{P}_5^4\),
we see that
\(\mathfrak{b}_{5}(N)=\mathfrak{P}_1\cdots\mathfrak{P}_5=\mathfrak{p}\mathcal{O}_N=\mathfrak{b}_{5}(K)\mathcal{O}_N=\xi\mathcal{O}_N\)
is certainly principal.
\end{proof}


\noindent
Leriche
\cite{Le1,Le2}
has proved the following analogue of our Theorem
\ref{thm:MainQuinticPolya},
giving a necessary and sufficient criterion for the Polya property of
Galois closures \(\mathbb{Q}(\zeta_3,\sqrt[3]{D})\)
of pure cubic fields \(\mathbb{Q}(\sqrt[3]{D})\).

\begin{theorem}
\label{thm:MainCubicPolya}
The normal field \(N=\mathbb{Q}(\zeta_3,\sqrt[3]{D})\)
of a pure cubic field \(L=\mathbb{Q}(\sqrt[3]{D})\)
with cube free radicand \(D\in\mathbb{Z}\), \(D\ge 2\),
is a Polya field if and only if
\begin{equation}
\label{eqn:MainCubicPolya}
(\forall\,p\in\mathbb{P},\ p \text{ ramified in } N/\mathbb{Q}(\zeta_3))\,(\exists\,\alpha\in L)\,N_{L/\mathbb{Q}}(\alpha)=p.
\end{equation}
\end{theorem}


Now we state our second Main Theorem concerning the \textit{classification}.

\begin{theorem}
\label{thm:MainQuinticDisplay}
Each pure metacyclic field \(N=\mathbb{Q}(\zeta_5,\sqrt[5]{D})\)
of absolute degree \(\lbrack N:\mathbb{Q}\rbrack=20\)
with \(5\)-th power free radicand \(D\in\mathbb{Z}\), \(D\ge 2\),
belongs to precisely one of the \(13\) differential principal factorization types in Table
\ref{tbl:QuinticDPFTypesDisplay},
in dependence on the invariant \(U\) and the triplet \((A,I,R)\).


\renewcommand{\arraystretch}{1.1}

\begin{table}[ht]
\caption{Differential principal factorization types of pure metacyclic fields \(N\)}
\label{tbl:QuinticDPFTypesDisplay}
\begin{center}
\begin{tabular}{|c||c||c||ccc|}
\hline
 Type           & \(U\) & \(U+1=A+I+R\) & \(A\) & \(I\) & \(R\) \\
\hline
\(\alpha_1\)    & \(2\) & \(3\) & \(1\) & \(0\) & \(2\) \\
\(\alpha_2\)    & \(2\) & \(3\) & \(1\) & \(1\) & \(1\) \\
\(\alpha_3\)    & \(2\) & \(3\) & \(1\) & \(2\) & \(0\) \\
\(\beta_1\)     & \(2\) & \(3\) & \(2\) & \(0\) & \(1\) \\
\(\beta_2\)     & \(2\) & \(3\) & \(2\) & \(1\) & \(0\) \\
\(\gamma\)      & \(2\) & \(3\) & \(3\) & \(0\) & \(0\) \\
\hline
\(\delta_1\)    & \(1\) & \(2\) & \(1\) & \(0\) & \(1\) \\
\(\delta_2\)    & \(1\) & \(2\) & \(1\) & \(1\) & \(0\) \\
\(\varepsilon\) & \(1\) & \(2\) & \(2\) & \(0\) & \(0\) \\
 \hline
\(\zeta_1\)     & \(1\) & \(2\) & \(1\) & \(0\) & \(1\) \\
\(\zeta_2\)     & \(1\) & \(2\) & \(1\) & \(1\) & \(0\) \\
\(\eta\)        & \(1\) & \(2\) & \(2\) & \(0\) & \(0\) \\
 \hline
\(\vartheta\)   & \(0\) & \(1\) & \(1\) & \(0\) & \(0\) \\
\hline
\end{tabular}
\end{center}
\end{table}


The types \(\delta_1\), \(\delta_2\), \(\varepsilon\)
are characterized additionally by \(\zeta_5\not\in N_{N/K}(U_N)\),
and the types \(\zeta_1\), \(\zeta_2\), \(\eta\)
by \(\zeta_5\in N_{N/K}(U_N)\).
\end{theorem}

\begin{proof}
The proof will be developed
in a sequence of partial results,
as explained in the following summary
of the layout of this paper.
\end{proof}


In \S\
\ref{s:PureMetacyclic},
basic arithmetic and Galois theory of pure quintic fields and their pure metacyclic normal closures
are recalled,
normalized radicands
are introduced,
and formulas for conductors and discriminants
are given.
In \S\S\
\ref{s:SemiLocal}
and
\ref{s:Dimensions},
differential principal factorizations are explained
in a series of steps.
Starting with a general definition of ambiguous ideals in arbitrary (not necessarily normal) extensions,
we define differential factors as primitive ambiguous ideals
by comparing their structure with the shape of the relative different.
Further, we investigate kernels of ideal norm homomorphisms and \(\mathbb{F}_5\)-dimensions.
In \S\S\
\ref{s:Cohomology}
and
\ref{s:Idempotents},
Galois cohomology is used to prove the Main Classification Theorem
\ref{thm:MainQuintic}
and \(\tau\)-invariance is analyzed.
Finally, \S\
\ref{s:CompRslt}
establishes the Main Theorem
\ref{thm:MainPolya}
and underpins all theoretical results in the previous sections
by concrete numerical examples for each of the \(13\) differential principal factorization types
and the statistical distribution of their occurrence
in the range \(2\le D<10^3\) of \(900\) normalized radicands.


\section{Pure metacyclic fields}
\label{s:PureMetacyclic}
\noindent
Let \(q_1,\ldots,q_s\) be pairwise distinct primes
such that \(s\ge 1\) and \(5\) may be among them.
Denote by \(L=\mathbb{Q}(\sqrt[5]{D})\) the \textit{pure quintic number field}
with fifth power free radicand \(D=q_1^{e_1}\cdots q_s^{e_s}\),
where the exponents are integers \(1\le e_i\le 4\).
The field \(L\) is generated by adjoining
the unique real and irrational solution of the pure quintic equation \(X^5-D=0\)
to the rational number field \(\mathbb{Q}\).
It is a non-Galois algebraic number field with signature \((1,2)\)
and thus possesses four conjugate and isomorphic complex fields
\(L_j=\mathbb{Q}(\zeta^{j}\cdot\sqrt[5]{D})\), \(1\le j\le 4\),
where \(\zeta=\zeta_5\) denotes a primitive fifth root of unity,
for instance \(\zeta=\exp(2\pi\sqrt{-1}/5)\).
All arithmetical invariants of the \(L_j\), \(1\le j\le 4\), coincide with those of \(L\).

The normal closure of \(L\) is the compositum \(N=L\cdot K=\mathbb{Q}(\sqrt[5]{D},\zeta)\)
of \(L=\mathbb{Q}(\sqrt[5]{D})\) with the cyclotomic field \(K=\mathbb{Q}(\zeta)\).
\(N\) is a complex \textit{pure metacyclic field} of degree \(20\)
with signature \((0,10)\) whose Galois group
\begin{equation}
\label{eqn:GaloisGroup}
\mathrm{Gal}(N/\mathbb{Q})=\langle\ \sigma,\tau\ \mid\ \sigma^5=1,\ \tau^4=1,\ \sigma\tau=\tau\sigma^2\ \rangle\simeq M_5=C_5\rtimes C_4
\end{equation}
is the semidirect product of two cyclic groups, isomorphic to SmallGroup\((20,3)\)
\cite{BEO1,BEO2}.
The action of the automorphisms is given by
\(\sigma(\sqrt[5]{D})=\zeta\cdot\sqrt[5]{D}\), \(\sigma(\zeta)=\zeta\),
\(\tau(\sqrt[5]{D})=\sqrt[5]{D}\), \(\tau(\zeta)=\zeta^2\).

The cyclotomic field \(K=\mathrm{Fix}(\langle\sigma\rangle)\) is a complex cyclic quartic field
and contains the real quadratic field \(\mathbb{Q}(\sqrt{5})=\mathrm{Fix}(\langle\sigma,\tau^2\rangle)\)
as its maximal real subfield \(K^+=\mathbb{Q}(\zeta+\zeta^{-1})\).
Consequently, there exists a \textit{real intermediate field}
\(M=\mathbb{Q}(\sqrt[5]{D},\sqrt{5})\) of degree \(10\)
between \(L\) and \(N\), which is non-Galois with signature \((2,4)\).
The two quintets of non-Galois subfields of \(N\) are
\begin{equation}
\label{eqn:ConjugateFields}
L_j=L^{\sigma^j}=\mathrm{Fix}(\langle\sigma^{-j}\tau\sigma^j\rangle),\quad M_j=M^{\sigma^j}=\mathrm{Fix}(\langle\sigma^{-j}\tau^2\sigma^j\rangle),\quad \text{ with } 0\le j\le 4.
\end{equation}


\subsection{Normalization of radicands}
\label{ss:Normalization}
\noindent
Let \(p\in\mathbb{P}\setminus\lbrace 2\rbrace\) be an odd prime number.
With an integer \(s\ge 1\), exponents \(1\le e_j\le p-1\) for \(1\le j\le s\),
and pairwise distinct prime numbers \(q_1,\ldots,q_s\in\mathbb{P}\)
(where \(2\) and \(p\) may be among them)
let \(D=\prod_{j=1}^s\,q_j^{e_j}\) be a \(p\)-th power free \textit{radicand}.

Since we desire a bijective correspondence between \(p\)-th power free radicands
and pairwise non-isomorphic \textit{pure number fields} \(L=\mathbb{Q}(\sqrt[p]{D})\) of degree \(p\),
we introduce the following concepts.

\begin{definition}
\label{dfn:Normalization}
For a \(p\)th power free radicand \(D\),
let \(D_k:=\prod\lbrace q\in\mathbb{P}\mid v_q(D)=k\rbrace\) for \(1\le k\le p-1\) be
the \textit{homogeneous component of degree} \(k\) of \(D\),
where \(v_q\) denotes the \(q\)-adic valuation \(\mathbb{Q}^\times\to\mathbb{Z}\),
\(\prod_{\ell\in\mathbb{P}}\,\ell^{n_\ell}\mapsto n_q\).
Then we have \(D=\prod_{k=1}^{p-1}\,D_k^k\), where each \(D_k\) is squarefree.

In contrast, we let \(D^{(1)}:=D\), and for \(2\le k\le p-1\) we construct
\(D^{(k)}\) by forming the \(k\)-th power \(D^k\) of \(D\)
and reducing all occurring exponents modulo \(p\).
Then the minimal positive integer
\(D^{(k_0)}:=\min\lbrace D^{(k)}\mid 1\le k\le p-1\rbrace\)
is called the \textit{normalized radicand} of the field \(L=\mathbb{Q}(\root{p}\of{D})\)
and all the bigger integers \(D^{(k)}\) with \(k\ne k_0\) are called the \textit{co-radicands}.
\end{definition}


\begin{example}
\label{exm:Normalization}
Only in the \textbf{cubic} case \(p=3\), where
\(D^{(1)}=D_1D_2^2\) and \(D^{(2)}=D_1^2D_2^4/D_2^3=D_1^2D_2\),
we can easily achieve the normalization by selecting two coprime integers with \(\gcd(D_1,D_2)=1\)
such that the quadratic component \(D_2<D_1\) is smaller than the linear component.
In this case we can definitely say that \(D^{(1)}\) is normalized and \(D^{(2)}\) is \text{the} co-radicand,
since \(D_1D_2^2<D_1^2D_2\).

In the \textbf{quintic} case \(p=5\), where the principal focus will lie in this work,
we have the following four radicands, one of which is normalized:
\begin{equation}
\label{eqn:QntRad}
\begin{aligned}
D^{(1)} &= D_1D_2^2D_3^3D_4^4, \\
D^{(2)} &= D_1^2D_2^4D_3^6D_4^8/D_3^5D_4^5=D_1^2D_2^4D_3D_4^3, \\
D^{(3)} &= D_1^3D_2^6D_3^9D_4^{12}/D_2^5D_3^5D_4^{10}=D_1^3D_2D_3^4D_4^2, \\
D^{(4)} &= D_1^4D_2^8D_3^{12}D_4^{16}/D_2^5D_3^{10}D_4^{15}=D_1^4D_2^3D_3^2D_4.
\end{aligned}
\end{equation}
Note that generally
\begin{equation}
\label{eqn:ModPowers}
D^{(k)}\equiv D^k(\mathrm{mod}^\times\,\mathbb{Z}^p) \text{ for each } 1\le k\le p-1.
\end{equation}
\end{example}


\begin{proposition}
\label{prp:Normalization}
There is a one-to-one correspondence between isomorphism classes of pure fields \(L=\mathbb{Q}(\root{p}\of{D})\) of degree \(p\)
and normalized \(p\)-th power free radicands \(D\).
\end{proposition}

\begin{proof}
Any pure number field \(L=\mathbb{Q}(\root{p}\of{R})\) of degree \(p\) over \(\mathbb{Q}\)
can be generated by adjoining to the rational number field \(\mathbb{Q}\)
the unique real solution \(\root{p}\of{R}\) of a \textit{pure equation} \(X^p-R=0\)
with \(R\in\mathbb{Q}^\times\) and \(R\not\in(\mathbb{Q}^\times)^p\).
By multiplication of the rational radicand \(R\) by the \(p\)th power of an integer,
we can achieve that \(R=D\) is a \(p\)th power free integer radicand
as defined at the beginning of this section.
Without loss of generality, we may assume that \(D=D^{(1)}\) is the \textit{normalized} radicand of
the field \(L=\mathbb{Q}(\root{p}\of{D})\).
Then the powers \((\root{p}\of{D})^k=\root{p}\of{D^k}\) with exponents \(2\le k\le p-1\)
of the radical \(\root{p}\of{D}\) are also elements of \(L\).
According to Formula
\eqref{eqn:ModPowers},
each of the \textit{co-radicals}
\(\root{p}\of{D^{(k)}}=\root{p}\of{\frac{D^k}{n^p}}=\frac{1}{n}\cdot\root{p}\of{D^k}\)
with some \(n\in\mathbb{Z}\setminus\lbrace 0\rbrace\) is also an element of \(L\).
Different normalized radicands are never multiplicatively congruent modulo \(\mathbb{Z}^p\),
but we have isomorphisms between all fields generated by co-radicals
\begin{equation}
\label{eqn:Iso}
L=\mathbb{Q}(\root{p}\of{D})\simeq\mathbb{Q}(\root{p}\of{D^{(k)}})\text{ for all } 1\le k\le p-1. \qedhere
\end{equation}
\end{proof}


\subsection{Conductor and discriminants}
\label{ss:CondDisc}
\noindent
Let
\(R = q_1\cdots q_s\)
be the squarefree product of all prime divisors of the radicand \(D\) 
of the pure quintic field
\(L=\mathbb{Q}(\root{5}\of{D})\). 
Independently of the exponents \(e_1,\ldots,e_s\), 
the conductor \(f\) of the cyclic quintic relative extension \(N/K\) 
\cite[Thm. 1, p. 103]{Ma1}
is given by

\begin{equation}
\label{eqn:Conductor}
f^4 =
\begin{cases}
5^2 R^4 & \text{ if } D\not\equiv \pm 1,\pm 7\,(\mathrm{mod}\,5^2) \text{ (field of the first species),}\\ 
R^4     & \text{ if } D\equiv \pm 1,\pm 7\,(\mathrm{mod}\,5^2) \text{ (field of the second species).}
\end{cases}
\end{equation}


\noindent 
It is well known that the cyclotomic discriminant takes the value 
\(d_K = +5^3 = 125\), 
and Hilbert's Theorem \(39\) on discriminants of composite fields shows the following result
(\cite[pp. 103--104]{Ma1}).

\begin{theorem}
\label{thm:CondDisc}

The \textbf{metacyclic} discriminant is given by 

\begin{equation}
\label{eqn:MetacycDisc}
d_N = d_K^5\cdot f^{16} = 
\begin{cases}
= 5^{23} R^{16} & \text{ for a field of the first species,}\\
= 5^{15} R^{16} & \text{ for a field of the second species,}
\end{cases}
\end{equation}

the \textbf{intermediate} discriminant is given by 
\begin{equation}
\label{eqn:IntermedDisc}
d_M = 5 d_K^2\cdot f^8 = 
\begin{cases}
= 5^{11} R^8 & \text{ for a field of the first species,}\\ 
= 5^7 R^8    & \text{ for a field of the second species,}
\end{cases}
\end{equation}

and the \textbf{pure quintic} discriminant is given by 
\begin{equation}
\label{eqn:PureDisc}
d_L = d_K\cdot f^4 = 
\begin{cases}
= 5^5 R^4 & \text{ for a field of the first species,}\\ 
= 5^3 R^4 & \text{ for a field of the second species.}
\end{cases}
\end{equation}

\end{theorem} 


\begin{example}
\label{exm:CondDisc}
The smallest radicands \(D\) of pure quintic fields
\(L=\mathbb{Q}(\root{5}\of{D})\) 
with a given species and increasing number of prime divisors are
\begin{itemize}
\item
\(2\), \(6 = 2\cdot 3\), \(42 = 2\cdot 3\cdot 7\),
for fields of the first species with \(\gcd(D,5) = 1\) (species 1b),
\item
\(5\), \(10 = 2\cdot 5\), \(30 = 2\cdot 3\cdot 5\),
for fields of the first species with \(5\mid D\) (species 1a),
\item
\(7\), \(18 = 2\cdot 3^2\), \(126 = 2\cdot 7\cdot 3^2\),
for fields of the second species. 
\end{itemize}
\end{example}


\section{Differential factors and norm kernels}
\label{s:SemiLocal}


\subsection{Ambiguous ideals}
\label{ss:AmbIdl}
\noindent
Let \(E/F\) be a relative extension of algebraic number fields
with relative degree \(d:=\lbrack E:F\rbrack\).
We do not assume that \(E/F\) is a normal extension, and thus
the law of decomposition of a prime ideal \(\mathfrak{p}\in\mathbb{P}_F\) of \(F\) in \(E\) may be arbitrary:
\begin{equation}
\label{eqn:PrmDec}
\begin{aligned}
\mathfrak{p}\mathcal{O}_E &= \mathfrak{P}_1^{e_1}\cdots\mathfrak{P}_g^{e_g},
\text{ where } g=g(\mathfrak{p})\in\mathbb{N}, \text{ and for } 1 \le i\le g:\ \mathfrak{P}_i\in\mathbb{P}_E,\ e_i\in\mathbb{N}, \\
f_i &= \lbrack (\mathcal{O}_E/\mathfrak{P}_i):(\mathcal{O}_F/\mathfrak{p}) \rbrack\in\mathbb{N},\
\sum_{i=1}^g\,e_if_i=d,\ e=e(\mathfrak{p}):=\gcd(e_1,\ldots,e_g),\ e\mid d.
\end{aligned}
\end{equation}
Since a Galois group \(\mathrm{Gal}(E/F)\) of automorphisms may be missing,
we cannot speak about ideals in \(E\) which are \textit{invariant} with respect to \(F\).
However, as a compensation, we define \textit{ambiguous} ideals of \(E\) relative to \(F\).


\begin{definition}
\label{dfn:AmbIdl}
An ideal \(\mathfrak{A}\in\mathcal{I}_E\) is called \textit{ambiguous} with respect to \(F\)
(symbol \(\mathfrak{A}\in\mathcal{I}_{E/F}\))
if there exists a positive integer \(n\in\mathbb{N}\) such that \(\mathfrak{A}^n\in\mathcal{I}_F\),
that is, \(\mathfrak{A}\cdot\mathcal{I}_F\) has finite order in \(\mathcal{I}_E/\mathcal{I}_F\).
\end{definition} 


\begin{lemma}
\label{lem:AmbIdl}
Let \(\mathfrak{A}\in\mathcal{I}_E\) be an ideal of \(E\), then the following assertions are equivalent
\cite{Ma}:
\begin{enumerate}
\item
\((\exists\,n\in\mathbb{N})\,\mathfrak{A}^n\in\mathcal{I}_F\), that is, \(\mathfrak{A}\in\mathcal{I}_{E/F}\),
\item
\((\exists\,e\mid d)\,\mathfrak{A}^e\in\mathcal{I}_F\),
\item
\(\mathfrak{A}^d\in\mathcal{I}_F\),
\item
\(\mathfrak{A}^d=N_{E/F}(\mathfrak{A})\).
\end{enumerate}
\end{lemma} 


\begin{problem}
\label{prb:SemiLocal}
The following two questions can be answered by a
\textit{semi-local} investigation with respect to
the prime ideals \(\mathfrak{p}\in\mathbb{P}_F\) of the base field \(F\)
(whereas a \textit{local} consideration concerns
the prime ideals \(\mathfrak{P}\in\mathbb{P}_E\) of the extension \(E\)):
\begin{enumerate}
\item
Which fractional ideals \(\mathfrak{A}\in\mathcal{I}_E\) are of finite order relative to \(\mathcal{I}_F\)?
\item
Which fractional ideals \(\mathfrak{A}\in\mathcal{I}_E\) become trivial under the relative norm \(N_{E/F}\)?
\end{enumerate}
\end{problem}


The first, respectively second, question of Problem
\ref{prb:SemiLocal}
will be answered by Theorem
\ref{thm:Ambiguity},
respectively Theorem
\ref{thm:TrivProj}.

\begin{theorem}
\label{thm:Ambiguity}
Let \(\mathfrak{A}\in\mathcal{I}_E\) be a fractional ideal of \(E\). Then
\begin{equation}
\label{eqn:Ambiguity}
\mathfrak{A}\in\mathcal{I}_{E/F} \quad \Longleftrightarrow \quad
\mathfrak{A}=\prod_{\mathfrak{p}\in\mathbb{P}_F}\,(\prod_{i=1}^{g(\mathfrak{p})}\,\mathfrak{P}^{\frac{e_i}{e(\mathfrak{p})}})^{v(\mathfrak{p})},
\end{equation}
with exponents \(v(\mathfrak{p})\in\mathbb{Z}\).
\end{theorem}

\begin{proof}
(of Lemma
\ref{lem:AmbIdl}
and Theorem
\ref{thm:Ambiguity})
If \((\exists\,e\mid d)\,\mathfrak{A}^e\in\mathcal{I}_F\),
that is, \(d=e\cdot c\) for some \(c\in\mathbb{N}\), then
\(\mathfrak{A}^d=(\mathfrak{A}^e)^c\in\mathcal{I}_F\).
Conversely, we trivially have \(d\mid d\).
The proof that the condition \((\exists\,n\in\mathbb{N})\,\mathfrak{A}^n\in\mathcal{I}_F\)
implies \(n\mid d\) is contained in the following proof of Theorem
\ref{thm:Ambiguity}.

We select an arbitrary prime ideal \(\mathfrak{p}\in\mathbb{P}_F\) of the base field \(F\)
and consider the condition \(\mathfrak{A}^n\in\mathcal{I}_F\) semi-locally with respect to \(\mathfrak{p}\):
The \(\mathfrak{p}\)-component of \(\mathfrak{A}^n\) lies in \(\mathcal{I}_F\) if and only if
there exists some exponent \(m\in\mathbb{Z}\) such that
\((\prod_{i=1}^g\,\mathfrak{P}_i^{v_i})^n=\mathfrak{p}^m\), where \(v_i:=v_{\mathfrak{P}_i}(\mathfrak{A})\) for each \(1\le i\le g\).
This is equivalent with
\((\prod_{i=1}^g\,\mathfrak{P}_i^{v_i})^n=(\prod_{i=1}^g\,\mathfrak{P}_i^{e_i})^m\),
respectively,
\(\prod_{i=1}^g\,\mathfrak{P}_i^{v_in}=\prod_{i=1}^g\,\mathfrak{P}_i^{e_im}\).
We obtain an equivalent system of linear equations with integer coefficients in the unknown integers \(v_i\),
\[(\forall\,1\le i\le g)\,v_in=e_im,\]
which we can divide by the greatest common divisor \(e\) of \(e_1,\ldots,e_g\),
\[(\forall\,1\le i\le g)\,\frac{v_in}{e}=\frac{e_i}{e}m.\]
A solution of this system is given by \(v_i:=\frac{e_i}{e}\) and \(m:=\frac{n}{e}\).
The \textit{minimal} possible semi-local choice is \(n:=e\), whence \(m=1\).
Finally, we can combine the semi-local solutions to a global solution
by putting \(n\) equal to the \textit{least} common multiple \(\mathrm{lcm}\lbrace e(\mathfrak{p})\mid\mathfrak{p}\in\mathbb{P}_F\rbrace\)
of all semi-local values of \(e=e(\mathfrak{p})\), which is still a divisor of \(d\),
since \(d\) is a common multiple of these values.
\end{proof}


\subsection{Primitive ambiguous ideals}
\label{ss:PrmAmbIdl}
\noindent
There exists an \textit{infinitude} of ambiguous ideals in \(E\) relative to \(F\),
since for instance the infinitely many inert prime ideals are all ambiguous.
We are interested in ambiguous ideals which are not contained in the base field
and therefore we define \textit{primitivity}.

\begin{definition}
\label{dfn:Primitivity}
An element \(\mathfrak{A}\cdot\mathcal{I}_F\) of the quotient group \(\mathcal{I}_E/\mathcal{I}_F\)
is called a \textit{primitive} ideal of \(E\) relative to \(F\). 
\end{definition} 


\begin{corollary}
\label{cor:PrimitiveAmbiguity}
A primitive ambiguous ideal of \(E\) relative to \(F\) possesses the shape
\begin{equation}
\label{eqn:PrimitiveAmbiguity}
\mathfrak{A}\cdot\mathcal{I}_F=\prod_{\mathfrak{p}\in\mathbb{P}_F,\ e(\mathfrak{p})>1}\,
(\prod_{i=1}^{g(\mathfrak{p})}\,\mathfrak{P}^{\frac{e_i}{e(\mathfrak{p})}})^{v(\mathfrak{p})}\cdot\mathcal{I}_F
\end{equation}
with exponents \(v(\mathfrak{p})\in\mathbb{Z}/e(\mathfrak{p})\mathbb{Z}\).
\end{corollary}

\begin{proof}
For \(e(\mathfrak{p})=1\), we have
\(\prod_{i=1}^{g(\mathfrak{p})}\,\mathfrak{P}^{\frac{e_i}{e(\mathfrak{p})}}=\prod_{i=1}^{g(\mathfrak{p})}\,\mathfrak{P}^{e_i}=\mathfrak{p}\mathcal{O}_E\in\mathcal{I}_F\).
This reduces the components of a primitive ambiguous ideal to the finitely many prime ideals \(\mathfrak{p}\in\mathbb{P}_F\) with \(e(\mathfrak{p})>1\).
However, for \(e(\mathfrak{p})>1\) and \(v(\mathfrak{p})=e(\mathfrak{p})\), we have
\((\prod_{i=1}^{g(\mathfrak{p})}\,\mathfrak{P}^{\frac{e_i}{e(\mathfrak{p})}})^{v(\mathfrak{p})}=\prod_{i=1}^{g(\mathfrak{p})}\,\mathfrak{P}^{e_i}=\mathfrak{p}\mathcal{O}_E\in\mathcal{I}_F\).
Thus we can restrict the exponents of the finitely many components to finitely many values \(v(\mathfrak{p})\in\mathbb{Z}/e(\mathfrak{p})\mathbb{Z}\).
\end{proof}


If \(N/F\) with \(F\le E\le N\) is the normal closure of \(E/F\)
with Galois group \(G:=\mathrm{Gal}(N/F)\),
then we set \(\mathcal{I}_{E/F}:=\mathcal{I}_E\cap\mathcal{I}_{N/F}\),
where \(\mathcal{I}_{N/F}:=\lbrace\mathfrak{A}\in\mathcal{I}_N\mid(\forall\,\phi\in G)\,\mathfrak{A}^\phi=\mathfrak{A}\rbrace\).

\begin{corollary}
\label{cor:Finiteness} \quad
\(\mathcal{I}_{E/F}/\mathcal{I}_F\simeq\prod_{\mathfrak{p}\in\mathbb{P}_F}\,\mathbb{Z}/e(\mathfrak{p})\mathbb{Z}\)  \quad is finite with \quad 
\(\#(\mathcal{I}_{E/F}/\mathcal{I}_F)=\prod_{\mathfrak{p}\in\mathbb{P}_F}\,e(\mathfrak{p})\).
\end{corollary}

\begin{proof}
This is an immediate consequence of Corollary
\ref{cor:PrimitiveAmbiguity}.
\end{proof}

\begin{remark}
\label{rmk:PrimeDegree}
If the relative degree \(d=\lbrack E:F\rbrack\) is a prime number \(q\in\mathbb{P}\),
then the condition \(e=e(\mathfrak{p})>1\) is only satisfied by prime ideals \(\mathfrak{p}\in\mathbb{P}_F\)
which \textit{ramify totally} in \(E\),
since \(e\mid d\) with \(e>1\) and \(d=q\) is only possible for \(e=q\) and thus necessarily \(f=g=1\).
Therefore, if \(T:=\#\lbrace\mathfrak{p}\in\mathbb{P}_F\mid e(\mathfrak{p})>1\rbrace\), then we have
\begin{equation}
\label{eqn:PrimeDegree}
\lbrack E:F\rbrack=q\in\mathbb{P} \quad \Longrightarrow \quad 
\mathcal{I}_{E/F}/\mathcal{I}_F\simeq(\mathbb{Z}/q\mathbb{Z})^T, \quad
\#(\mathcal{I}_{E/F}/\mathcal{I}_F)=q^T.
\end{equation}
\end{remark}


\subsection{Relative different}
\label{ss:RelDiff}
\noindent
As the title of the present paper suggests,
our main goal is the investigation of \textit{differential principal factors},
that is, principal ideal divisors \((A)=A\mathcal{O}_N\in\mathcal{P}_N\)
of the relative different \(\mathfrak{D}_{N/K}\) of pure metacyclic fields \(N=\mathbb{Q}(\root{p}\of{D},\zeta)\)
with respect to the cyclotomic subfield \(K=\mathbb{Q}(\zeta)\).
For this purpose, a brief excursion to the general theory will be illuminating.


Let \(E/F\) be an extension of algebraic number fields.
We fix a prime ideal \(\mathfrak{P}\in\mathbb{P}_E\) of \(E\)
and we denote the ramification exponent of \(\mathfrak{P}\)
over the prime ideal \(\mathfrak{p}:=F\cap\mathfrak{P}\in\mathbb{P}_F\) of \(F\) below \(\mathfrak{P}\)
by \(e:=v_{\mathfrak{P}}(\mathfrak{p})\in\mathbb{N}\).

\begin{theorem}
\label{thm:RelDiff}
The contribution of the prime ideal \(\mathfrak{P}\in\mathbb{P}_E\)
to the relative different \(\mathfrak{D}_{E/F}\in\mathcal{I}_E\) of \(E/F\)
depends on the ramification of \(\mathfrak{P}\) in \(E/F\) and is given by
\begin{equation}
\label{eqn:RelDiff}
\begin{cases}
v_{\mathfrak{P}}(\mathfrak{D}_{E/F})=e-1                             & \text{ if } v_{\mathfrak{P}}(e)=0 \text{ (tame ramification)}, \\
e\le v_{\mathfrak{P}}(\mathfrak{D}_{E/F})\le e-1+v_{\mathfrak{P}}(e) & \text{ if } v_{\mathfrak{P}}(e)\ge 1 \text{ (wild ramification)}.
\end{cases}
\end{equation}
The relative discriminant \(\mathfrak{d}_{E/F}\in\mathcal{I}_F\) of \(E/F\)
is the relative norm of the relative different,
\begin{equation}
\label{eqn:RelDisc}
\mathfrak{d}_{E/F}=N_{E/F}(\mathfrak{D}_{E/F}).
\end{equation}
\end{theorem}

\begin{proof}
Formula
\ref{eqn:RelDiff}
is proved in
\cite[Thm. 2.6, pp. 199-200]{Nk}
and formula
\ref{eqn:RelDisc}
is proved in
\cite[Thm. 2.9, pp. 201-202]{Nk}.
Note that, in particular, \(v_{\mathfrak{P}}(\mathfrak{D}_{E/F})=0\) \(\Longleftrightarrow\) \(e=1\) (\(\mathfrak{P}\) unramified).
\end{proof}


Now we apply the general theory to our special situations,
either the \textbf{Kummer} extension \(E:=N=\mathbb{Q}(\root{p}\of{D},\zeta)\) of the cyclotomic base field \(F:=K=\mathbb{Q}(\zeta)\)
or the \textbf{pure} extension \(E:=L=\mathbb{Q}(\root{p}\of{D})\) of the rational base field \(F:=\mathbb{Q}\).
We begin with results for the \textbf{quintic} case \(p=5\).

\begin{theorem}
\label{thm:RelDiffQuintic}
Let \(q\in\mathbb{P}\) be a prime number and
denote by \(\mathfrak{Q}\in\mathbb{P}_N\) a prime ideal of \(N=\mathbb{Q}(\root{5}\of{D},\zeta)\)
dividing the extension ideal \(q\mathcal{O}_N\).
As before, \(f\) denotes the conductor of \(N/K\).
Then the contribution of \(\mathfrak{Q}\) to the relative different \(\mathfrak{D}_{N/K}\) is given by
\begin{equation}
\label{eqn:RelDiffQuinticNK}
v_{\mathfrak{Q}}(\mathfrak{D}_{N/K})=
\begin{cases}
24 & \text{ if } q=5,\ q^6\parallel f^4 \text{ (for \(N\) of species } 1\mathrm{a}), \\
 8 & \text{ if } q=5,\ q^2\parallel f^4 \text{ (for \(N\) of species } 1\mathrm{b}), \\
 4 & \text{ if } q\ne 5,\ q\mid f \text{ (for \(N\) of any species)}, \\
 0 & \text{ if } q\not\vert\, f \text{ (in particular, if } q=5 \text{ and \(N\) of species } 2).
\end{cases}
\end{equation}
Now denote by \(\mathfrak{Q}\in\mathbb{P}_L\) a prime ideal of \(L=\mathbb{Q}(\root{5}\of{D})\)
dividing the extension ideal \(q\mathcal{O}_L\).
Then the contribution of \(\mathfrak{Q}\) to the absolute different \(\mathfrak{D}_{L/\mathbb{Q}}\) is given by
\begin{equation}
\label{eqn:RelDiffQuinticLQ}
v_{\mathfrak{Q}}(\mathfrak{D}_{L/\mathbb{Q}})=
\begin{cases}
 9 & \text{ if } q=5,\ q^6\parallel f^4 \text{ (for \(L\) of species } 1\mathrm{a}), \\
 5 & \text{ if } q=5,\ q^2\parallel f^4 \text{ (for \(L\) of species } 1\mathrm{b}), \\
 3 & \text{ if } q=5,\ q\not\vert\, f,\ v_{\mathfrak{Q}}(q\mathcal{O}_L)=4 \text{ (for \(L\) of species } 2), \\
 0 & \text{ if } q=5,\ q\not\vert\, f,\ v_{\mathfrak{Q}}(q\mathcal{O}_L)=1 \text{ (for \(L\) of species } 2), \\
 4 & \text{ if } q\ne 5,\ q\mid f \text{ (for \(L\) of any species)}, \\
 0 & \text{ if } q\ne 5,\ q\not\vert\, f .
\end{cases}
\end{equation}
\end{theorem}


Remark
\ref{rmk:PrimeDegree}
together with Theorem
\ref{thm:RelDiffQuintic}
explains the designation \textit{differential factors}
for primitive ambiguous ideals
\(\mathfrak{A}\in\mathcal{I}_N\) and \(\mathfrak{a}\in\mathcal{I}_L\)
which are not divisible by a prime ideal above \(5\): \\
\(\mathfrak{A}\cdot\mathcal{I}_K\in\mathcal{I}_{N/K}/\mathcal{I}_K\)
\(\Longleftrightarrow\) \(\mathfrak{A}\mid\mathfrak{R}_{N/K}^4\)
\(\Longleftrightarrow\) \(\mathfrak{A}\mid\mathfrak{D}_{N/K}\),
\(\mathfrak{a}\cdot\mathcal{I}_{\mathbb{Q}}\in\mathcal{I}_{L/\mathbb{Q}}/\mathcal{I}_{\mathbb{Q}}\)
\(\Longleftrightarrow\) \(\mathfrak{a}\mid\mathfrak{R}_{L/\mathbb{Q}}^4\)
\(\Longleftrightarrow\) \(\mathfrak{a}\mid\mathfrak{D}_{L/\mathbb{Q}}\).


\subsection{Kernel of the relative norm}
\label{ss:KerNorm}
\noindent
In Theorem
\ref{thm:TrivProj}
we shall find the answer to the second question of Problem
\ref{prb:SemiLocal}.

\begin{theorem}
\label{thm:TrivProj}
Let \(\mathfrak{A}\in\mathcal{I}_E\) be a fractional ideal of \(E\). Then
\begin{equation}
\label{eqn:TrivProj}
N_{E/F}(\mathfrak{A})=\mathcal{O}_F \quad \Longleftrightarrow \quad
(\forall\,\mathfrak{p}\in\mathbb{P}_F)\,\sum_{\mathfrak{P}\mid\mathfrak{p}}\,f(\mathfrak{P}/\mathfrak{p})\cdot v_{\mathfrak{P}}(\mathfrak{A})=0.
\end{equation}
\end{theorem}

\begin{proof}
Let \(\mathfrak{p}\in\mathbb{P}_F\) be a prime ideal of the base field \(F\).
We consider the condition \(N_{E/F}(\mathfrak{A})=\mathcal{O}_F\) semi-locally with respect to \(\mathfrak{p}\):
\(0=v_{\mathfrak{p}}(\mathcal{O}_F)=v_{\mathfrak{p}}(N_{E/F}(\mathfrak{A}))=v_{\mathfrak{p}}(N_{E/F}(\prod_{i=1}^g\,\mathfrak{P}_i^{v_i}))=
v_{\mathfrak{p}}(\prod_{i=1}^g\,N_{E/F}(\mathfrak{P}_i)^{v_i})=v_{\mathfrak{p}}(\prod_{i=1}^g\,(\mathfrak{p}^{f_i})^{v_i})=
v_{\mathfrak{p}}(\mathfrak{p}^{\sum_{i=1}^g\,f_iv_i})=\sum_{i=1}^g\,f_iv_i\),
where \(v_i=v_{\mathfrak{P}_i}(\mathfrak{A})\).
\end{proof}

\noindent
Note that, in the preceding proof, \(\prod_{i=1}^g\,\mathfrak{P}_i^{v_i}\) is only
the semi-local component of \(\mathfrak{A}\) with respect to \(\mathfrak{p}\).
Other components of  \(\mathfrak{A}\) do not contribute to the valuation \(v_{\mathfrak{p}}\).


\begin{corollary}
\label{cor:TrivProj}
\begin{enumerate}
\item
If \(N_{E/F}(\mathfrak{A})=\mathcal{O}_F\), then only prime ideals \(\mathfrak{p}\in\mathbb{P}_F\) which split in \(E\) can contribute to
\(\mathfrak{A}=\prod_{g(\mathfrak{p})\ge 2}\,\prod_{\mathfrak{P}\mid\mathfrak{p}}\,\mathfrak{P}^{v_{\mathfrak{P}}(\mathfrak{A})}\).
\item
If \(\mathfrak{A}\) is an integral ideal of \(E\),
then \(N_{E/F}(\mathfrak{A})=\mathcal{O}_F \ \Longleftrightarrow \  \mathfrak{A}=\mathcal{O}_E\).
\item
If \(E/F\) is a quadratic extension,
then \(N_{E/F}(\mathfrak{A})=\mathcal{O}_F \ \Longleftrightarrow \ \mathfrak{A}=\prod_{g(\mathfrak{p})=2}\,(\mathfrak{P}_1\mathfrak{P}_2^{-1})^{v_{\mathfrak{P_1}}(\mathfrak{A})}\).
\end{enumerate}
\end{corollary}

\begin{proof}
\begin{enumerate}
\item
For non-split prime ideals \(\mathfrak{p}\in\mathbb{P}_F\) with \(g(\mathfrak{p})=1\),
the sum \(\sum_{\mathfrak{P}\mid\mathfrak{p}}\,f(\mathfrak{P}/\mathfrak{p})\cdot v_{\mathfrak{P}}(\mathfrak{A})\)
degenerates to a single summand \(f(\mathfrak{P}/\mathfrak{p})\cdot v_{\mathfrak{P}}(\mathfrak{A})=0\), which implies \(v_{\mathfrak{P}}(\mathfrak{A})=0\).
\item
If \(\mathfrak{A}\) is integral, then all \(v_{\mathfrak{P}}(\mathfrak{A})\ge 0\) are non-negative.
Consequently, the sum \(\sum_{\mathfrak{P}\mid\mathfrak{p}}\,f(\mathfrak{P}/\mathfrak{p})\cdot v_{\mathfrak{P}}(\mathfrak{A})\)
with positive \(f(\mathfrak{P}/\mathfrak{p})\ge 1\) can be zero only for constant \(v_{\mathfrak{P}}(\mathfrak{A})=0\).
\item
If \(d=\lbrack E:F\rbrack=2\), then the necessary condition \(g(\mathfrak{p})\ge 2\) of item (1) implies
\(g(\mathfrak{p})=2\), \(\mathfrak{p}\mathcal{O}_E=\mathfrak{P}_1\mathfrak{P}_2\), \(f(\mathfrak{P_i}/\mathfrak{p})=1\),
and \(\sum_{i=1}^2\,v_{\mathfrak{P_i}}(\mathfrak{A})=0\), that is, \(v_{\mathfrak{P_2}}(\mathfrak{A})=-v_{\mathfrak{P_1}}(\mathfrak{A})\).\qedhere
\end{enumerate}
\end{proof}


\section{Dimensions of spaces of differential factors}
\label{s:Dimensions}


\subsection{Absolute differential factors}
\label{ss:AbsDF}
\noindent
Let \(T\) be the number of primes \(q_1,\ldots,q_T\)
dividing the conductor of the cyclic quintic extension \(N/K\),
and \(\mathfrak{q}_1,\ldots,\mathfrak{q}_T\) their overlying prime ideals in \(L\).

\begin{theorem}
\label{thm:AbsDim}
The \(\mathbb{F}_5\)-vectorspace \(\mathcal{I}_{L/\mathbb{Q}}/\mathcal{I}_{\mathbb{Q}}\)
of \textbf{absolute} ambiguous ideals has the dimension
\begin{equation}
\label{eqn:AbsDim}
\dim_{\mathbb{F}_5}(\mathcal{I}_{L/\mathbb{Q}}/\mathcal{I}_{\mathbb{Q}})=T,
\end{equation}
which is finite but unbounded. A basis representation is given by
\begin{equation}
\label{eqn:AbsBas}
\mathcal{I}_{L/\mathbb{Q}}/\mathcal{I}_{\mathbb{Q}}=\bigoplus_{i=1}^T\,\mathbb{F}_5\,\mathfrak{q}_i\,.
\end{equation}
\end{theorem}

\begin{proof}
This follows from formulas
\eqref{eqn:PrimitiveAmbiguity}
and
\eqref{eqn:PrimeDegree}
by taking \(d=q=5\).
\end{proof}


\begin{corollary}
\label{cor:AbsDim}
The \(\mathbb{F}_5\)-subspace \(\mathcal{P}_{L/\mathbb{Q}}/\mathcal{P}_{\mathbb{Q}}\le\mathcal{I}_{L/\mathbb{Q}}/\mathcal{P}_{\mathbb{Q}}\)
of absolute ambiguous \textbf{principal} ideals has the dimension
\begin{equation}
\label{eqn:AbsSub}
A:=\dim_{\mathbb{F}_5}(\mathcal{P}_{L/\mathbb{Q}}/\mathcal{P}_{\mathbb{Q}}),
\end{equation}
which is finite and \textbf{bounded} by
\begin{equation}
\label{eqn:AbsBnd}
1\le A\le\min(3,T).
\end{equation}
In particular, the DPF type \(\gamma\) with maximal value \(A=3\) can occur for \(T\ge 3\) only.
\end{corollary}

\begin{proof}
The bound in formula
\eqref{eqn:AbsBnd}
is obtained by combining Theorem
\ref{thm:MainQuintic}
with formula
\eqref{eqn:AbsDim}.
\end{proof}


\subsection{Intermediate and relative differential factors}
\label{ss:IntRelDF}
\noindent
We start with a general statement concerning groups of primitive ambiguous ideals.

\begin{theorem}
\label{thm:BottomTop}
\noindent
Let \(p\) be a prime and \(q>1\) be an integer coprime to \(p\).
Suppose \(E_0/F_0\) is a number field extension of degree \(q\),
\(F/F_0\) is an extension of degree \(p\),
and \(E=F\cdot E_0\) is the compositum of \(F\) and \(E_0\).

Then the norm homomorphism \(N_{E/F}:\,\mathcal{I}_E\to\mathcal{I}_F\) satisfies
\[N_{E/F}(\mathcal{I}_{E/E_0})\le\mathcal{I}_{F/F_0} \text{ and } N_{E/F}(\mathcal{I}_{E_0})\le\mathcal{I}_{F_0},\]
and induces an epimorphism
\[N_{E/F}:\,\mathcal{I}_{E/E_0}/\mathcal{I}_{E_0}\to\mathcal{I}_{F/F_0}/\mathcal{I}_{F_0}.\]

There are isomorphisms of elementary abelian \(p\)-groups
\[\mathcal{I}_{F/F_0}/\mathcal{I}_{F_0}\simeq(\mathcal{I}_{E/E_0}/\mathcal{I}_{E_0})/\ker(N_{E/F})\]
and
\[\mathcal{I}_{E/E_0}/\mathcal{I}_{E_0}\simeq(\mathcal{I}_{F/F_0}/\mathcal{I}_{F_0})\times\ker(N_{E/F}).\]
\end{theorem}

\begin{proof}
The image of an ambiguous ideal \(\mathfrak{A}\in\mathcal{I}_{E/E_0}\)
under the norm \(N_{E/F}\)
is an ambiguous ideal in \(\mathcal{I}_{F/F_0}\):
according to item (3) and (4) of Lemma
\ref{lem:AmbIdl},
we have
\(\mathfrak{A}^p=\mathfrak{n}:=N_{E/E_0}(\mathfrak{A})\),
since
\(\lbrack E:E_0\rbrack=\lbrack (F\cdot E_0):E_0\rbrack=\lbrack F:(F\cap E_0)\rbrack=\lbrack F:F_0\rbrack=p\),
and therefore
\(N_{E/F}(\mathfrak{A})^p=N_{E/F}(\mathfrak{A}^p)=N_{E/F}(\mathfrak{n})=N_{E/F}(N_{E/E_0}(\mathfrak{A}))
=N_{E_0/F_0}(N_{E/E_0}(\mathfrak{A}))=N_{E/F_0}(\mathfrak{A})=N_{F/F_0}(N_{E/F}(\mathfrak{A}))
\in\mathcal{I}_{F_0}\).

So, \(N_{E/F}:\,\mathcal{I}_E\to\mathcal{I}_F\) restricts to
\(N_{E/F}:\,\mathcal{I}_{E/E_0}\to\mathcal{I}_{F/F_0}\),
but the restriction is not surjective:
if \(\mathfrak{L}\in\mathcal{I}_{E/E_0}\) is an inert prime ideal of \(E/F\), then
\(N_{E/F}(\mathfrak{L})=\mathfrak{l}^q\) and thus
\(\mathfrak{l}=\mathfrak{L}\cap\mathcal{O}_F\not\in\mathrm{im}(N_{E/F})\).

However, generally we have
\(N_{E/F}(\mathcal{I}_{E_0})=N_{E_0/F_0}(\mathcal{I}_{E_0})\le\mathcal{I}_{F_0}\)
and there exists an induced mapping \(N_{E/F}:\,\mathcal{I}_{E/E_0}/\mathcal{I}_{E_0}\to\mathcal{I}_{F/F_0}/\mathcal{I}_{F_0}\)
which is an epimorphism,
since \(\gcd(p,q)=1\), respectively \((\exists\,a,b\in\mathbb{Z})\,ap+bq=1\), and thus
\(N_{E/F}(\mathfrak{L}^b)=\mathfrak{l}^{bq}=\mathfrak{l}\cdot\mathfrak{l}^{-ap}
\equiv\mathfrak{l}\,(\mathrm{mod}\,\mathcal{I}_{F_0})\),
since \(\mathfrak{l}^p\in\mathcal{I}_{F_0}\).

The isomorphism theorem yields a quotient representation
\(\mathcal{I}_{F/F_0}/\mathcal{I}_{F_0}\simeq(\mathcal{I}_{E/E_0}/\mathcal{I}_{E_0})/\ker(N_{E/F})\),
and since the groups of primitive ambiguous ideals
\(\mathcal{I}_{F/F_0}/\mathcal{I}_{F_0}\) and \(\mathcal{I}_{E/E_0}/\mathcal{I}_{E_0}\)
are elementary abelian \(p\)-groups
(for instance, \(\mathfrak{A}\in\mathcal{I}_{E/E_0}\) implies \(\mathfrak{A}^p=N_{E/E_0}(\mathfrak{A})\in\mathcal{I}_{E_0}\)),
there is an equivalent direct product representation
\(\mathcal{I}_{E/E_0}/\mathcal{I}_{E_0}\simeq(\mathcal{I}_{F/F_0}/\mathcal{I}_{F_0})\times\ker(N_{E/F})\).
\end{proof}


\noindent
Let \(s_2\), resp. \(s_4\), be the number of prime divisors
\(\ell_1,\ldots,\ell_{s_2}\equiv -1\,(\mathrm{mod}\,5)\), resp.
\(\ell_{s_2+1},\ldots,\ell_{s_2+s_4}\equiv +1\,(\mathrm{mod}\,5)\),
of the conductor \(f\) of \(N/K\), which split in \(M\).
Denote by \(\mathcal{L}_1,\mathcal{L}_1^{\tau},\ldots,\mathcal{L}_{s_2+s_4},\mathcal{L}_{s_2+s_4}^{\tau}\) their overlying prime ideals in \(M\).

\begin{theorem}
\label{thm:IntDim}
The \(\mathbb{F}_5\)-vectorspace \((\mathcal{I}_{M/K^+}/\mathcal{I}_{K^+})\bigcap\ker(N_{M/L})\)
of \textbf{intermediate} ambiguous ideals has the dimension
\begin{equation}
\label{eqn:IntDim}
\dim_{\mathbb{F}_5}\left((\mathcal{I}_{M/K^+}/\mathcal{I}_{K^+})\bigcap\ker(N_{M/L})\right)=s_2+s_4
\end{equation}
which is finite but unbounded. A basis representation is given by
\begin{equation}
\label{eqn:IntBas}
(\mathcal{I}_{M/K^+}/\mathcal{I}_{K^+})\bigcap\ker(N_{M/L})=\bigoplus_{i=1}^{s_2+s_4}\,\mathbb{F}_5\,\mathcal{K}^{(\ell_i)}\,,
\end{equation}
where \(\mathcal{K}^{(\ell_i)}=\mathcal{L}_i^{1+4\tau}=\mathcal{L}_i\cdot(\mathcal{L}_i^{\tau})^4\), for each \(1\le i\le s_2+s_4\).
\end{theorem}

\begin{proof}
The formulas
\eqref{eqn:PrimitiveAmbiguity}
and
\eqref{eqn:PrimeDegree}
with substitutions \(d=q\mapsto 5\) and \(T\mapsto T+s_2+s_4\) imply
\(\dim_{\mathbb{F}_5}(\mathcal{I}_{M/K^+}/\mathcal{I}_{K^+})=T+s_2+s_4\).
Consequently, Theorem
\ref{thm:BottomTop}
yields formula
\eqref{eqn:IntDim}
and Theorem
\ref{thm:TrivProj}
together with item (3) of Corollary
\ref{cor:TrivProj}
yields formula
\eqref{eqn:IntBas}.
\end{proof}


\begin{corollary}
\label{cor:IntDim}
The \(\mathbb{F}_5\)-subspace
\((\mathcal{P}_{M/K^+}/\mathcal{P}_{K^+})\bigcap\ker(N_{M/L})\le(\mathcal{I}_{M/K^+}/\mathcal{P}_{K^+})\bigcap\ker(N_{M/L})\)
of intermediate ambiguous \textbf{principal} ideals has the dimension
\begin{equation}
\label{eqn:IntSub}
I:=\dim_{\mathbb{F}_5}\left((\mathcal{P}_{M/K^+}/\mathcal{P}_{K^+})\bigcap\ker(N_{M/L})\right),
\end{equation}
which is finite and \textbf{bounded} by
\begin{equation}
\label{eqn:IntBnd}
0\le I\le\min(2,s_2+s_4).
\end{equation}
In particular, the DPF type \(\alpha_3\) with maximal value \(I=2\) can occur for \(s_2+s_4\ge 2\) only.
\end{corollary}

\begin{proof}
The bound in formula
\eqref{eqn:IntBnd}
is obtained by combining Theorem
\ref{thm:MainQuintic}
with formula
\eqref{eqn:IntDim}.
\end{proof}


\noindent
Finally, denote by
\(\mathfrak{L}_{s_2+1},\mathfrak{L}_{s_2+1}^{\tau^2},\mathfrak{L}_{s_2+1}^{\tau},\mathfrak{L}_{s_2+1}^{\tau^3},
\ldots,\mathfrak{L}_{s_2+s_4},\mathfrak{L}_{s_2+s_4}^{\tau^2},\mathfrak{L}_{s_2+s_4}^{\tau},\mathfrak{L}_{s_2+s_4}^{\tau^3}\)
the overlying prime ideals in \(N\) of the prime divisors
\(\ell_{s_2+1},\ldots,\ell_{s_2+s_4}\equiv +1\,(\mathrm{mod}\,5)\)
of the conductor \(f\) of \(N/K\) which \(2\)-split in \(M\) and \(4\)-split in \(N\).

\begin{theorem}
\label{thm:RelDim}
The \(\mathbb{F}_5\)-vectorspace \((\mathcal{I}_{N/K}/\mathcal{I}_{K})\bigcap\ker(N_{N/M})\)
of \textbf{relative} ambiguous ideals has the dimension
\begin{equation}
\label{eqn:RelDim}
\dim_{\mathbb{F}_5}((\mathcal{I}_{N/K}/\mathcal{I}_{K})\bigcap\ker(N_{N/M}))=2s_4
\end{equation}
which is finite but unbounded.
A basis representation with generators of \(\tau\)-invariant \(1\)-dimensional \(\mathbb{F}_5\)-subspaces is given by
\begin{equation}
\label{eqn:RelBas}
(\mathcal{I}_{N/K}/\mathcal{I}_{K})\bigcap\ker(N_{N/M})
=\bigoplus_{i=s_2+1}^{s_2+s_4}\,\left(\mathbb{F}_5\,\mathfrak{K}_1^{(\ell_i)}\oplus\mathbb{F}_5\,\mathfrak{K}_2^{(\ell_i)}\right)\,,
\end{equation}
where
\(\mathfrak{K}_1^{(\ell_i)}=\mathfrak{L}_i^{1+4\tau^2+2\tau+3\tau^3}\)
and
\(\mathfrak{K}_2^{(\ell_i)}=\mathfrak{L}_i^{1+4\tau^2+3\tau+2\tau^3}\),
for each \(s_2+1\le i\le s_2+s_4\).
\end{theorem}

\begin{proof}
The formulas
\eqref{eqn:PrimitiveAmbiguity}
and
\eqref{eqn:PrimeDegree}
with substitutions \(d=q\mapsto 5\) and \(T\mapsto T+s_2+3s_4\) imply
\(\dim_{\mathbb{F}_5}(\mathcal{I}_{N/K}/\mathcal{I}_K)=T+s_2+3s_4\).
Consequently, Theorem
\ref{thm:BottomTop}
yields formula
\eqref{eqn:RelDim}.
But Theorem
\ref{thm:TrivProj}
does not yield formula
\eqref{eqn:RelBas}
with generators of \(1\)-dimensional \(\tau\)-invariant subspaces.
Here, we must use Proposition
\ref{prp:Invariance}.
\end{proof}


\begin{corollary}
\label{cor:RelDim}
The \(\mathbb{F}_5\)-subspace
\((\mathcal{P}_{N/K}/\mathcal{P}_{K})\bigcap\ker(N_{N/M})\le(\mathcal{I}_{N/K}/\mathcal{P}_{K})\bigcap\ker(N_{N/M})\)
of relative ambiguous \textbf{principal} ideals has the dimension
\begin{equation}
\label{eqn:RelSub}
R:=\dim_{\mathbb{F}_5}((\mathcal{P}_{N/K}/\mathcal{P}_{K})\bigcap\ker(N_{N/M})),
\end{equation}
which is finite and \textbf{bounded} by
\begin{equation}
\label{eqn:RelBnd}
0\le R\le\min(2,2s_4).
\end{equation}
In particular, the DPF type \(\alpha_1\) with maximal value \(R=2\) can occur for \(s_4\ge 1\) only.
\end{corollary}

\begin{proof}
The bound in formula
\eqref{eqn:RelBnd}
is obtained by combining Theorem
\ref{thm:MainQuintic}
with formula
\eqref{eqn:RelDim}.
\end{proof}


\section{Class number relations and index of subfield units}
\label{s:ClassNumbers}
\noindent
In \(1973\), Charles J. Parry has determined the \textit{class number relation}
between \(h_N=\#\mathrm{Cl}(N)\) and \(h_L=\#\mathrm{Cl}(L)\)
for a pure metacyclic field \(N=\mathbb{Q}(\zeta_5,\sqrt[5]{D})\)
with pure quintic subfield \(L=\mathbb{Q}(\sqrt[5]{D})\):

\begin{equation}
\label{eqn:ParryClNoRel}
h_N=\frac{(U_N:U_0)}{5^5}\cdot h_L^4,\quad U_0=\langle U_K\cdot\prod_{j=0}^4\,U_{L_j}\rangle
\end{equation}

\noindent
\cite[Thm. I, p. 476]{Pa},
where
\(U_N\) denotes the unit group of \(N\),
\(U_0\) is the subgroup of \(U_N\) generated by
all units of the conjugate fields
\(L_j=\mathbb{Q}(\zeta_5^{j}\cdot\root{5}\of{D})\),
\(0\le j\le 4\),
of \(L\) and of \(K\),
and the \textit{index of subfield units} \((U_N:U_0)\) can take
seven possible values \(5^E\) with \(0\le E\le 6\)
\cite[Thm. II, p. 478]{Pa}.

Parry's class number relation is the special case \(p=5\)
of the following general class number formula by Colin D. Walter.

\begin{theorem}
\label{thm:Walter}
Let \(p\) be an odd prime.
For a \(p\)-th power free integer \(D\ge 2\)
let \(N=\mathbb{Q}(\zeta_p,\sqrt[p]{D})\) be a pure metacyclic field of degree \(p(p-1)\)
with pure subfield \(L=\mathbb{Q}(\sqrt[p]{D})\) of degree \(p\).
Denote by \(U_N\) the unit group of \(N\), and by
\(U_0\) the subgroup of \(U_N\) generated by
all units of conjugate fields
\(L_j=\mathbb{Q}(\zeta_p^{j}\cdot\root{p}\of{D})\),
\(0\le j\le p-1\),
of \(L\) and of \(K\).
Then the following class number formula holds:
\begin{equation}
\label{eqn:Walter}
h_N=\frac{(U_N:U_0)}{p^r}\cdot h_K\cdot h_L^{p-1},
\text{ where } r=\frac{p^2-5}{4},\
(U_N:U_0)=p^E,\ E\le\frac{(p-1)(p-2)}{2}.
\end{equation}
\end{theorem}

\begin{proof}
This general class number formula is a consequence
of results on Frobenius extensions by Walter
\cite[Theorem 3.6, p. 222, and Theorem 4.4, p. 223]{Wa1},
who applied the results to pure metacyclic extensions in
\cite[(1.7), (1.9), p. 4]{Wa2}.
\end{proof}


\begin{example}
\label{exm:UnitIndex}
Anticipating some of our numerical results in section
\ref{s:CompRslt}
and the types in section
\ref{s:Cohomology}, 
we give the smallest radicands \(D\) of pure quintic fields
\(L=\mathbb{Q}(\sqrt[5]{D})\) 
where the various values of the exponent \(e\) actually occur:

\begin{itemize}
\item
\(E = 6\) for \(D = 6 = 2\cdot 3\) of type \(\gamma\),
\item
\(E = 5\) for \(D = 2\) of type \(\varepsilon\),
\item
\(E = 4\) for \(D = 22 = 2\cdot 11\) of type \(\beta_2\),
\item
\(E = 3\) for \(D = 11\) of type \(\alpha_2\),
\item
\(E = 2\) for \(D = 31\) of type \(\alpha_1\),
\item
\(E = 1\) for \(D = 33 = 3\cdot 11\) of type \(\alpha_2\).
\end{itemize}

\noindent
However, we point out that we did not find a realization of 
\(E = 0\), 
that is, obviously \(U_N\) is never generated exclusively by subfield units. 
\end{example} 


\section{Galois cohomology and Herbrand quotient of \(U_N\)}
\label{s:Cohomology}
\noindent
Let \(p\) be an odd prime number,
\(D\ge 2\) a \(p\)-th power free integer,
and \(\zeta_p\) be a primitive \(p\)-th root of unity.
Our \textit{classification} of pure metacyclic fields \(N=\mathbb{Q}(\zeta_p,\sqrt[p]{D})\),
which are cyclic Kummer extensions of the cyclotomic field \(K=\mathbb{Q}(\zeta_p)\)
with relative automorphism group \(G=\mathrm{Gal}(N/K)=\langle\sigma\rangle\),
is based on the Galois cohomology of the unit group \(U_N\) viewed as a \(G\)-module.
The \textit{primary invariant} is the group \(H^0(G,U_N)=U_K/N_{N/K}(U_N)\)
of order \(p^U\) with \(0\le U\le\frac{p-1}{2}\)
which is related to the group \(H^1(G,U_N)=(U_N\cap\ker(N_{N/K}))/U_N^{1-\sigma}\)
of order \(p^P\) by the theorem on the Herbrand quotient of \(U_N\),
\begin{equation}
\label{eqn:HerbrandQuot}
\#H^1(G,U_N)=\#H^0(G,U_N)\cdot\lbrack N:K\rbrack, \text{ respectively } P=U+1. 
\end{equation}

The \textit{secondary invariant} is a natural decomposition
of the group \(H^1(G,U_N)\simeq\mathcal{P}_{N/K}/\mathcal{P}_K\)
of \textit{primitive ambiguous principal ideals} of \(N/K\),
which can be viewed as principal ideals dividing the relative different \(\mathfrak{D}_{N/K}\),
and are therefore called \textit{differential principal factors} (DPF) of \(N/K\),
\begin{equation}
\label{eqn:NaturalDecomp}
\mathcal{P}_{N/K}/\mathcal{P}_K\simeq\mathcal{P}_{L/\mathbb{Q}}/\mathcal{P}_{\mathbb{Q}}
\times\left((\mathcal{P}_{N/K}/\mathcal{P}_K)\cap\ker(N_{N/L})\right),
\text{ respectively } P=B+T,
\end{equation}
where \(L=\mathbb{Q}(\sqrt[p]{D})\) denotes the real non-Galois pure subfield of degree \(p\) of \(N\),
the group of \textit{bottom} DPF, \(\mathcal{P}_{L/\mathbb{Q}}/\mathcal{P}_{\mathbb{Q}}\), is of order \(p^B\),
and the group of \textit{top} DPF, \((\mathcal{P}_{N/K}/\mathcal{P}_K)\cap\ker(N_{N/L})\), of order \(p^T\).


Based on the preceding preparation,
we now state our \textit{Main Theorem}
on pure quintic fields \(L=\mathbb{Q}(\sqrt[5]{D})\)
and their Galois closure \(N=\mathbb{Q}(\zeta_5,\sqrt[5]{D})\),
that is the case \(p=5\).
Since there exist intermediate extensions \(\mathbb{Q}<K^+<K\) and \(L<M<N\),
namely the maximal real subfield \(K^+=\mathbb{Q}(\sqrt{5})\) of \(K\)
and the real non-Galois subfield \(M=\mathbb{Q}(\sqrt{5},\sqrt[5]{D})\) of absolute degree \(10\) of \(N\),
the top DPF can be decomposed further
(again using Theorem
\ref{thm:BottomTop},
as in formula
\eqref{eqn:NaturalDecomp})
\begin{equation}
\label{eqn:QuinticDecomp}
\begin{aligned}
& \mathcal{P}_{N/K}/\mathcal{P}_K\simeq\mathcal{P}_{L/\mathbb{Q}}/\mathcal{P}_{\mathbb{Q}}
\times\left((\mathcal{P}_{M/K^+}/\mathcal{P}_{K^+})\cap\ker(N_{M/L})\right)
\times\left((\mathcal{P}_{N/K}/\mathcal{P}_K)\cap\ker(N_{N/M})\right), \\
& \text{ respectively } P=A+I+R,
\end{aligned}
\end{equation}
where the group of \textit{absolute} DPF, \(\mathcal{P}_{L/\mathbb{Q}}/\mathcal{P}_{\mathbb{Q}}\), is of order \(p^A\),
the group of \textit{intermediate} DPF, \((\mathcal{P}_{M/K^+}/\mathcal{P}_{K^+})\cap\ker(N_{M/L})\), is of order \(p^I\),
and the group of \textit{relative} DPF, \((\mathcal{P}_{N/K}/\mathcal{P}_K)\cap\ker(N_{N/M})\), is of order \(p^R\).
Note that \(A=B\), \(I+R=T\), and Theorem
\ref{thm:MainQuintic}
coincides with Theorem
\ref{thm:MainQuinticDisplay}.


\begin{theorem}
\label{thm:MainQuintic}
Each pure metacyclic field \(N=\mathbb{Q}(\zeta_5,\sqrt[5]{D})\)
of absolute degree \(\lbrack N:\mathbb{Q}\rbrack=20\)
with \(5\)-th power free radicand \(D\in\mathbb{Z}\), \(D\ge 2\),
belongs to precisely one of the following \(13\) differential principal factorization types,
in dependence on the invariant \(U\) and the triplet \((A,I,R)\).


\renewcommand{\arraystretch}{1.1}

\begin{table}[ht]
\label{tbl:QuinticDPFTypes}
\begin{center}
\begin{tabular}{|c||c||c||ccc|}
\hline
 Type           & \(U\) & \(U+1=A+I+R\) & \(A\) & \(I\) & \(R\) \\
\hline
\(\alpha_1\)    & \(2\) & \(3\) & \(1\) & \(0\) & \(2\) \\
\(\alpha_2\)    & \(2\) & \(3\) & \(1\) & \(1\) & \(1\) \\
\(\alpha_3\)    & \(2\) & \(3\) & \(1\) & \(2\) & \(0\) \\
\(\beta_1\)     & \(2\) & \(3\) & \(2\) & \(0\) & \(1\) \\
\(\beta_2\)     & \(2\) & \(3\) & \(2\) & \(1\) & \(0\) \\
\(\gamma\)      & \(2\) & \(3\) & \(3\) & \(0\) & \(0\) \\
\hline
\(\delta_1\)    & \(1\) & \(2\) & \(1\) & \(0\) & \(1\) \\
\(\delta_2\)    & \(1\) & \(2\) & \(1\) & \(1\) & \(0\) \\
\(\varepsilon\) & \(1\) & \(2\) & \(2\) & \(0\) & \(0\) \\
 \hline
\(\zeta_1\)     & \(1\) & \(2\) & \(1\) & \(0\) & \(1\) \\
\(\zeta_2\)     & \(1\) & \(2\) & \(1\) & \(1\) & \(0\) \\
\(\eta\)        & \(1\) & \(2\) & \(2\) & \(0\) & \(0\) \\
 \hline
\(\vartheta\)   & \(0\) & \(1\) & \(1\) & \(0\) & \(0\) \\
\hline
\end{tabular}
\end{center}
\end{table}


The types \(\delta_1\), \(\delta_2\), \(\varepsilon\)
are characterized additionally by \(\zeta_5\not\in N_{N/K}(U_N)\),
and the types \(\zeta_1\), \(\zeta_2\), \(\eta\)
by \(\zeta_5\in N_{N/K}(U_N)\).
\end{theorem}

\begin{proof}
The claim is a consequence of combining formulas
\eqref{eqn:HerbrandQuot},
\eqref{eqn:NaturalDecomp}, and
\eqref{eqn:QuinticDecomp}.
\end{proof}

For the sake of comparison, we state the much more simple analogue
for pure cubic fields \(L=\mathbb{Q}(\sqrt[3]{D})\)
and their Galois closure \(N=\mathbb{Q}(\zeta_3,\sqrt[3]{D})\),
that is the case \(p=3\).

\begin{theorem}
\label{thm:MainCubic}
Each pure metacyclic field \(N=\mathbb{Q}(\zeta_3,\sqrt[3]{D})\)
of absolute degree \(\lbrack N:\mathbb{Q}\rbrack=6\)
with cube free radicand \(D\in\mathbb{Z}\), \(D\ge 2\),
belongs to precisely one of the following \(3\) differential principal factorization types,
in dependence on the invariant \(U\) and the pair \((B,T)\):


\renewcommand{\arraystretch}{1.1}

\begin{table}[ht]
\label{tbl:CubicDPFTypes}
\begin{center}
\begin{tabular}{|c||c||c||cc|}
\hline
 Type      & \(U\) & \(U+1=B+T\) & \(B\) & \(T\) \\
\hline
\(\alpha\) & \(1\) & \(2\) & \(1\) & \(1\) \\
\(\beta\)  & \(1\) & \(2\) & \(2\) & \(0\) \\
\hline
\(\gamma\) & \(0\) & \(1\) & \(1\) & \(0\) \\
\hline
\end{tabular}
\end{center}
\end{table}


\end{theorem}

\begin{proof}
A part of the proof is due to Barrucand and Cohn
\cite{BaCo2}
who distinguished \(4\) different types,
\(\mathrm{I}\hat{=}\beta\), \(\mathrm{II}\), \(\mathrm{III}\hat{=}\alpha\), and \(\mathrm{IV}\hat{=}\gamma\).
However, Halter-Koch
\cite{HK}
showed the impossibility of one of these types, namely type \(\mathrm{II}\).
Our new proof is the combination of formulas
\eqref{eqn:HerbrandQuot}
and
\eqref{eqn:NaturalDecomp}.
\end{proof}


\subsection{Herbrand quotient of the units of \(N/K\)}
\label{ss:HerbQuot}

In section
\ref{s:ClassNumbers},
we have seen that the unit index 
\((U_N:U_0)=5^e\)
admits a coarse classification of pure quintic fields 
according to the seven possible values of the exponent \(0\le e\le 6\). 
However, the Galois cohomology of the unit group \(U_N\) 
with respect to the cyclic quintic Kummer extension \(N/K\) 
provides additional structural information. 
The \textit{Herbrand quotient} of \(U_N\) under the action of cyclic group
\(G=\mathrm{Gal}(N/K) = \langle \sigma \rangle\)
is defined by

\begin{equation}
\label{eqn:HerbrandQuotDfn}
\frac{\#\mathrm{H}^0(G,U_N)}{\#\mathrm{H}^{-1}(G,U_N)}=
\frac{(\ker(\Delta):\mathrm{im}(\mathcal{N}))}{(\ker(\mathcal{N}):\mathrm{im}(\Delta))}=
\frac{(U_K:N_{N/K}(U_N))}{(E_{N/K}:U_N^{\sigma - 1})},
\end{equation}

\noindent
where the symbolic difference
\(\Delta : E \mapsto E^{\sigma - 1}\)
and the norm
\(\mathcal{N} : E \mapsto E^{1 + \sigma + \cdots + \sigma^4} \)
are endomorphisms of the unit group \(U_N\) 
and the subgroup \(E_{N/K}:=U_N\cap\ker(N_{N/K})\) consists of the units \(E\) in \(U_N\) 
with relative norm \(N_{N/K}{E} = 1\). 
According to Takagi (\(1920\)), Hasse (\(1927\)) and Herbrand (\(1932\)), 
the Herbrand quotient of the \(G\)-module \(U_N\) has the value

\begin{equation}
\label{eqn:HerbrandQuotThm}
\frac{\#\mathrm{H}^0(G,U_N)}{\#\mathrm{H}^{-1}(G,U_N)} = \frac{1}{\lbrack N : K\rbrack},
\end{equation}

\noindent
since Archimedean places do not yield a contribution, if \(\lbrack N : K\rbrack\) is an odd prime. 


\subsection{Ambiguous principal ideals}
\label{ss:AmbPrId}

Since the cyclotomic unit group \(U_K\) is generated by \(\langle -1, \eta, \zeta \rangle\), 
where \(\eta > 1\) denotes the fundamental unit of \(K^+=\mathbb{Q}(\sqrt{5})\), 
and since
\(N_{N/K}(U_N) \ge N_{N/K}(U_K) = U_K^5 = \langle -1, \eta^5 \rangle\), 
we obtain the possible values of the unit norm index 
\((U_K:N_{N/K}(U_N)) \in \lbrace 1,5,25\rbrace\), 
according to whether
\(N_{N/K}(U_N)\) contains \(\lbrace\eta,\zeta\rbrace\)
or only \(\eta\), resp. \(\zeta\),
or none of them. 

The somewhat abstract quotient
\(E_{N/K}/U_N^{\sigma - 1}\)
is isomorphic to the more ostensive quotient
\(\mathcal{P}_r:=\mathcal{P}_{N/K} / \mathcal{P}_K \)
of the group of \textit{relative ambiguous principal ideals} of \(N/K\)
modulo the subgroup of principal ideals of \(K\), 
according to Iwasawa (or also to Hilbert's Theorems \(92\) and \(94\)), 
and by the Hasse theorem on the Herbrand quotient of the \(G\)-module \(U_N\), we have 

\begin{equation}
\label{eqn:TakagiHasse}
(\mathcal{P}_{N/K}:\mathcal{P}_K) = (E_{N/K}:U_N^{\sigma - 1}) = 5\cdot (U_K:N_{N/K}(U_N)) \in \lbrace 5,25,125\rbrace.
\end{equation}

\noindent
Even in the worst case
\((\mathcal{P}_{N/K}:\mathcal{P}_K) = 5\), 
there are at least the radicals \(\root{5}\of{D},\ldots,\root{5}\of{D^4}\), and the unit \(1\) 
which generate \(5\) distinct \textit{absolute ambiguous principal ideals} of \(L/\mathbb{Q}\). 


\begin{example}
\label{exm:TakagiHasse} 
Anticipating some of our computational results in section \S\
\ref{s:CompRslt}, 
we give the smallest radicands \(D\) of pure quintic fields
\(L=\mathbb{Q}(\root{5}\of{D})\)
where the various values of
\(\#\mathcal{P}_r = (\mathcal{P}_{N/K}:\mathcal{P}_K)\)
actually occur: 

\begin{itemize}
\item
\(\#\mathcal{P}_r = 5\) for \(D = 5\) of type \(\vartheta\), 
\item
\(\#\mathcal{P}_r = 25\) for \(D = 2\) of type \(\varepsilon\), 
\item
\(\#\mathcal{P}_r = 125\) for \(D = 6 = 2\cdot 3\) of type \(\gamma\). 
\end{itemize}

\end{example}


\noindent
Generally, the fixed value of the Herbrand quotient 
for a given type of field extension \(N/K\) 
can be interpreted by the following principle.

\noindent
\textbf{Trade-off Principle.} 
If many units in \(U_K\) can be represented as norms of units in \(U_N\), 
then the extension \(N/K\) contains only few ambiguous principal ideals. \\
If only few units in \(U_K\) can be represented as norms of units in \(U_N\), 
then the extension \(N/K\) contains many ambiguous principal ideals.


\subsection{Differential principal factorizations (DPF)}
\label{ss:DiffPrFact}

As opposed to the claim
\(\mathcal{P}_N \cap \mathcal{I}_K>\mathcal{P}_K\)
of Hilbert's Theorem \(94\) for an unramified cyclic extension \(N/K\) of prime degree with conductor \(f=1\),
the subgroup 
\(1=\mathcal{P}_N \cap \mathcal{I}_K/\mathcal{P}_K < \mathcal{P}_{N/K}/\mathcal{P}_K\), 
the so-called \textit{capitulation kernel} of \(N/K\),
is trivial for our ramified relative extension with \(f>1\),
because the cyclotomic field \(K\) has class number \(h_K = 1\). 
However, the elementary abelian \(5\)-group
\(\mathcal{P}_r = \mathcal{P}_{N/K} / \mathcal{P}_K\), 
whose generators are principal ideals dividing the relative different of \(N/K\), 
so-called \textit{differential principal factors}, 
consists of three nested subgroups, \(\mathcal{P}_r \ge \mathcal{P}_i \ge \mathcal{P}_a\) 
(similar to but not identical with our definitions in
\cite{Ma0}
and in \S\
\ref{s:Dimensions},
since we only want to indicate another slightly different point of view), 

\begin{itemize}
\item
\textit{absolute DPF} of \(L/\mathbb{Q}\), \(\mathcal{P}_a\),
always containing the above mentioned radicals, 
\item
\textit{intermediate DPF} of \(M\vert K^+\), \(\mathcal{P}_i \setminus \mathcal{P}_a\),
such that the proper cosets of \(\mathcal{P}_a\) in \(\mathcal{P}_i\) 
do not project down to \(L/\mathbb{Q}\) by taking the norm \(N_{M/L}\), and 
\item
\textit{relative DPF} of \(N/K\), \(\mathcal{P}_r \setminus \mathcal{P}_i\),
such that the proper cosets of \(\mathcal{P}_i\) in \(\mathcal{P}_r\)
do not even project down to \(M/K^+\) by taking the norm \(N_{N/M}\). 
\end{itemize}


\subsection{DPF types of pure quintic fields}
\label{ss:DPFTypes}

We define thirteen possible differential principal factorization types 
of pure quintic number fields
\(L=\mathbb{Q}(\root{5}\of{D})\), 
according to the generators of the group \(\mathrm{Norm}_{N|K}{U_N}\) as the primary invariant 
and the triplet \((A,I,R)\) defined by \\
the order \(5^A := \#\mathcal{P}_a\) of the group of absolute DPF of \(L/\mathbb{Q}\), \\
the index \(5^I := (\mathcal{P}_i:\mathcal{P}_a)\)
in the group of intermediate DPF of \(M/K^+\), and \\
the index \(5^R := (\mathcal{P}_r:\mathcal{P}_i)\)
in the group of relative DPF of \(N/K\), \\ 
as the secondary invariant 
(similar to but not identical with our definitions in
\cite{Ma0}
and in \S\
\ref{s:Dimensions}). 
The connection between the various quantities is given by the chain of equations 
\[5^{A + I + R} = \#\mathcal{P}_r = (\mathcal{P}_{N/K}:\mathcal{P}_K) = (E_{N/K}:U_N^{\sigma - 1}) = 5\cdot (U_K:N_{N/K}(U_N)).\]
Note that the index of subfield units \((U_N:U_0)\) does not enter the definition of the DPF types. 


\section{Orthogonal idempotents}
\label{s:Idempotents}
\noindent
Firstly, let \(G=\langle\tau\rangle\simeq C_2\) be the unique group of prime order \(2\), which is cyclic.
For subsequent applications, \(G\) may be viewed as the relative group \(\mathrm{Gal}(N/L)\)
of the Galois closure \(N=\mathbb{Q}(\root{3}\of{D},\zeta_3)\) over a pure cubic field \(L=\mathbb{Q}(\root{3}\of{D})\).

\begin{lemma}
\label{Lem:SecondRoots}
Let \(\mu_2=\langle -1\rangle\simeq C_2\) be the group of square roots of unity,
then the character group \(G^\ast:=\mathrm{Hom}(G,\mu_2)\simeq G\) of \(G\)
consists of two characters \(\lbrace 1=\chi_0,\chi_1\rbrace\) such that \(\chi_j(\tau)=-1^j\) for \(0\le j\le 1\).
The values of these characters for \(g\in G\) are shown in Table
\ref{tbl:Char2}.


\begin{table}[ht]
\label{tbl:Char2}
\begin{center}
\begin{tabular}{|c|cc|}
\hline
 \(g\)         & \(1\) & \(\tau=\tau^{-1}\)  \\
\hline
 \(\chi_0(g)\) & \(1\) & \(1\)               \\
 \(\chi_1(g)\) & \(1\) & \(-1\)              \\
\hline
\end{tabular}
\end{center}
\end{table}


\noindent
Two \textit{central orthogonal idempotents} in the group ring \(R=\mathbb{Q}\lbrack G\rbrack\) are given by
\begin{equation}
\label{eqn:IdemChar2}
\psi_j := \frac{1}{2}\sum_{k=0}^1\,\chi_j(\tau^{-k})\tau^k \quad \text{ for } 0\le j\le 1.
\end{equation}
More explicitly, we have
\(\psi_0=\frac{1}{2}(1+\tau)\) and \(\psi_1=\frac{1}{2}(1-\tau)\).
They satisfy the relations \(\psi_0+\psi_1=1\) and \(\psi_i\cdot\psi_j=\delta_{i,j}\psi_j\) for \(0\le i,j\le 1\).
\end{lemma}


For our application to the arithmetic of pure metacyclic fields \(N=\mathbb{Q}(\root{3}\of{D},\zeta_3)\) of degree \(6\),
we can view these central orthogonal idempotents as elements of the group ring \((\mathbb{Z}/3\mathbb{Z})\lbrack G\rbrack\)
using the isomorphism \(\mu_2=\langle -1\rangle\simeq C_2\simeq U(\mathbb{Z}/3\mathbb{Z})\)
which arises by mapping \(-1\mapsto\bar{2}\).
Then we obtain:
\begin{equation}
\label{eqn:Idempotents3}
\psi_0=-(1+\tau),\ \psi_1=-(1+2\tau).
\end{equation}


Secondly, let \(G=\langle\tau\rangle\simeq C_4\) be the cyclic group of order \(4\).
For subsequent applications, \(G\) may be viewed as the relative group \(\mathrm{Gal}(N/L)\)
of the Galois closure \(N=\mathbb{Q}(\root{5}\of{D},\zeta_5)\) over a pure quintic field \(L=\mathbb{Q}(\root{5}\of{D})\).

\begin{lemma}
\label{Lem:FourthRoots}
Let \(\mu_4=\langle\sqrt{-1}\rangle\simeq C_4\) be the group of fourth roots of unity,
then the character group \(G^\ast:=\mathrm{Hom}(G,\mu_4)\simeq G\) of \(G\)
consists of four characters \(\lbrace 1=\chi_0,\chi_1,\chi_2,\chi_3\rbrace\) such that \(\chi_j(\tau)=\sqrt{-1}^j\) for \(0\le j\le 3\).
The values of these characters for \(g\in G\) are shown in Table
\ref{tbl:Char4}.


\begin{table}[ht]
\label{tbl:Char4}
\begin{center}
\begin{tabular}{|c|cccc|}
\hline
 \(g\)         & \(1\) & \(\tau=\tau^{-3}\) & \(\tau^2=\tau^{-2}\) & \(\tau^3=\tau^{-1}\) \\
\hline
 \(\chi_0(g)\) & \(1\) & \(1\)              & \(1\)                & \(1\)                \\
 \(\chi_1(g)\) & \(1\) & \(\sqrt{-1}\)      & \(-1\)               & \(-\sqrt{-1}\)       \\
 \(\chi_2(g)\) & \(1\) & \(-1\)             & \(1\)                & \(-1\)               \\
 \(\chi_3(g)\) & \(1\) & \(-\sqrt{-1}\)     & \(-1\)               & \(\sqrt{-1}\)        \\
\hline
\end{tabular}
\end{center}
\end{table}


\noindent
Four \textit{central orthogonal idempotents} in the group ring \(R=\mathbb{C}\lbrack G\rbrack\) are given generally by
\begin{equation}
\label{eqn:IdemChar4}
\psi_j := \frac{1}{4}\sum_{k=0}^3\,\chi_j(\tau^{-k})\tau^k \quad \text{ for } 0\le j\le 3.
\end{equation}
More explicitly, we have \\
\(\psi_0=\frac{1}{4}(1+\tau+\tau^2+\tau^3),\ \psi_1=\frac{1}{4}(1-\sqrt{-1}\tau-\tau^2+\sqrt{-1}\tau^3),\
\psi_2=\frac{1}{4}(1-\tau+\tau^2-\tau^3)\), and \(\psi_3=\frac{1}{4}(1+\sqrt{-1}\tau-\tau^2-\sqrt{-1}\tau^3)\),
and consequently the \textit{sum relation} \(\psi_0+\psi_1+\psi_2+\psi_3=1\).
\end{lemma}


\begin{proof}
For all \(0\le i,j\le 3\), we have
\(\psi_i \cdot \psi_j =\)
\begin{equation*}
\begin{aligned}
& \frac{1}{4}\sum_{k=0}^3\,\chi_i(\tau^{-k})\tau^k \cdot \frac{1}{4}\sum_{\ell=0}^3\,\chi_j(\tau^{-\ell})\tau^\ell = \\
& \frac{1}{16}\sum_{m=0}^3\,(\sum_{k+\ell=m}\,\chi_i(\tau^{-k})\chi_j(\tau^{-\ell}))\tau^m = \\
& \frac{1}{16}\sum_{m=0}^3\,(\sum_{k=0}^3\,\chi_i(\tau^{-k})\chi_j(\tau^{-(m-k)}))\tau^m = \\
& \frac{1}{16}\sum_{m=0}^3\,(\sum_{k=0}^3\,\chi_i(\tau^{-k})\chi_j(\tau^{k})\chi_j(\tau^{-m}))\tau^m = \\
& \frac{1}{16}\sum_{m=0}^3\,(\sum_{k=0}^3\,(\overline{\chi_i(\tau)}\chi_j(\tau))^{k})\chi_j(\tau^{-m})\tau^m = \\
& \frac{1}{16}\sum_{m=0}^3\,4\delta_{i,j}\chi_j(\tau^{-m})\tau^m = 
\begin{cases}
\psi_i & \text{ for } i=j, \\
0      & \text{ for } i\ne j,
\end{cases}
\end{aligned}
\end{equation*}
since \(\overline{\chi_i(\tau)}\chi_j(\tau)\) is a fixed fourth root of unity \(\zeta\)
which satisfies \(1+\zeta+\zeta^2+\zeta^3=0\) for \(i\ne j\), \(\zeta\ne 1\),
and \(1+\zeta+\zeta^2+\zeta^3=4\) for \(i=j\), \(\zeta=1\).
(Here, \(\delta_{i,j}\) denotes the Kronecker delta.)
\end{proof}


For our application to the arithmetic of pure metacyclic fields \(N=\mathbb{Q}(\root{5}\of{D},\zeta_5)\) of degree \(20\),
we can view the four central orthogonal idempotents as elements of the group ring \((\mathbb{Z}/5\mathbb{Z})\lbrack G\rbrack\)
using the isomorphism \(\mu_4=\langle\sqrt{-1}\rangle\simeq C_4\simeq U(\mathbb{Z}/5\mathbb{Z})\)
which arises by mapping \(\sqrt{-1}\mapsto\bar{3}\), \(-1\mapsto\bar{4}\), \(-\sqrt{-1}\mapsto\bar{2}\).
Then we obtain:
\begin{equation}
\label{eqn:Idempotents5}
\psi_0=-(1+\tau+\tau^2+\tau^3),\ \psi_1=-(1+2\tau+4\tau^2+3\tau^3),\
\psi_2=-(1+4\tau+\tau^2+4\tau^3),\ \psi_3=-(1+3\tau+4\tau^2+2\tau^3).
\end{equation}


\subsection{Invariance}
\label{ss:Invariance}

Based on the preceding discussion of central orthogonal idempotents,
let us now consider the \(\tau\)-invariance of associated particular ambiguous ideals in the semi-local \(\mathbb{F}_5\)-subspace
\[\mathbb{F}_5\mathfrak{L}\oplus\mathbb{F}_5\mathfrak{L}^{\tau^2}\oplus\mathbb{F}_5\mathfrak{L}^{\tau}\oplus\mathbb{F}_5\mathfrak{L}^{\tau^3}
<\mathcal{I}_{N/K}/\mathcal{I}_K\]
with respect to a prime number \(\ell\in\mathbb{P}\) such that \(\ell\equiv 1\,(\mathrm{mod}\,5)\),
\(\ell\mathcal{O}_L=\mathfrak{l}^5\), \(\mathfrak{l}\mathcal{O}_M=\mathcal{L}\cdot\mathcal{L}^{\tau}\), and
\(\mathcal{L}\mathcal{O}_N=\mathfrak{L}\cdot\mathfrak{L}^{\tau^2}\),
\(\mathcal{L}^{\tau}\mathcal{O}_N=\mathfrak{L}^{\tau}\cdot\mathfrak{L}^{\tau^3}\)
with \(\ell\in\mathbb{P}_L\), \(\mathcal{L}\in\mathbb{P}_M\), and \(\mathfrak{L}\in\mathbb{P}_N\).

\begin{proposition}
\label{prp:Invariance}
\begin{enumerate}
\item
The \(4\)-norm image element
\(\mathfrak{l}\mathcal{O}_N=\mathfrak{L}^{1+\tau^2+\tau+\tau^3}\hat{=}(1111)\)
is \(\tau\)-invariant.
\item
The \(2\)-norm kernel element
\(\mathcal{K}\mathcal{O}_N:=\mathcal{L}\mathcal{O}_N(\mathcal{L}^{\tau}\mathcal{O}_N)^4=\mathfrak{L}^{1+\tau^2+4\tau+4\tau^3}\hat{=}(1414)\)
is not \(\tau\)-invariant but is mapped to its (linearly dependent) inverse \(\mathcal{K}^{-1}\mathcal{O}_N\hat{=}(4141)\) by \(\tau\).
\item
The \(4\)-norm kernel element
\(\mathfrak{K}_1:=\mathfrak{L}^{1+4\tau^2+2\tau+3\tau^3}\hat{=}(1243)\)
is not \(\tau\)-invariant but is mapped to its (linearly dependent) third power \(\mathfrak{K}_1^3\hat{=}(3124)\) by \(\tau\).
Similarly, \\
the \(4\)-norm kernel element
\(\mathfrak{K}_2:=\mathfrak{L}^{1+4\tau^2+3\tau+2\tau^3}\hat{=}(1342)\)
is not \(\tau\)-invariant but is mapped to its (linearly dependent) second power \(\mathfrak{K}_2^2\hat{=}(2134)\) by \(\tau\).
\item
The \(4\)-norm kernel element
\(\mathfrak{L}^{1+4\tau^2+\tau+4\tau^3}\hat{=}(1144)\)
is not \(\tau\)-invariant and is mapped to the linearly independent \(4\)-norm kernel element \((4114)\) by \(\tau\).
Similarly, \\
\(\mathfrak{L}^{1+4\tau^2}\hat{=}(1040)\), resp. \((0104)\), is mapped to the linearly independent \((0104)\), resp. \((4010)\).
\end{enumerate}
\end{proposition}


\begin{proof}
\begin{enumerate}
\item
\((\mathfrak{L}^{1+\tau^2+\tau+\tau^3})^{\tau}=\mathfrak{L}^{\tau+\tau^3+\tau^2+1}\hat{=}(1111)\).
\item
\((\mathfrak{L}^{1+\tau^2+4\tau+4\tau^3})^{\tau}=\mathfrak{L}^{\tau+\tau^3+4\tau^2+4}\hat{=}(4141)\).
\item
\((\mathfrak{L}^{1+4\tau^2+2\tau+3\tau^3})^{\tau}=\mathfrak{L}^{\tau+4\tau^3+2\tau^2+3}\hat{=}(3124)\),
\((\mathfrak{L}^{1+4\tau^2+3\tau+2\tau^3})^{\tau}=\mathfrak{L}^{\tau+4\tau^3+3\tau^2+2}\hat{=}(2134)\).
\item
\((\mathfrak{L}^{1+4\tau^2+\tau+4\tau^3})^{\tau}=\mathfrak{L}^{\tau+4\tau^3+\tau^2+4}\hat{=}(4114)\),
\((\mathfrak{L}^{1+4\tau^2})^{\tau}=\mathfrak{L}^{\tau+4\tau^3}\hat{=}(0104)\), \\
\((\mathfrak{L}^{\tau+4\tau^3)^{\tau}=\mathfrak{L}^{\tau^2+4}}\hat{=}(4010). \qedhere\)
\end{enumerate}
\end{proof}


\begin{theorem}
\label{thm:KernelDimension}
For a pure metacyclic field with one of the types \(\alpha_2\), \(\beta_1\), \(\delta_1\), \(\zeta_1\),
the one-dimensional relative principal factorization in the case \(s_4=1\) must be generated by
\[\text{either } \quad \mathfrak{K}_1:=\mathfrak{L}^{1+4\tau^2+2\tau+3\tau^3}\hat{=}(1243)
 \quad \text{ or }  \quad \mathfrak{K}_2:=\mathfrak{L}^{1+4\tau^2+3\tau+2\tau^3}\hat{=}(1342).\] 
In any case \(s_4\ge 1\), it cannot be generated by any of the ideals
\[\mathfrak{L}^{1+4\tau^2+\tau+4\tau^3}\hat{=}(1144)  \quad \text{ or } \quad
\mathfrak{L}^{1+4\tau^2}\hat{=}(1040)  \quad \text{ or } \quad
\mathfrak{L}^{\tau+4\tau^3}\hat{=}(0104).\]
\end{theorem}


\begin{proof}
In the case \(s_4=1\),
the only two \(\tau\)-invariant \(1\)-dimensional subspaces of the
\(2\)-dimensional \(\mathbb{F}_5\)-vectorspace of primitive relatively ambiguous ideals
\((\mathcal{I}_{N/K}/\mathcal{I}_K)\bigcap\ker{N_{N/L}}\)
are generated by
\[\mathfrak{K}_1:=\mathfrak{L}^{1+4\tau^2+2\tau+3\tau^3}\hat{=}(1243)
\quad \text{ and } \quad \mathfrak{K}_2:=\mathfrak{L}^{1+4\tau^2+3\tau+2\tau^3}\hat{=}(1342),\]
according to Proposition
\ref{prp:Invariance}.
If, for instance, \(\mathfrak{K}_1=A\mathcal{O}_N\) is principal, then
\((3124)\hat{=}\mathfrak{L}^{\tau+4\tau^3+2\tau^2+3}=(\mathfrak{L}^{1+4\tau^2+2\tau+3\tau^3})^{\tau}=A^{\tau}\mathcal{O}_N\)
is also principal, but being equal to the third power \(\mathfrak{K}_1^3\)
it also belongs to the \(1\)-dimensional subspace generated by \(\mathfrak{K}_1\).

In any case \(s_4\ge 1\), we give the proof exemplarily for \(\mathfrak{L}^{1+4\tau^2}\).
If \((1040)\hat{=}\mathfrak{L}^{1+4\tau^2}=A\mathcal{O}_N\) were principal,
then \((0104)\hat{=}\mathfrak{L}^{\tau+4\tau^3}=(\mathfrak{L}^{1+4\tau^2})^{\tau}=A^{\tau}\mathcal{O}_N\) would also be principal,
according to Proposition
\ref{prp:Invariance}.
Since \((1040)\) and \((0104)\) are linearly independent over \(\mathbb{F}_5\),
we would obtain a two-dimensional relative principal factorization,
which is impossible for the types \(\alpha_2\), \(\beta_1\), \(\delta_1\), \(\zeta_1\).
(This can occur only for a field with type  \(\alpha_1\).)
\end{proof}


\subsection{Application to metacyclic fields}
\label{ss:Applications}

For the investigation of differential principal factors,
i.e., \textit{ambiguous principal ideals}, in the relative extensions
\(L/\mathbb{Q}\), \(M/K^+\) and \(N/K\),
it is illuminating to begin by determining the order of the
finite group of all \textit{primitive ambiguous ideals} of these extensions.
The word \lq\lq primitive\rq\rq\ always refers to the base field of an extension.
Denoting by \(G:=\mathrm{Gal}(N/K)\)
the subgroup of the metacyclic group \(\mathrm{Gal}(N/\mathbb{Q})=\langle\sigma,\tau\rangle\)
which is generated by the automorphism \(\sigma\) of order \(5\),
we let \(\mathcal{I}_N^G:=\lbrace\mathfrak{A}\in\mathcal{I}_N\mid\mathfrak{A}^\sigma=\mathfrak{A}\rbrace\),
\(\mathcal{I}_M^G:=\mathcal{I}_M\cap\mathcal{I}_N^G\) and
\(\mathcal{I}_L^G:=\mathcal{I}_L\cap\mathcal{I}_N^G\).
Note that these groups correspond to
\(\mathcal{I}_{N/K}=\mathcal{I}_N^G\), \(\mathcal{I}_{M/K^+}=\mathcal{I}_M^G\) and \(\mathcal{I}_{L/\mathbb{Q}}=\mathcal{I}_L^G\).


\begin{theorem}
\label{thm:AmbIdl}
Let \(R\) be the product of all prime numbers ramified in \(L\) and put
\begin{equation}
\label{eqn:Counters}
\begin{aligned}
T       &:=\#\lbrace q\in\mathbb{P}\mid v_q(R)=1\rbrace, \text{ that is } R=q_1\cdots q_T, \\
s_{2,4} &:=\#\lbrace\ell\in\mathbb{P}\mid v_\ell(R)=1,\ \ell\equiv\pm 1\,(\mathrm{mod}\,5)\rbrace \text{ and} \\
s_4     &:=\#\lbrace\ell\in\mathbb{P}\mid v_\ell(R)=1,\ \ell\equiv +1\,(\mathrm{mod}\,5)\rbrace.
\end{aligned}
\end{equation}
\begin{enumerate}
\item
The group of \textbf{absolute} primitive ambiguous ideals is of order \(5^T\),
\begin{equation}
\label{eqn:AbsAmb}
\mathcal{I}_L^G/\mathcal{I}_{\mathbb{Q}}\simeq\prod_{j=1}^T\,\langle q_j\rangle/\langle q_j\rangle^5\simeq(\mathbb{Z}/5\mathbb{Z})^T,
\end{equation}
and it is characterized uniquely with the aid of absolute norms in \(\mathbb{Z}\).
\item
The group of \textbf{intermediate} primitive ambiguous ideals is of order \(5^{T+s_{2,4}}\),
\begin{equation}
\label{eqn:IntAmb}
\mathcal{I}_M^G/\mathcal{I}_{K^+}\simeq(\mathbb{Z}/5\mathbb{Z})^{T+s_{2,4}}
\end{equation} 
and
\((\mathcal{I}_M^G/\mathcal{I}_{K^+})/\ker(N_{M/L})\simeq\mathcal{I}_L^G/\mathcal{I}_{\mathbb{Q}}\),
where the norm kernel of order \(5^{s_{2,4}}\) is
\begin{equation}
\label{eqn:NormKernelL}
\ker(N_{M/L})=
\prod_{\ell\mid R,\ \ell\equiv\pm 1(5)}\,\langle\mathcal{L}^{1-\tau}\rangle/\langle\mathcal{L}^{1-\tau}\rangle^5.
\end{equation}
Here it is assumed that \(\ell\mathcal{O}_M=(\mathcal{L}^{1+\tau})^5\) with \(\mathcal{L}\in\mathbb{P}_M\) for \(\ell\equiv\pm 1\,(\mathrm{mod}\,5)\).
\item
The group of \textbf{relative} primitive ambiguous ideals is of order \(5^{T+s_{2,4}+2s_4}\),
\begin{equation}
\label{eqn:RelAmb}
\mathcal{I}_N^G/\mathcal{I}_K\simeq(\mathbb{Z}/5\mathbb{Z})^{T+s_{2,4}+2s_4}
\end{equation} 
and
\((\mathcal{I}_N^G/\mathcal{I}_K)/\ker(N_{N/M})\simeq\mathcal{I}_M^G/\mathcal{I}_{K^+}\),
where the norm kernel of order \(5^{2s_4}\) is
\begin{equation}
\label{eqn:NormKernelM}
\ker(N_{N/M})=
\prod_{\ell\mid R,\ \ell\equiv +1(5)}\,(\langle\mathfrak{L}^{1-\tau^2}\rangle/\langle\mathfrak{L}^{1-\tau^2}\rangle^5\times\langle\mathfrak{L}^{\tau-\tau^3}\rangle/\langle\mathfrak{L}^{\tau-\tau^3}\rangle^5)
\end{equation}
Here it is assumed that \(\ell\mathcal{O}_N=(\mathfrak{L}^{1+\tau^2+\tau+\tau^3})^5\) with \(\mathfrak{L}\in\mathbb{P}_N\) for \(\ell\equiv +1\,(\mathrm{mod}\,5)\). \\
\textbf{Summarized} and expressed as a direct product of elementary abelian \(5\)-groups:
\begin{equation}
\label{eqn:AmbIdl}
\mathcal{I}_N^G/\mathcal{I}_K\simeq\mathcal{I}_L^G/\mathcal{I}_{\mathbb{Q}}\times\ker(N_{M/L})\times\ker(N_{N/M}).
\end{equation}
\end{enumerate}
\end{theorem}


\begin{proof}
Generally, for any prime number \(q\in\mathbb{P}\) which divides \(R\), that is \(v_q(R)=1\),
we have full ramification \(q\mathcal{O}_L=\mathfrak{q}^5\) in \(L\) with a prime ideal \(\mathfrak{q}\in\mathbb{P}_L\).
\begin{enumerate}
\item
Since the extension \(L/\mathbb{Q}\) is of prime degree \(5\),
an ambiguous prime ideal \(\mathfrak{q}\) in \(\mathcal{I}_L^G\) is
either inert
or totally ramified over \(\mathbb{Q}\).
In the former case it is imprimitive, \(\mathfrak{q}\in\mathcal{I}_{\mathbb{Q}}\),
and in the latter case it lies over a prime divisor \(q\) of \(R\).
The relation \(\mathfrak{q}^5=\mathfrak{q}^{1+\sigma+\ldots+\sigma^4}=N_{L/\mathbb{Q}}\mathfrak{q}=q\mathbb{Z}\in\mathcal{I}_{\mathbb{Q}}\)
shows that a primitive ambiguous ideal must be free of fifth powers, whence
\(\mathcal{I}_L^G/\mathcal{I}_{\mathbb{Q}}\simeq\prod_{q\mid R}\,\lbrace\mathfrak{q}^{v_q}\mid 0\le v_q\le 4\rbrace\)
and multiplication involves reduction of exponents modulo \(5\).
If \(R=q_1\cdots q_T\),
the norm map \(N_{L/\mathbb{Q}}\) establishes an isomorphism
\(\mathcal{I}_L^G/\mathcal{I}_{\mathbb{Q}}\simeq
N_{L/\mathbb{Q}}(\mathcal{I}_L^G)/(\mathbb{Q}^\times)^5=
(\mathbb{Q}^\times)^5\cdot\langle q_1,\ldots,q_T\rangle/(\mathbb{Q}^\times)^5
\simeq\langle q_1,\ldots,q_T\rangle/(\langle q_1,\ldots,q_T\rangle\cap(\mathbb{Q}^\times)^5)=\)\\
\(\langle q_1,\ldots,q_T\rangle/\langle q_1,\ldots,q_T\rangle^5
\simeq(\langle q_1\rangle/\langle q_1\rangle^5)\times\ldots\times(\langle q_T\rangle/\langle q_T\rangle^5)
\simeq(\mathbb{Z}/5\mathbb{Z})^T\).
\item
For the intermediate extension \(M/K^+\), we have to take into account the decomposition law
for the primes \(q\equiv\pm 1\,(\mathrm{mod}\,5)\).
A similar reasoning as in item (1) shows that \(\mathcal{I}_M^G/\mathcal{I}_{K^+}\simeq\mathcal{I}_1\times\mathcal{I}_2\)
is the direct product of two components,
\(\mathcal{I}_1:=\prod_{q\not\equiv\pm 1\,(\mathrm{mod}\,5)}\,\lbrace\mathcal{Q}^{v_q}\mid 0\le v_q\le 4\rbrace\)
of order \(5^{t-s_2}\) and, semi-locally with respect to \(q\),
\(\mathcal{I}_2:=\prod_{q\equiv\pm 1\,(\mathrm{mod}\,5)}\,\lbrace\mathcal{Q}^{v_q}(\mathcal{Q}^\tau)^{v_q^\prime}\mid 0\le v_q,v_q^\prime\le 4\rbrace\)
of order \((5^2)^{s_2}=5^{2s_2}\).\\
The norm map \(N_{M/L}\) is injective on \(\mathcal{I}_1\)
but has kernel
\(\ker(N_{M/L})=
\prod_{q\equiv\pm 1(5)}\,\lbrace\mathcal{Q}^{v_q}(\mathcal{Q}^\tau)^{v_q^\prime}\mid v_q+v_q^\prime\equiv 0\,(\mathrm{mod}\,5)\rbrace\)
of order \(5^{s_2}\) on \(\mathcal{I}_2\).
For the justification of these claims observe, firstly,
that \(N_{M/L}(\mathcal{Q})=\mathfrak{q}^2\), with exponent \(2\) coprime to \(5\),
for \(q\not\equiv\pm 1\,(\mathrm{mod}\,5)\), \(q\ne 5\), since \(\mathfrak{q}\) remains inert in \(M\),
secondly, that \(N_{M/L}(\mathcal{Q})=\mathfrak{q}\) for \(q=5\), since \(\mathfrak{q}\mathcal{O}_M=\mathcal{Q}^2\) ramifies in \(M\).
\item
Similar as (2). Formula
\eqref{eqn:AmbIdl} corresponds to formulas
\eqref{eqn:AIR}
and
\eqref{eqn:QuinticDecomp}.\qedhere
\end{enumerate}
\end{proof}
 

\section{Rational congruence conditions and asymptotic limit densities}
\label{s:Asymptotic}
\noindent
Using connections between differential principal factorization types and invertible residue classes,
we prove that the asymptotic limit density of certain types is zero.

\noindent
The existence, resp. the lack, of certain prime divisors of the conductor \(f\) 
permits some criteria for the classification of pure quintic fields. 


\begin{theorem}
\label{thm:DiffPrFact}

\begin{enumerate}
\item
\(\zeta \in N_{N/K}(U_N)\) can occur
only when \(f\) is divisible by no other primes than 
\(5\) or primes \(q_j \equiv\pm 1,\pm 7\,(\mathrm{mod}\,5^2)\). 
\item
Relative DPF of \(N/K\) can occur
only if some prime \(q_j \equiv +1\,(\mathrm{mod}\,5)\) divides \(f\). 
\item
Intermediate DPF of \(M/K^+\) can occur
only if some prime \(q_j \equiv \pm 1\,(\mathrm{mod}\,5)\) divides \(f\). 
\item
Absolute DPF of \(L/\mathbb{Q}\), distinct from radicals, must exist
when no prime \(\equiv \pm 1\,(\mathrm{mod}\,5)\) divides \(f\) 
and \(f\) has at least one prime factor
distinct from \(5\) and from primes \(\equiv \pm 7\,(\mathrm{mod}\,5^2)\). 
\end{enumerate}

\end{theorem}


\begin{proof} 
\begin{enumerate}
\item
The condition for \(\zeta\in N_{N/K}(U_N)\) is due to 
the properties of the quintic Hilbert symbol (fifth power norm residue symbol) over \(K\). 
\item
The condition for relative DPF of \(N/K\) is a consequence of
the decomposition law \((e,f,g) = (1,1,4)\) for primes \(\equiv +1\,(\mathrm{mod}\,5)\) in \(K\). 
\item
The condition for intermediate DPF of \(M/K^+\) is due to
the decomposition types \((e,f,g) = (1,1,2)\), resp. \((1,2,2)\),
of primes \(\equiv -1\,(\mathrm{mod}\,5)\) in \(K^+\), resp. \(K\), 
and \((e,f,g) = (1,1,2)\) of primes \(\equiv +1\,(\mathrm{mod}\,5)\) in \(K^+\). \qedhere
\end{enumerate}
\end{proof} 


\begin{theorem}
\label{thm:PQ}
The asymptotic limit density of pure quintic fields \(L=\mathbb{Q}(\sqrt[5]{D})\)
with \(5\)th power free radicand \(D\) and one of the differential principal factorization types
\(\zeta_1\), \(\zeta_2\), \(\eta\), \(\vartheta\),
that is, where a primitive fifth root of unity \(\zeta_5\) is representable
as norm of a unit in the Galois closure \(N=K\cdot L\) with \(K=\mathbb{Q}(\zeta_5)\),
i.e. \((\exists Z\in U_N)\,\mathrm{Norm}_{N/K}(Z)=\zeta_5\), vanishes.
\end{theorem}

\begin{proof}
A necessary condition for the existence of a unit \(Z\in U_N\)
with \(\mathrm{Norm}_{N/K}(Z)=\zeta_5\)
is that the conductor \(f\) associated with the fifth power free radicand \(D\)
has only prime divisors \(p=5\) or \(p\equiv\pm 1,\pm 7\,(\mathrm{mod}\,25)\)
but no prime divisor \(q\) in other invertible residue classes \((\mathrm{mod}\,25)\).

The group of invertible residue classes modulo \(25\) is \\
\(U(\mathbb{Z}/25\mathbb{Z})=\lbrace 1,2,3,4,6,7,8,9,11,12,13,14,16,17,18,19,21,22,23,24\rbrace\) with \(20\) elements. \\
For an increasing number \(n\) of prime factors of the conductor \(f\),
we have the following probabilities \(P(n)\) for a constitution by prime divisors \(p\equiv\pm 1,\pm 7\,(\mathrm{mod}\,9)\):
\begin{itemize}
\item
for \(n=1\): only \(4\), namely \(1,7,18,24\), among \(20\) residue classes, and thus \(P(1)=\frac{4}{20}=\frac{1}{5}\),
\item
for \(n=2\): only \(4^2\) among \(20^2\) pairs, and thus \(P(2)=\frac{4^2}{20^2}=\frac{1}{5^2}\),
\item
generally for \(n=t\): \(P(t)=\frac{1}{5^t}\).
\end{itemize}
Therefore, the limit probability is given by
\(\lim_{t\to\infty}\,P(t)=\lim_{t\to\infty}\,\frac{1}{5^t}=0\).
However, the asymptotic limit density of arbitrary positive integers
without prime divisors in an entire residue class, or even from several residue classes,
is generally zero, independently of the number \(n\).
\end{proof}


\begin{remark}
\label{rmk:Asymptotic}
In Theorem
\ref{thm:PQ},
we made use of Dirichlet's Theorem
on the density of primes which populate invertible residue classes.
\end{remark}


\section{Computational simplification}
\label{s:SimpleComp}
\noindent
By means of the following lemma, we shall obtain
a reduction of computational complexity
for determining the differential principal factorization type
of certain pure quintic number fields \(L=\mathbb{Q}(\root{5}\of{D})\).

\begin{lemma}
\label{lem:UnitNormInd}
Let \(\eta>1\) be the fundamental unit
of the real qudratic number field \(K^+=\mathbb{Q}(\sqrt{5})\),
that is \(\eta=\frac{1}{2}(1+\sqrt{5})\).

For the unit norm index \(U^+:=(U_{K^+}:N_{M/K^+}(U_M))\in\lbrace 1,5\rbrace\)
the following conditions hold:
\begin{equation}
\label{eqn:UnitNormInd}
\begin{aligned}
U^+=1 &\Longleftrightarrow \eta\in N_{N/K}(U_N), \text{ respectively} \\
U^+=5 &\Longleftrightarrow \eta\not\in N_{N/K}(U_N).
\end{aligned}
\end{equation}
\end{lemma}

\begin{proof}
Since \(-1=(-1)^5=N_{M/K^+}(-1)\) and thus \(-1\in N_{M/K^+}(U_M)\), we have
\begin{center}
\(U^+=1\)
\(\Longleftrightarrow\) \(N_{M/K^+}(U_M)=U_{K^+}=\langle -1,\eta\rangle\)
\(\Longleftrightarrow\) \((\exists\,\upsilon\in U_M)\,N_{M/K^+}(\upsilon)=\eta\).
\end{center}
\textit{Necessity} of the condition:
If \(\eta=N_{M/K^+}(\upsilon)\) for some \(\upsilon\in U_M<U_N\),
then we trivially also have \(\eta=\upsilon^{Tr(\sigma)}=N_{N/K}(\upsilon)\),
where \(\mathrm{Gal}(N/K)=\langle\sigma\rangle\) and \(Tr(\sigma)=1+\sigma+\sigma^2+\sigma^3+\sigma^4\). \\
\textit{Sufficiency} of the condition:
If \(\eta=N_{N/K}(H)\) for some \(H\in U_N\),
then the commutativity of the norm diagram


\begin{table}[ht]
\label{tbl:NormDiagram}
\begin{center}
\begin{tabular}{ccccc}
               &                & \(N_{N/M}\)    &                &             \\
               & \(M\)          & \(\leftarrow\) & \(N\)          &             \\
 \(N_{M/K^+}\) & \(\downarrow\) & \(///\)        & \(\downarrow\) & \(N_{N/K}\) \\
               & \(K^+\)        & \(\leftarrow\) & \(K\)          &             \\
               &                & \(N_{K/K^+}\)  &                &             \\
\end{tabular}
\end{center}
\end{table}


\noindent
implies
\(\eta^2=N_{K/K^+}(\eta)=N_{K/K^+}(N_{N/K}(H))=N_{M/K^+}(N_{N/M}(H))=N_{M/K^+}(\upsilon)\),
where \(\upsilon:=N_{N/M}(H)\in U_M\).
Finally, \(\eta^2\) generates \(\langle\eta\rangle\) modulo \(\eta^5\),
since \((\eta^2)^3=\eta^6\equiv\eta\,(\mathrm{mod}\,\langle\eta^5\rangle)\).
Thus, we have \(U^+=1\).
\end{proof}


\begin{theorem}
\label{thm:NonSplitNonTU}
Let \(D=q_1^{e_1}\cdots q_s^{e_s}\) be a \(5\)th power free radicand
with \(s\ge 1\), distinct primes \(q_1,\ldots,q_s\) and exponents \(1\le e_j\le 4\).
If
\begin{equation}
\label{eqn:NonSplitNonTU}
(\forall\,1\le j\le s)\,q_j\not\equiv\pm 1\,(\mathrm{mod}\,{5}) \text{ and }
(\exists\,1\le j\le s)\,(q_j\ne 5 \text{ and } q_j\not\equiv\pm 7\,(\mathrm{mod}\,{5}))
\end{equation}
then the following criteria hold for \(L=\mathbb{Q}(\root{5}\of{D})\):
\begin{enumerate}
\item
\(L\) is of type \(\varepsilon\) \(\Leftrightarrow\) \((U_K:N_{N/K}(U_N))=5\).
\item
\(L\) is of type \(\gamma\) \(\Leftrightarrow\) \((U_K:N_{N/K}(U_N))=25\).
\end{enumerate}
\end{theorem}

\begin{proof}
Firstly, the condition
\((\forall\,1\le j\le s)\,q_j\not\equiv\pm 1\,(\mathrm{mod}\,5)\)
implies \(I=R=0\)
and thus discourages the types \(\alpha_1,\alpha_2,\alpha_3,\beta_1,\beta_2,\delta_1,\delta_2,\zeta_1,\zeta_2\)
for which either \(I\ge 1\) or \(R\ge 1\).
Moreover, the condition also implies 
\((\forall\,1\le j\le s)\,q_j\not\equiv\pm 1\,(\mathrm{mod}\,25)\).
Secondly, the other condition
\((\exists\,1\le j\le s)\,(q_j\ne 5 \text{ and } q_j\not\equiv\pm 7\,(\mathrm{mod}\,25))\),
together with 
\((\forall\,1\le j\le s)\,q_j\not\equiv\pm 1\,(\mathrm{mod}\,25)\),
excludes the types \(\vartheta,\eta\) and, once more, the types \(\zeta_1,\zeta_2\).
Consequently, only the types \(\gamma\) and \(\varepsilon\) remain as possibilities
and can be distinguished by means of the unit norm index \(U:=(U_K:N_{N/K}(U_N))\).
\end{proof}


\begin{corollary}
\label{cor:NonSplitNonTU}
Under the assumptions of Theorem
\ref{thm:NonSplitNonTU},
the decision between type \(\gamma\) and \(\varepsilon\) can be reduced
from the Galois closure \(N\) of degree \(20\)
to the non-Galois intermediate field \(M\) of degree \(10\).
\begin{enumerate}
\item
\(L\) is of type \(\varepsilon\) \(\Leftrightarrow\) \((U_{K^+}:N_{M/K^+}(U_M))=1\).
\item
\(L\) is of type \(\gamma\) \(\Leftrightarrow\) \((U_{K^+}:N_{M/K^+}(U_M))=5\).
\end{enumerate}
\end{corollary}

\begin{proof}
Since the unique type \(\vartheta\) with \(U=1\) is impossible,
the Lemma
\ref{lem:UnitNormInd}
yields the equivalences
\(U=5 \Longleftrightarrow U^+=1\) and \(U=25 \Longleftrightarrow U^+=5\).
\end{proof}


\section{Computational and theoretical results}
\label{s:CompRslt}
\noindent
At the end of his \(1975\) article on class numbers of pure quintic fields,
Parry suggested verbatim
\lq\lq In conclusion the author would like to say that he believes
a numerical study of pure quintic fields would be most interesting\rq\rq\
\cite[p. 484]{Pa}.
Of course, it would have been rather difficult
to realize Parry's desire in \(1975\).
But now, \(40\) years later, we are in the position to use
the powerful computer algebra systems PARI/GP
\cite{PARI}
and MAGMA
\cite{BCP,BCFS,MAGMA}
for starting an attack against this hard problem.
This will actually be done in the present paper.

Even in \(1991\), when we generalized
Barrucand and Cohn's theory
\cite{BaCo2}
of principal factorization types
from pure cubic fields \(\mathbb{Q}(\root{3}\of{D})\)
to pure quintic fields \(\mathbb{Q}(\root{5}\of{D})\)
\cite{Ma0},
it was still impossible to verify our hypothesis
about the distinction between
absolute, intermediate and relative \textit{differential principal factors} (\S\
\ref{ss:DiffPrFact})
and about the values of the \textit{unit norm index} \((U_K:N_{N/K}(U_N))\) (\S\
\ref{ss:HerbQuot})
by actual computations.


All these conjectures have been proven by our most recent numerical investigations.
Our classification is based on the Hasse theorem
about the Herbrand quotient of the unit group \(U_N\) of the Galois closure \(N\)
as a module over the relative group \(G=\mathrm{Gal}(N/K)\) with respect to the cyclotomic subfield \(K\).
It only involves the unit norm index \((U_K:N_{N/K}(U_N))\) and
our \(13\) types of differential principal factors,
but not the index of subfield units \((U_N:U_0)\)
in Parry's class number formula
\eqref{eqn:ParryClNoRel}.


\subsection{Refined Dedekind species and prototypes}
\label{ss:Prototypes}

As before, let \(N=\mathbb{Q}(\root{5}\of{D},\zeta)\) be a pure metacyclic field of degree \(20\)
with normalized radicand \(D=5^{e_\ast}\cdot q_1^{e_1}\ldots q_t^{e_t}\) where \(0\le e_\ast\le 4\).
The fourth power of the conductor has the shape
\(f^4=5^{e_0}\cdot q_1^4\cdots q_t^4\) with \(e_0\in\lbrace 0,2,6\rbrace\).

For the primes dividing the conductor,
the splitting behavior is described by
the number \(s_2=\#\lbrace 1\le j\le t\mid q_j\equiv -1\,(\mathrm{mod}\,5)\rbrace\) of \(2\)-split primes in \(N\)
and the number \(s_4=\#\lbrace 1\le j\le t\mid q_j\equiv +1\,(\mathrm{mod}\,5)\rbrace\) of \(4\)-split primes in \(N\),
and the role in formula
\cite[Thm. 2, p. 104]{Ma1}
for the multiplicity \(m\) is described by
the number \(u=\#\lbrace 1\le j\le t\mid q_j\equiv\pm 1,\pm 7\,(\mathrm{mod}\,25)\rbrace\) of free primes
and the number \(v=t-u\) of restrictive primes.
\(n:=t-s_2-s_4\) is the number of non-split primes.


\begin{definition}
\label{dfn:RefinedDedekind}
The multiplet \((e_0;t,u,v,m;n,s_2,s_4)\) is called the \textit{refined Dedekind species} of the field \(N\).
The first entry characterizes the (coarse) Dedekind species:
\begin{equation}
\label{eqn:Dedekind}
e_0=
\begin{cases}
6 & \text{ if } N \text{ is of species 1a}, \\
2 & \text{ if } N \text{ is of species 1b}, \\
0 & \text{ if } N \text{ is of species 2}.
\end{cases}
\end{equation}
\end{definition}


\begin{definition}
\label{dfn:Prototypes}
A normalized radicand \(D\) is called a \textit{prototype}
if it is minimal among all normalized radicands sharing a common
\begin{itemize}
\item
refined Dedekind species \((e_0;t,u,v,m;n,s_2,s_4)\),
\item
differential principal factorization type \((U,\eta,\zeta;A,I,R)\),
\item
set of abelian type invariants of the \(5\)-class groups \(\mathrm{Cl}_5(F)\) with \(F\in\lbrace L,M,N\rbrace\)
and logarithmic unit index \(E=v_5((U_N:U_0))\).
\end{itemize}
\end{definition}

\noindent
Note that, since \(t\) is unbounded,
there are infinitely many prototypes \(D\) of pure metacyclic fields \(N\) of degree \(20\).


\subsection{Prime radicands}
\label{ss:PrmRad}

Trivially, any prime radicand is fifth power free and normalized, a priori.

\begin{theorem}
\label{thm:PrmRad}
Let \(L=\mathbb{Q}(\root{5}\of{D})\) be a pure quintic field with prime radicand \(D=q\in\mathbb{P}\).
\begin{enumerate}
\item
If \(q\equiv\pm 2\,(\mathrm{mod}\,5)\) but \(q\not\equiv\pm 7\,(\mathrm{mod}\,25)\),
then \(N\) is a Polya field of type \(\varepsilon\).
\item
If \(q=5\) or \(q\equiv\pm 7\,(\mathrm{mod}\,25)\),
then \(N\) is a Polya field of type \(\vartheta\).
\end{enumerate}
\end{theorem}

\begin{proof}
In both cases,
since \(q\not\equiv\pm 1\,(\mathrm{mod}\,5)\),
it follows that \(q\) does not split in \(K^+,M,K,N\).
Consequently, intermediate and relative principal factors are discouraged, \(I=R=0\),
and only absolute principal factors are possible, \(A\ge 1\).
This eliminates the possibility of the types
\(\alpha_1\), \(\alpha_2\), \(\alpha_3\), \(\beta_1\), \(\beta_2\),
\(\delta_1\), \(\delta_2\), \(\zeta_1\), \(\zeta_2\),
and there only remain the types \(\gamma\), \(\varepsilon\), \(\eta\) and \(\vartheta\).
\begin{enumerate}
\item
Since \(q\not\equiv\pm 7\,(\mathrm{mod}\,25)\), the field \(L\) is of the species 1b with \(f^4=5^2q^4\).
Since the additional prime \(5\) is ramified in \(L\), we have \(A\le 2\),
which discourages type \(\gamma\) but enables types \(\varepsilon\) and \(\eta\).
However, \(q\not\equiv\pm 7\,(\mathrm{mod}\,25)\) together with \(q\ne 5\)
denies the existence of a unit \(Z\in U_N\) with \(N_{N/K}(Z)=\zeta\),
whence types \(\eta\) and \(\vartheta\) are impossible.
So the type must be \(\varepsilon\) with invariants \((A,I,R)=(2,0,0)\)
and there must exist a unit \(H\in U_N\) with \(N_{N/K}(H)=\eta\).
\item
For \(q=5\), the field \(L\) is of the species 1a with \(f^4=5^6\).
For \(q\equiv\pm 7\,(\mathrm{mod}\,25)\), the field \(L\) is of the species 2 with \(f^4=q^4\).
Since the conductor \(f\) is divisible by a single prime only, we have \(A\le 1\),
there are only \(5^A=5\) possible absolute principal factors,
and they are occupied by the trivial radicals \(\delta^e\) with \(0\le e\le 4\),
where \(\delta:=\root{5}\of{q}\).
Consequently, the invariants \((A,I,R)\) are given by \((1,0,0)\)
which uniquely characterizes type \(\vartheta\) with \((U_K:N_{N/K}(U_N))=1\),
and enforces the existence of units \(Z,H\in U_N\) such that
\(N_{N/K}(Z)=\zeta\) and \(N_{N/K}(H)=\eta\). \qedhere
\end{enumerate}
\end{proof}


\subsection{Splitting prime radicands}
\label{ss:SplPrimRad}
\noindent
Let \(\ell\in\mathbb{P}\) be a prime number which splits in \(M\).
Suppose the pure quintic field \(L=\mathbb{Q}(\sqrt[5]{D})\) is generated by
the \textit{prime radicand} \(D=\ell\).
Then there arise two situations for the conductor \(f\) of \(N/K\),
either \(f^4=\ell^4\) when \(\ell\equiv\pm 1\,(\mathrm{mod}\,25)\)
or \(f^4=5^2\cdot\ell^4\) when \(\ell\equiv\pm 1\,(\mathrm{mod}\,5)\) but \(\ell\not\equiv\pm 1\,(\mathrm{mod}\,25)\).


\subsection{Prime conductors}
\label{ss:SplPrimCnd}

\begin{enumerate}
\item
\(\ell\equiv -1\,(\mathrm{mod}\,25)\):

\begin{example}
\label{exm:SplPrimCndMinus}
There are \(6\) occurrences up to the upper bound \(D<10^3\):
\[D\in\lbrace 149,199,349,449,499,599\rbrace.\]
Here, \(N\) is a Polya field and exclusively of type \(\delta_2\)
but never of the types \(\zeta_2\) or \(\vartheta\).
\end{example}

\item
\(\ell\equiv +1\,(\mathrm{mod}\,25)\):

\begin{example}
\label{exm:SplPrimCndPlus}
There are \(6\) occurrences up to the upper bound \(D<10^3\):
\[D\in\lbrace 101,151,251,401,601,701\rbrace.\]
The corresponding pure metacyclic field \(N\) has the Polya property and is of type \\
either \(\alpha_1\), for \(D\in\lbrace 401,701\rbrace\), even with \(\mathrm{Cl}_5(L)\simeq C_5\times C_5\), \\
or \(\alpha_2\), for \(D\in\lbrace 151,251,601\rbrace\), \\
or \(\zeta_1\), for \(D=101\),
but never of the types \(\delta_1\),\(\delta_2\), \(\zeta_2\), \(\vartheta\).
\end{example}

\end{enumerate}


\subsection{Composite conductors}
\label{ss:CmpCnd}

\begin{enumerate}
\item
\(\ell\equiv -1\,(\mathrm{mod}\,5)\) but not \(\ell\equiv -1\,(\mathrm{mod}\,25)\):

\begin{example}
\label{exm:CmpCnd}
There are \(32\) occurrences up to the upper bound \(D<10^3\)
\[D\in\lbrace 19,29,59,79,89,109,139,179,229,239,269,359,379,389,409,419,439,\]
\[479,509,569,619,659,709,719,739,769,809,829,839,859,919,929\rbrace.\]
The field \(N\) is of type \\
either \(\delta_2\), for \(D\in\lbrace 19,29,59,79,89,109,179,229,239,269,389,409,439,479,509,569,619,\) \\
\(659,709,719,739,769,809,839,859,919,929\rbrace\), without Polya property, \\
or \(\beta_2\), for \(D\in\lbrace 139,359,419,829\rbrace\), \textbf{with} Polya property, \\
or \(\varepsilon\), for \(D=379\), \textbf{with} Polya property.
\end{example}

\item
\(\ell\equiv +1\,(\mathrm{mod}\,5)\) but not \(\ell\equiv +1\,(\mathrm{mod}\,25)\):

\begin{example}
\label{exm:CmpPrimCndPlus}
There are \(33\) occurrences up to the upper bound \(D<10^3\):
\[D\in\lbrace 11,31,41,61,71,131,181,191,211,241,271,281,311,331,421,431,461,\]
\[491,521,541,571,631,641,661,691,761,811,821,881,911,941,971,991\rbrace.\]
The pure metacyclic field \(N\) is of type \\
either \(\alpha_1\), for \(D\in\lbrace 31,281,761\rbrace\), without Polya property, \\
or \(\alpha_2\), for \(D\in\lbrace 11,41,61,71,131,181,241,311,331,431,491,541,571,631,661,691,811,\) \\
\(821,911,941,971\rbrace\), without Polya property, \\
or \(\beta_1\), for \(D\in\lbrace 191,271,641\rbrace\), \textbf{with} Polya property, \\
or \(\delta_1\), for \(D\in\lbrace 211,421,461,521,881,991\rbrace\), without Polya property.
\end{example}

\end{enumerate}


\subsection{Fields of the rare type \(\alpha_3\) with \(2\)-dimensional intermediate DPF}
\label{ss:SpecialCaseAlpha3}


\begin{theorem}
\label{thm:IntTwo}
For all normalized radicands of the shape \(D=q^{e_0}\cdot\ell_1^{e_1}\cdot\ell_2^{e_2}\) with
\(0\le e_0\le 4\), \(q\equiv\pm 2\,(\mathrm{mod}\,5)\), \(1\le e_1,e_2\le 4\), \(\ell_1,\ell_2\equiv\pm 1\,(\mathrm{mod}\,5)\),
where \(L=\mathbb{Q}(\sqrt[5]{D})\) is of DPF type \(\alpha_3\) with maximal invariant \(I=2\),
the \(2\)-dimensional \(\mathbb{F}_5\)-vectorspace
\((\mathcal{P}_{M/K^+}/\mathcal{P}_{K^+})\bigcap\ker(N_{M/L})\)
of intermediate ambiguous principal ideals is generated by the kernel ideals
\(\mathcal{K}^{(\ell_1)}\) and \(\mathcal{K}^{(\ell_2)}\).
\end{theorem}

\begin{proof}
In fact, we have the extreme situation that \(I=s_2+s_4=2\) and consequently
\[(\mathcal{P}_{M/K^+}/\mathcal{P}_{K^+})\bigcap\ker(N_{M/L})=(\mathcal{I}_{M/K^+}/\mathcal{P}_{K^+})\bigcap\ker(N_{M/L})
=\bigoplus_{i=1}^{2}\,\mathbb{F}_5\,\mathcal{K}^{(\ell_i)},\]
that is, all intermediate ambiguous ideals are principal.
\end{proof}

\begin{example}
\label{exm:IntTwo}
In the range \(2\le D<1000\) of normalized radicands,
there are \(8\) cases which are covered by Theorem
\ref{thm:IntTwo}.
For all of them,
the logarithmic index of subfield units takes the value \(e=2\).
\textbf{None} of the corresponding pure metacyclic fields \(N\) can be a Polya field.
\begin{enumerate}
\item
\(D=319=11\cdot 29\) with \(e_0=0\) of species 1b,
\item
\(D=551=19\cdot 29\) with \(e_0=0\) of species 2,
\item
\(D=589=19\cdot 31\) with \(e_0=0\) of species 1b,
\item
\(D=627=3\cdot 11\cdot 19\) with \(e_0=1\) of species 1b,
\item
\(D=649=11\cdot 59\) with \(e_0=0\) of species 2,
\item
\(D=869=11\cdot 79\) with \(e_0=0\) of species 1b,
\item
\(D=899=29\cdot 31\) with \(e_0=0\) of species 2,
\item
\(D=957=3\cdot 11\cdot 29\) with \(e_0=1\) of species 2.
\end{enumerate}
Except for \(D=627\) where \(V_L=3\), \(V_M=4\), \(V_N=9\),
the logarithmic \(5\)-class numbers are always given by \(V_L=2\), \(V_M=2\), \(V_N=5\).
\end{example}


\subsection{Fields of the types \(\gamma\), \(\varepsilon\), \(\eta\), \(\vartheta\) without splitting prime divisors of \(f\)}
\label{ss:GammaEpsilonTheta}



\begin{theorem}
\label{thm:KobayashiParry}
(Pure quintic fields with trivial \(5\)-class group) \\
Let \(L=\mathbb{Q}(\sqrt[5]{D})\) be a pure quintic field,
\(N=\mathbb{Q}(\zeta_5,\sqrt[5]{D})\) be the pure metacyclic normal closure of \(L\), and
\(M=\mathbb{Q}(\sqrt{5},\sqrt[5]{D})\) be the maximal real subfield of \(N\).
For the class numbers \(h_L\), \(h_M\) and \(h_N\) of \(L\), \(M\) and \(N\),
the following implications describe divisibility by \(5\):
\begin{equation}
\label{eqn:KobayashiParry}
5\mid h_L \Longleftrightarrow 5\mid h_M, \quad
5\mid h_L \Longrightarrow 5\mid h_N \quad \text{ but only } \quad
5^2\mid h_N \Longrightarrow 5\mid h_L.
\end{equation}
\end{theorem}

\begin{proof}
In terms of \(5\)-valuations, we prove the contrapositive equivalence
\(v_5(h_L)=0\Longleftrightarrow v_5(h_M)=0\)
using the abbreviations \(V_L:=v_5(h_L)\) and \(V_M:=v_5(h_M)\).
Let \(\mathcal{D}\) be the value of the determinant used by Kobayashi
to describe the relation between the unit groups \(U_L\) of \(L\) and \(U_M\) of \(M\).
The Kobayashi class number formula
\(h_M=\frac{\mathcal{D}}{5^2}h_L^2\)
can be expressed additively by \(5\)-valuations
\(V_M=2\cdot V_L+Q^+-2\), where \(\mathcal{D}=5^{Q^+}\) with \(Q^+\in\lbrace 0,1,2\rbrace\)
\cite[Prop.2, p.467]{Ky1},
\cite[Prop.3, p.22]{Ky2}. \\
Firstly, \(V_L=0\) \(\Longrightarrow\)
\(V_M=Q^+-2\) and, since \(V_M\) cannot be negative, \(Q^+=2\) and \(V_M=0\). \\
Secondly, \(V_M=0\) \(\Longrightarrow\) \(2\cdot V_L=2-Q^+\le 2\) must be even,
that is, either \(Q^+=2\) and \(2\cdot V_L=0\), whence \(V_L=0\),
or \(Q^+=0\) and \(2\cdot V_L=2\), whence \(V_L=1\).
However, the last situation cannot occur,
since a class of order \(5\) cannot capitulate in \(M/L\) of degree \(\lbrack M:L\rbrack=2\).

Furthermore, we prove the contrapositive implications
\(v_5(h_N)=0\Longrightarrow v_5(h_L)=0\) and \(v_5(h_L)=0\Longrightarrow v_5(h_N)\in\lbrace 0,1\rbrace\)
using the abbreviations \(V_L:=v_5(h_L)\) and \(V_N:=v_5(h_N)\).
Let \(U_0\) be the subgroup generated by the units of \(K\) and all conjugate fields of \(L\)
in the unit group \(U_N\) of \(N\).
The Parry class number formula
\(h_N=\frac{(U_N:U_0)}{5^5}h_L^4\)
can be expressed additively by \(5\)-valuations
\(V_N=4\cdot V_L+Q-5\), where \((U_N:U_0)=5^Q\) with \(Q\in\lbrace 0,\ldots,6\rbrace\)
\cite[Thm.1, p.476, Thm.2, p.478]{Pa}. \\
Firstly, \(V_L=0\) \(\Longrightarrow\)
\(V_N=Q-5\) and, since \(V_N\) cannot be negative,
either \(Q=5\) and \(V_N=0\) or \(Q=6\) and \(V_N=1\). \\
Secondly, \(V_N=0\) \(\Longrightarrow\) \(4\cdot V_L=5-Q\le 5\) must be a multiple of \(4\),
that is, either \(Q=5\) and \(4\cdot V_L=0\), whence \(V_L=0\),
or \(Q=1\) and \(4\cdot V_L=4\), whence \(V_L=1\).
However, the last situation cannot occur,
since a class of order \(5\) cannot capitulate in \(N/L\) of degree \(\lbrack N:L\rbrack=4\).
\end{proof}

\begin{corollary}
\label{cor:KobayashiParry}
The relation \(5\nmid h_N\) enforces
an index of subfield units \((U_N:U_0)=5^5\) and thus a DPF type either \(\varepsilon\) or \(\vartheta\).
The weaker relation \(5\nmid h_L\) admits two values of
the index of subfield units \((U_N:U_0)\in\lbrace 5^5,5^6\rbrace\)
and consequently one of the DPF types \(\gamma\), \(\varepsilon\), \(\eta\) or \(\vartheta\).
Both relations imply the value \(\mathcal{D}=5^2\) of the Kobayashi determinant.
\end{corollary}


\noindent
The cubic analogue of Theorem
\ref{thm:KobayashiParry}
is stated in the following well-known result
\cite{Ho}.

\begin{theorem}
\label{thm:Scholz}
(Pure cubic fields with trivial \(3\)-class group) \\
Let \(L=\mathbb{Q}(\sqrt[3]{D})\) be a pure cubic field
and \(N=\mathbb{Q}(\zeta_3,\sqrt[3]{D})\) be its pure metacyclic normal closure.
For the class numbers \(h_L\) and \(h_N\) of \(L\) and \(N\),
the following equivalence describes divisibility by \(3\):
\begin{equation}
\label{eqn:Scholz}
3\mid h_L \quad \Longleftrightarrow \quad 3\mid h_N.
\end{equation}
\end{theorem}

\begin{proof}
In terms of \(3\)-valuations, we prove the contrapositive equivalence
\(v_3(h_L)=0\Longleftrightarrow v_3(h_N)=0\)
using the abbreviations \(V_L:=v_3(h_L)\) and \(V_N:=v_3(h_N)\).
Let \(U_0\) be the subgroup generated by the units of all proper subfields of \(N\)
in the unit group \(U_N\) of \(N\).
The Scholz class number formula
\(h_N=\frac{(U_N:U_0)}{3}h_L^2\)
can be expressed additively by \(3\)-valuations
\(V_N=2\cdot V_L+Q-1\), where \((U_N:U_0)=3^Q\) with \(Q\in\lbrace 0,1\rbrace\)
\cite[Lem. 1, p. 7]{Ho}. \\
Firstly, \(V_L=0\) \(\Longrightarrow\)
\(V_N=Q-1\) and, since \(V_N\) cannot be negative, \(Q=1\) and \(V_N=0\). \\
Secondly, \(V_N=0\) \(\Longrightarrow\) \(2\cdot V_L=1-Q\le 1\) must be even,
that is, \(Q=1\) and \(2\cdot V_L=0\), whence \(V_L=0\).
\end{proof}

\begin{corollary}
\label{cor:Scholz}
The relation \(3\nmid h_L\) enforces
an index of subfield units \((U_N:U_0)=3\) and thus a DPF type either \(\beta\) or \(\gamma\).
\end{corollary}


\subsection{General criterion for Polya fields closely related to DPF types}
\label{ss:PolyaAndDPF}

\noindent
We begin with a general criterion which is valid for the pure metacyclic normal closures \(N=L(\zeta)\)
of both, pure cubic fields \(L=\mathbb{Q}(\sqrt[3]{D})\) and pure quintic fields \(L=\mathbb{Q}(\sqrt[5]{D})\).
Let \(G=\mathrm{Gal}(N/\mathbb{Q})\) be the absolute Galois group of \(N\).
Furthermore, denote by \(T=\dim_{\mathbb{F}_5}(\mathcal{I}_{L/\mathbb{Q}}/\mathcal{P}_{\mathbb{Q}})\) and by
\(A=\dim_{\mathbb{F}_5}(\mathcal{P}_{L/\mathbb{Q}}/\mathcal{P}_{\mathbb{Q}})\).

\begin{theorem}
\label{thm:MainPolya}
All the following statements are equivalent:
\begin{enumerate}
\item
\(N\) is a Polya field.
\item
The subgroup \((\mathcal{I}_N^G\cdot\mathcal{P}_N)/\mathcal{P}_N\le\mathrm{Cl}(N)\)
of strongly ambiguous classes of \(N/\mathbb{Q}\) is trivial.
\item
\(\mathcal{I}_{N/\mathbb{Q}}/\mathcal{P}_{\mathbb{Q}}=\mathcal{P}_{N/\mathbb{Q}}/\mathcal{P}_{\mathbb{Q}}\)
(where \(\mathcal{I}_{N/\mathbb{Q}}/\mathcal{P}_{\mathbb{Q}}\simeq\mathcal{I}_{L/\mathbb{Q}}/\mathcal{P}_{\mathbb{Q}}\)
and \(\mathcal{P}_{N/\mathbb{Q}}/\mathcal{P}_{\mathbb{Q}}\simeq\mathcal{P}_{L/\mathbb{Q}}/\mathcal{P}_{\mathbb{Q}}\)).
\item
\(\mathcal{I}_{L/\mathbb{Q}}/\mathcal{P}_{\mathbb{Q}}=\mathcal{P}_{L/\mathbb{Q}}/\mathcal{P}_{\mathbb{Q}}\),
reduced from the metacyclic normal field to the pure field.
\item
\(A=T\), in terms of dimensions over \(\mathbb{F}_5\) 
(which necessarily yields an upper bound for the number of primes ramified in \(L/\mathbb{Q}\),
\(T\le 2\) in the cubic case and \(T\le 3\) in the quintic case).
\item
\((\forall\,p\in\mathbb{P},\ p\mid f)\,(\exists\,\alpha\in L)\,N_{L/\mathbb{Q}}(\alpha)=p\).
\end{enumerate}
\end{theorem}

\begin{proof}
According to Zantema
\cite[Thm. 1.3, p. 157]{Za},
the normal field \(N\) with absolute Galois group \(G=\mathrm{Gal}(N/\mathbb{Q})\) is a Polya field
if and only if the short exact sequence
\begin{equation}
\label{eqn:Zantema}
1\to H^1(G,U_N)\to\bigoplus_{p\in\mathbb{P}}\,(\mathbb{Z}/e(p)\mathbb{Z})\to\mathrm{Po}(N)\to 1
\end{equation}
collapses to an isomorphism
\(H^1(G,U_N)\simeq\bigoplus_{p\in\mathbb{P}}\,(\mathbb{Z}/e(p)\mathbb{Z})\),
that is, if and only if the \textit{Polya group} \(\mathrm{Po}(N)\),
which is generated by the classes of all Ostrowski ideals of \(N\) (see formula
\eqref{eqn:Ostrowski}),
is trivial.
This abstract cohomological statement can be interpreted in the language of algebraic number theory.
According to Iwasawa's isomorphism
\(H^1(G,U_N)\simeq\mathcal{P}_N^G/\mathcal{P}_{\mathbb{Q}}=\mathcal{P}_{N/\mathbb{Q}}/\mathcal{P}_{\mathbb{Q}}\)
in formula 
\eqref{eqn:AIR},
and the isomorphism
\(\mathcal{I}_N^G/\mathcal{I}_{\mathbb{Q}}=\mathcal{I}_{N/\mathbb{Q}}/\mathcal{I}_{\mathbb{Q}}
\simeq\bigoplus_{p\in\mathbb{P}}\,(\mathbb{Z}/e(p)\mathbb{Z})\)
in Corollary
\ref{cor:Finiteness}
with \(F=\mathbb{Q}\) and \(E=N\),
the Polya group
\begin{equation}
\label{eqn:StrAmbCls}
\mathrm{Po}(N)\simeq(\mathcal{I}_N^G/\mathcal{P}_{\mathbb{Q}})/(\mathcal{P}_N^G/\mathcal{P}_{\mathbb{Q}})
\simeq\mathcal{I}_N^G/\mathcal{P}_N^G=\mathcal{I}_N^G/(\mathcal{P}_N\cap\mathcal{I}_N^G)\simeq(\mathcal{I}_N^G\cdot\mathcal{P}_N)/\mathcal{P}_N
\end{equation}
is precisely the group of strongly ambiguous classes of \(N/\mathbb{Q}\).
Thus, (1) \(\Longleftrightarrow\) (2) \(\Longleftrightarrow\) (3).

It remains to show that the group of primitive absolutely invariant ideals
\(\mathcal{I}_{N/\mathbb{Q}}/\mathcal{P}_{\mathbb{Q}}=\mathcal{I}_N^G/\mathcal{I}_{\mathbb{Q}}\),
which has not appeared in the literature up to now,
is isomorphic to the well-known group
\(\mathcal{I}_{L/\mathbb{Q}}/\mathcal{P}_{\mathbb{Q}}\)
in Theorem \ref{thm:AbsDim}.
(The group of primitive relatively invariant ideals
\(\mathcal{I}_{N/K}/\mathcal{P}_K=\mathcal{I}_N^H/\mathcal{I}_K\)
with the cyclic subgroup \(H=\mathrm{Gal}(N/K)\) of the metacyclic group \(G\)
is also very well known.)

Using the notation of section \S\
\ref{s:Dimensions},
and the number \(n:=T-s_2-s_4\) (non-split primes), we have:
\begin{equation}
\label{eqn:RelativelyInvariant}
\mathcal{I}_{N/K}/\mathcal{P}_K\simeq
\left(\bigoplus_{i=1}^{n}\,\mathbb{F}_5\,\mathfrak{Q}_i\right)
\oplus\left(\bigoplus_{i=1}^{s_2}\,(\mathbb{F}_5\,\mathfrak{L}_i\oplus\mathbb{F}_5\,\mathfrak{L}_i^{\tau})\right)
\oplus\left(\bigoplus_{i=s_2+1}^{s_2+s_4}\,(\mathbb{F}_5\,\mathfrak{L}_i\oplus\cdots\oplus\mathbb{F}_5\,\mathfrak{L}_i^{\tau^3})\right)\,.
\end{equation}
Now we have to select the \(\tau\)-invariant components of this direct sum, according to section
\ref{ss:Invariance}:
\begin{equation}
\label{eqn:AbsolutelyInvariant}
\mathcal{I}_{N/\mathbb{Q}}/\mathcal{P}_{\mathbb{Q}}\simeq
\left(\bigoplus_{i=1}^{n}\,\mathbb{F}_5\,\mathfrak{Q}_i\right)
\oplus\left(\bigoplus_{i=1}^{s_2}\,\mathbb{F}_5\,\mathfrak{L}_i^{1+\tau}\right)
\oplus\left(\bigoplus_{i=s_2+1}^{s_2+s_4}\,\mathbb{F}_5\,\mathfrak{L}_i^{1+\tau+\tau^2+\tau^3}\right)
\simeq\mathcal{I}_{L/\mathbb{Q}}/\mathcal{P}_{\mathbb{Q}}\,,
\end{equation}
since \((\mathfrak{L}_i^{1-\tau})^{\tau}=(\mathfrak{L}_i^{1-\tau})^{-1}\) is not invariant.
Therefore, (3) \(\Longleftrightarrow\) (4) \(\Longleftrightarrow\) (5)\(\Longleftrightarrow\) (6).
\end{proof}


\noindent
Finally we apply Theorem
\ref{thm:MainPolya}
to a few special situations.

\begin{theorem}
\label{thm:Epsilon}
(The ground state of DPF type \(\varepsilon\) with trivial Polya property) \\
Let \(q_1,q_2\) be prime numbers \(q_i\equiv\pm 2\,(\mathrm{mod}\,5)\) but \(q_i\not\equiv\pm 7\,(\mathrm{mod}\,25)\).
A pure metacyclic field \(N=\mathbb{Q}(\zeta,\sqrt[5]{D})\) with conductor \(f\) such that
\begin{enumerate}
\item
\(f^4=5^2q_1^4\) of species \(1\mathrm{b}\)
forms a singulet, \(m=1\), with DPF type \(\varepsilon\).
\item
\(f^4=5^6q_1^4\) of species \(1\mathrm{a}\)
belongs to a quartet, \(m=4\), with homogeneous DPF type \((\varepsilon,\varepsilon,\varepsilon,\varepsilon)\).
\item
\(f^4=q_1^4q_2^4\) of species \(2\)
forms a singulet, \(m=1\), with DPF type \(\varepsilon\).
\end{enumerate}
\end{theorem}

\begin{proof}
We start by proving the multiplicities \(m=m(f)\) of the conductors \(f\) with the aid of
\cite[Thm. 2, p. 104]{Ma1},
where \(t_0:=u+v\) denotes the number of primes different from \(5\) dividing \(f\).
\begin{enumerate}
\item
For species \(1\mathrm{b}\) with \(D=q_1^{e_1}\), \(u=0\), \(v=1\), and \(D^4\not\equiv 1\,(\mathrm{mod}\,25)\),
we have \(m(f)=(5-1)^0\cdot X_1=1\cdot\frac{1}{5}(4^1-(-1)^1))=\frac{5}{5}=1\).
\item
For species \(1\mathrm{a}\) with \(D=5^{e_0}\cdot q_1^{e_1}\), \(t_0=1\) and \(5\mid D\), we have \(m(f)=(5-1)^{t_0}=4^1=4\).
\item
For species \(2\) with \(D=q_1^{e_1}\cdot q_2^{e_2}\), \(u=0\), \(v=2\), and \(D^4\equiv 1\,(\mathrm{mod}\,25)\),
we have \(m(f)=(5-1)^0\cdot X_{2-1}=1\cdot\frac{1}{5}(4^1-(-1)^1))=\frac{5}{5}=1\).
\end{enumerate}
In each case, we have \(T=2\) non-split prime divisors of \(f\), \(s_2=s_4=0\), and thus
\(1\le A\le T=2\), by
\eqref{eqn:AbsBnd},
\(0\le I\le s_2+s_4=0\), by
\eqref{eqn:IntBnd},
\(0\le R\le 2s_4=0\), by
\eqref{eqn:RelBnd}.
The assumption
\(q_i\equiv\pm 2\,(\mathrm{mod}\,5)\) but \(q_i\not\equiv\pm 7\,(\mathrm{mod}\,25)\)
excludes the DPF types \(\eta\) with \(U=1\), \((A,I,R)=(2,0,0)\)
and \(\vartheta\) with \(U=0\), \((A,I,R)=(1,0,0)\),
and only the possibility of DPF type \(\varepsilon\) with \(U=1\), \((A,I,R)=(2,0,0)\), \(A=T\) remains.
Recall that generally \(U+1=A+I+R\), according to
\eqref{eqn:HerbrandQuot}
and
\eqref{eqn:QuinticDecomp}.
\end{proof}

\begin{corollary}
\label{cor:Epsilon}
In each case, the \(5\)-class numbers of \(L\), \(M\), \(N\) are given by
\(h_L=h_M=h_N=5^w\) with \(w=0\) (ground state of DPF type \(\varepsilon\)), and the index of subfield units by \(5^E\) with \(E=5\).
\end{corollary}

\begin{proof}
This claim was proved by Parry in
\cite[Thm. IV and Formula (10), p. 481]{Pa}.
\end{proof}


\begin{theorem}
\label{thm:GammaEpsilon}
(Ground state of type \(\gamma\) with non-trivial Polya property, excited state of type \(\varepsilon\)) \\
Let \(q_1,q_2,q_3\) be prime numbers \(q_i\equiv\pm 2\,(\mathrm{mod}\,5)\) but \(q_i\not\equiv\pm 7\,(\mathrm{mod}\,25)\).
A pure metacyclic field \(N=\mathbb{Q}(\zeta,\sqrt[5]{D})\) with conductor \(f\) such that
\begin{enumerate}
\item
\(f^4=5^2q_1^4q_2^4\) of species \(1\mathrm{b}\)
forms a triplet, \(m=3\), with DPF type \((\gamma^c,\varepsilon^b)\), \(c+b=3\).
\item
\(f^4=5^6q_1^4q_2^4\) of species \(1\mathrm{a}\)
belongs to a hexadecuplet, \(m=16\), with DPF type \((\gamma^c,\varepsilon^b)\), \(c+b=16\).
\item
\(f^4=q_1^4q_2^4q_3^4\) of species \(2\)
forms a triplet, \(m=3\), with DPF type \((\gamma^c,\varepsilon^b)\), \(c+b=3\).
\end{enumerate}
The formal integer exponents \(0\le c,b\le m\) indicate repetition (\(c\) times \(\gamma\), \(b\) times \(\varepsilon\)).
\end{theorem}

\begin{conjecture}
\label{cnj:GammaEpsilon}
In each case, the \(5\)-class numbers of \(L\), \(M\), \(N\) are given by either
\(h_L=h_M=5^0\), \(h_N=5^1\) (ground state of DPF type \(\gamma\)), and the index of subfield units by \(5^E\) with \(E=6\)
or \(h_L=5^1\), \(h_M=5^2\), \(h_N=5^4\) (excited state of DPF type \(\varepsilon\)), and the index of subfield units by \(5^E\) with \(E=5\).
\end{conjecture}

\begin{proof}
We start by proving the multiplicities \(m=m(f)\) of the conductors \(f\) with the aid of
\cite[Thm. 2, p. 104]{Ma1},
where \(t_0:=u+v\) denotes the number of primes different from \(5\) dividing \(f\).
\begin{enumerate}
\item
For species \(1\mathrm{b}\) with \(D=q_1^{e_1}\cdot q_2^{e_2}\), \(u=0\), \(v=2\), and \(D^4\not\equiv 1\,(\mathrm{mod}\,25)\),
we have \(m(f)=(5-1)^0\cdot X_2=1\cdot\frac{1}{5}(4^2-(-1)^2))=\frac{15}{5}=3\).
\item
For species \(1\mathrm{a}\) with \(D=5^{e_0}\cdot q_1^{e_1}\cdot q_2^{e_2}\), \(t_0=2\) and \(5\mid D\), we have \(m(f)=(5-1)^{t_0}=4^2=16\).
\item
For species \(2\) with \(D=q_1^{e_1}\cdot q_2^{e_2}\cdot q_3^{e_3}\), \(u=0\), \(v=3\), and \(D^4\equiv 1\,(\mathrm{mod}\,25)\),
we have \(m(f)=(5-1)^0\cdot X_{3-1}=1\cdot\frac{1}{5}(4^2-(-1)^2))=\frac{15}{5}=3\).
\end{enumerate}
In each case, we have \(T=3\) non-split prime divisors of \(f\), \(s_2=s_4=0\), and thus
\(1\le A\le T=3\), by
\eqref{eqn:AbsBnd},
\(0\le I\le s_2+s_4=0\), by
\eqref{eqn:IntBnd},
\(0\le R\le 2s_4=0\), by
\eqref{eqn:RelBnd}.
The assumption
\(q_i\equiv\pm 2\,(\mathrm{mod}\,5)\) but \(q_i\not\equiv\pm 7\,(\mathrm{mod}\,25)\)
excludes the DPF types \(\eta\) with \(U=1\), \((A,I,R)=(2,0,0)\)
and \(\vartheta\) with \(U=0\), \((A,I,R)=(1,0,0)\),
and only the possibilities of either DPF type \(\gamma\) with \(U=2\), \((A,I,R)=(3,0,0)\)
or DPF type \(\varepsilon\) with \(U=1\), \((A,I,R)=(2,0,0)\) remain.
Recall that generally \(U+1=A+I+R\), according to
\eqref{eqn:HerbrandQuot} and
\eqref{eqn:QuinticDecomp}.
Only the fields \(N\) of type \(\gamma\) have the Polya property \(A=T\).
\end{proof}


\subsection{Numerical tables}
\label{ss:Tables}
\noindent
The supplementary paper
\lq\lq Tables of pure quintic fields\rq\rq\
\cite{Ma2}
establishes a complete classification of all \(900\)
pure metacyclic fields \(N=\mathbb{Q}(\zeta,\root{5}\of{D})\)
with normalized radicands in the range \(2\le D\le 1000\).
With the aid of PARI/GP
\cite{PARI}
and MAGMA
\cite{MAGMA}
we have determined the \textit{differential principal factorization type}, T,
of each field \(N\)
by means of other invariants \(U,A,I,R\).
After several weeks of CPU time,
the date of completion was Sep. \(17\), \(2018\).

The possible DPF types are listed in dependence on \(U,A,I,R\) in Table
\ref{tbl:DPFTypes},
where the symbol \(\times\) in the column \(\eta\), resp. \(\zeta\),
indicates the existence of a unit \(H\in U_N\), resp. \(Z\in U_N\),
such that \(\eta=N_{N/K}(H)\), resp. \(\zeta=N_{N/K}(Z)\).
The \(5\)-valuation of the \textit{unit norm index} \((U_K:N_{N/K}{U_N})\) is abbreviated by \(U\). 
The complete statistics is given in Table
\ref{tbl:Statistics}.


\renewcommand{\arraystretch}{1.0}

\begin{table}[ht]
\caption{Differential principal factorization types of pure metacyclic fields \(N\)}
\label{tbl:DPFTypes}
\begin{center}
\begin{tabular}{|c|ccc|ccc|}
\hline
 Type            & \(U\) & \(\eta\)   & \(\zeta\)  & \(A\)& \(I\) & \(R\) \\
\hline
 \(\alpha_1\)    & \(2\) & \(-\)      & \(-\)      & \(1\)& \(0\) & \(2\) \\
 \(\alpha_2\)    & \(2\) & \(-\)      & \(-\)      & \(1\)& \(1\) & \(1\) \\
 \(\alpha_3\)    & \(2\) & \(-\)      & \(-\)      & \(1\)& \(2\) & \(0\) \\
 \(\beta_1\)     & \(2\) & \(-\)      & \(-\)      & \(2\)& \(0\) & \(1\) \\
 \(\beta_2\)     & \(2\) & \(-\)      & \(-\)      & \(2\)& \(1\) & \(0\) \\
 \(\gamma\)      & \(2\) & \(-\)      & \(-\)      & \(3\)& \(0\) & \(0\) \\
\hline
 \(\delta_1\)    & \(1\) & \(\times\) & \(-\)      & \(1\)& \(0\) & \(1\) \\
 \(\delta_2\)    & \(1\) & \(\times\) & \(-\)      & \(1\)& \(1\) & \(0\) \\
 \(\varepsilon\) & \(1\) & \(\times\) & \(-\)      & \(2\)& \(0\) & \(0\) \\
\hline
 \(\zeta_1\)     & \(1\) & \(-\)      & \(\times\) & \(1\)& \(0\) & \(1\) \\
 \(\zeta_2\)     & \(1\) & \(-\)      & \(\times\) & \(1\)& \(1\) & \(0\) \\
 \(\eta\)        & \(1\) & \(-\)      & \(\times\) & \(2\)& \(0\) & \(0\) \\
\hline
 \(\vartheta\)   & \(0\) & \(\times\) & \(\times\) & \(1\)& \(0\) & \(0\) \\
\hline
\end{tabular}
\end{center}
\end{table}


\renewcommand{\arraystretch}{1.0}

\begin{table}[ht]
\caption{Absolute frequencies of differential principal factorization types}
\label{tbl:Statistics}
\begin{center}
\begin{tabular}{|c|rrrrrrrrrr|r|}
\hline
 Type            &\(100\) &\(200\) &\(300\) &\(400\) &\(500\) &\(600\) &\(700\) &\(800\) &\(900\) &\(1000\)&   \(\%\) \\
\hline
 \(\alpha_1\)    &  \(1\) &  \(2\) &  \(3\) &  \(4\) &  \(5\) &  \(5\) &  \(5\) &  \(9\) &  \(9\) &  \(9\) & \\
 \(\alpha_2\)    & \(10\) & \(17\) & \(23\) & \(30\) & \(35\) & \(42\) & \(52\) & \(57\) & \(63\) & \(75\) &  \(8.3\) \\
 \(\alpha_3\)    &  \(0\) &  \(0\) &  \(0\) &  \(1\) &  \(1\) &  \(3\) &  \(5\) &  \(5\) &  \(7\) &  \(8\) & \\
 \(\beta_1\)     &  \(0\) &  \(2\) &  \(4\) &  \(7\) &  \(8\) & \(11\) & \(15\) & \(18\) & \(22\) & \(23\) & \\
 \(\beta_2\)     &  \(7\) & \(24\) & \(40\) & \(54\) & \(80\) & \(94\) &\(108\) &\(126\) &\(146\) &\(161\) & \(17.9\) \\
 \(\gamma\)      & \(25\) & \(55\) & \(88\) &\(117\) &\(148\) &\(187\) &\(222\) &\(259\) &\(290\) &\(324\) & \(36.0\) \\
\hline
 \(\delta_1\)    &  \(0\) &  \(0\) &  \(1\) &  \(1\) &  \(3\) &  \(4\) &  \(4\) &  \(4\) &  \(6\) &  \(7\) & \\
 \(\delta_2\)    &  \(8\) & \(14\) & \(19\) & \(23\) & \(31\) & \(35\) & \(38\) & \(44\) & \(51\) & \(53\) &  \(5.9\) \\
 \(\varepsilon\) & \(26\) & \(45\) & \(67\) & \(95\) &\(110\) &\(128\) &\(150\) &\(165\) &\(184\) &\(208\) & \(23.1\) \\
\hline
 \(\zeta_1\)     &  \(0\) &  \(1\) &  \(1\) &  \(1\) &  \(1\) &  \(1\) &  \(1\) &  \(1\) &  \(1\) &  \(1\) & \\
 \(\zeta_2\)     &  \(0\) &  \(0\) &  \(0\) &  \(0\) &  \(0\) &  \(1\) &  \(1\) &  \(4\) &  \(4\) &  \(5\) & \\
 \(\eta\)        &  \(1\) &  \(2\) &  \(4\) &  \(5\) &  \(5\) &  \(6\) &  \(6\) &  \(6\) &  \(6\) &  \(7\) & \\
\hline
 \(\vartheta\)   &  \(3\) &  \(6\) &  \(8\) &  \(9\) & \(11\) & \(13\) & \(15\) & \(17\) & \(18\) & \(19\) & \\
\hline
 Total           & \(81\) &\(168\) &\(258\) &\(347\) &\(438\) &\(530\) &\(622\) &\(715\) &\(807\) &\(900\) &\(100.0\) \\
\hline
\end{tabular}
\end{center}
\end{table}


The \textit{normalized} radicand \(D=q_1^{e_1}\cdots q_s^{e_s}\)
of a pure metacyclic field \(N\) of degree \(20\)
is minimal among the powers \(D^n\), \(1\le n\le 4\),
with corresponding exponents \(e_j\) reduced modulo \(5\). 
The normalization of the radicands \(D\) provides a warranty that
all fields are pairwise non-isomorphic.

Prime factors are given for composite \(D\) only. 
Dedekind's \textit{species}, S, of radicands is refined by
distinguishing \(5\mid D\) (species 1a) and \(\gcd(5,D) = 1\) (species 1b)
among radicands \(D\not\equiv\pm 1,\pm 7\,(\mathrm{mod}\,25)\) (species 1).
By the species and factorization of \(D\),
the shape of the \textit{conductor} \(f\) is determined.
We give the fourth power \(f^4\) to avoid fractional exponents.
Additionally, the \textit{multiplicity} \(m\) indicates
the number of non-isomorphic fields sharing a common conductor \(f\).
The symbol \(V_F\) briefly denotes the \(5\)-valuation of the order \(h(F)=\#\mathrm{Cl}(F)\)
of the class group \(\mathrm{Cl}(F)\) of a number field \(F\).
By \(E\) we denote the exponent of the power in the \textit{unit index} \((U_N:U_0)=5^E\). 

An asterisk denotes the smallest radicand
with given Dedekind species, DPF type and \(5\)-class groups \(\mathrm{Cl}_5(F)\), \(F\in\lbrace L,M,N\rbrace\).
The latter are usually elementary abelian, except for the cases indicated by an additional asterisk.

Principal factors, P, are listed
when their constitution is not a consequence of the other information.
According to Theorem
\ref{thm:AmbIdl},
item (1), it suffices to give the rational integer norm of \textit{absolute} principal factors.
For \textit{intermediate} principal factors, we use the symbols
\(\mathcal{K}:=\mathcal{L}^{1-\tau}=\alpha\mathcal{O}_M\) with \(\alpha\in M\)
or \(\mathcal{L}=\lambda\mathcal{O}_M\) with a prime element \(\lambda\in M\)
(which implies \(\mathcal{L}^\tau=\lambda^\tau\mathcal{O}_M\)
and thus also \(\mathcal{K}=\lambda^{1-\tau}\mathcal{O}_M\)).
Here, \((\mathcal{L}^{1+\tau})^5=\ell\mathcal{O}_M\)
when a prime \(\ell\equiv\pm 1\,(\mathrm{mod}\,5)\) divides the radicand \(D\).
For \textit{relative} principal factors, we use the symbols
\(\mathfrak{K}_1:=\mathfrak{L}^{1-\tau^2+2\tau-2\tau^3}=A_1\mathcal{O}_N\)
and
\(\mathfrak{K}_2:=\mathfrak{L}^{1-\tau^2-2\tau+2\tau^3}=A_2\mathcal{O}_N\)
with \(A_1,A_2\in N\).
Here, \((\mathfrak{L}^{1+\tau+\tau^2+\tau^3})^5=\ell\mathcal{O}_N\)
when a prime number \(\ell\equiv +1\,(\mathrm{mod}\,5)\) divides the radicand \(D\).
(Kernel ideals in \S\
\ref{s:Idempotents}.)

The quartet \((1,2,4,5)\) indicates conditions which
either enforce a reduction of possible DPF types
or enable certain DPF types.
The lack of a prime divisor \(\ell\equiv\pm 1\,(\mathrm{mod}\,5)\)
together with the existence of a prime divisor \(q\not\equiv\pm 7\,(\mathrm{mod}\,25)\) and \(q\ne 5\) of \(D\)
is indicated by a symbol \(\times\) for the component \(1\).
In these cases, only the two DPF types \(\gamma\) and \(\varepsilon\) can occur.

A symbol \(\times\) for the component \(2\)
emphasizes a prime divisor \(\ell\equiv -1\,(\mathrm{mod}\,5)\) of \(D\)
and the possibility of intermediate principal factors in \(M\), like \(\mathcal{L}\) and \(\mathcal{K}\).
A symbol \(\times\) for the component \(4\)
emphasizes a prime divisor \(\ell\equiv +1\,(\mathrm{mod}\,5)\) of \(D\)
and the possibility of relative principal factors in \(N\), like \(\mathfrak{K}_1\) and \(\mathfrak{K}_2\).
The \(\times\) symbol is replaced by \(\otimes\) if the facility is used completely,
and by \((\times)\) if the facility is only used partially. \\
If \(D\) has only prime divisors \(q\equiv\pm 1,\pm 7\,(\mathrm{mod}\,25)\) or \(q=5\),
a symbol \(\times\) is placed in component \(5\).


\renewcommand{\arraystretch}{1.1}

\begin{table}[hb]
\caption{\(38\) pure metacyclic fields with normalized radicands \(0<D<50\)}
\label{tbl:PureQuinticFields50}
\begin{center}
\begin{tabular}{|r|rc|ccr|cccc|ccc|}
\hline
 No. &   \(D\) &             Factors & S  &             \(f^4\) &  \(m\) & \(V_L\) & \(V_M\) & \(V_N\) & \(E\) & \((1,2,4,5)\)        & T               & P               \\
\hline
   1 &  *\(2\) &                     & 1b &         \(5^2 2^4\) &  \(1\) &   \(0\) &   \(0\) &   \(0\) & \(5\) & \((\times,-,-,-)\)   & \(\varepsilon\) &                 \\
   2 &   \(3\) &                     & 1b &         \(5^2 3^4\) &  \(1\) &   \(0\) &   \(0\) &   \(0\) & \(5\) & \((\times,-,-,-)\)   & \(\varepsilon\) &                 \\
   3 &  *\(5\) &                     & 1a &             \(5^6\) &  \(1\) &   \(0\) &   \(0\) &   \(0\) & \(5\) & \((-,-,-,\otimes)\)  & \(\vartheta\)   &                 \\
   4 &  *\(6\) &        \(2\cdot 3\) & 1b &     \(5^2 2^4 3^4\) &  \(3\) &   \(0\) &   \(0\) &   \(1\) & \(6\) & \((\times,-,-,-)\)   & \(\gamma\)      &                 \\
   5 &  *\(7\) &                     & 2  &             \(7^4\) &  \(1\) &   \(0\) &   \(0\) &   \(0\) & \(5\) & \((-,-,-,\otimes)\)  & \(\vartheta\)   &                 \\
   6 & *\(10\) &        \(2\cdot 5\) & 1a &         \(5^6 2^4\) &  \(4\) &   \(0\) &   \(0\) &   \(0\) & \(5\) & \((\times,-,-,-)\)   & \(\varepsilon\) &                 \\
   7 & *\(11\) &                     & 1b &        \(5^2 11^4\) &  \(1\) &   \(1\) &   \(1\) &   \(2\) & \(3\) & \((-,-,\otimes,-)\)  & \(\alpha_2\)    & \(\mathcal{L},\mathfrak{K}_1\) \\
   8 &  \(12\) &      \(2^2\cdot 3\) & 1b &     \(5^2 2^4 3^4\) &  \(3\) &   \(0\) &   \(0\) &   \(1\) & \(6\) & \((\times,-,-,-)\)   & \(\gamma\)      &                 \\
   9 &  \(13\) &                     & 1b &        \(5^2 13^4\) &  \(1\) &   \(0\) &   \(0\) &   \(0\) & \(5\) & \((\times,-,-,-)\)   & \(\varepsilon\) &                 \\
  10 & *\(14\) &        \(2\cdot 7\) & 1b &     \(5^2 2^4 7^4\) &  \(4\) &   \(0\) &   \(0\) &   \(1\) & \(6\) & \((\times,-,-,-)\)   & \(\gamma\)      &                 \\
  11 &  \(15\) &        \(3\cdot 5\) & 1a &         \(5^6 3^4\) &  \(4\) &   \(0\) &   \(0\) &   \(0\) & \(5\) & \((\times,-,-,-)\)   & \(\varepsilon\) &                 \\
  12 &  \(17\) &                     & 1b &        \(5^2 17^4\) &  \(1\) &   \(0\) &   \(0\) &   \(0\) & \(5\) & \((\times,-,-,-)\)   & \(\varepsilon\) &                 \\
  13 & *\(18\) &      \(2\cdot 3^2\) & 2  &         \(2^4 3^4\) &  \(1\) &   \(0\) &   \(0\) &   \(0\) & \(5\) & \((\times,-,-,-)\)   & \(\varepsilon\) &                 \\
  14 & *\(19\) &                     & 1b &        \(5^2 19^4\) &  \(1\) &   \(1\) &   \(1\) &   \(2\) & \(3\) & \((-,\otimes,-,-)\)  & \(\delta_2\)    & \(\mathcal{L}\) \\
  15 &  \(20\) &      \(2^2\cdot 5\) & 1a &         \(5^6 2^4\) &  \(4\) &   \(0\) &   \(0\) &   \(0\) & \(5\) & \((\times,-,-,-)\)   & \(\varepsilon\) &                 \\
  16 &  \(21\) &        \(3\cdot 7\) & 1b &     \(5^2 3^4 7^4\) &  \(4\) &   \(0\) &   \(0\) &   \(1\) & \(6\) & \((\times,-,-,-)\)   & \(\gamma\)      &                 \\
  17 & *\(22\) &       \(2\cdot 11\) & 1b &    \(5^2 2^4 11^4\) &  \(3\) &   \(1\) &   \(1\) &   \(3\) & \(4\) & \((-,-,(\times),-)\) & \(\beta_2\)     & \(2\cdot 5,\mathcal{K}\) \\
  18 &  \(23\) &                     & 1b &        \(5^2 23^4\) &  \(1\) &   \(0\) &   \(0\) &   \(0\) & \(5\) & \((\times,-,-,-)\)   & \(\varepsilon\) &                 \\
  19 &  \(26\) &       \(2\cdot 13\) & 2  &        \(2^4 13^4\) &  \(1\) &   \(0\) &   \(0\) &   \(0\) & \(5\) & \((\times,-,-,-)\)   & \(\varepsilon\) &                 \\
  20 &  \(28\) &      \(2^2\cdot 7\) & 1b &     \(5^2 2^4 7^4\) &  \(4\) &   \(0\) &   \(0\) &   \(1\) & \(6\) & \((\times,-,-,-)\)   & \(\gamma\)      &                 \\
  21 &  \(29\) &                     & 1b &        \(5^2 29^4\) &  \(1\) &   \(1\) &   \(1\) &   \(2\) & \(3\) & \((-,\otimes,-,-)\)  & \(\delta_2\)    & \(\mathcal{L}\) \\
  22 & *\(30\) & \(2\cdot 3\cdot 5\) & 1a &     \(5^6 2^4 3^4\) & \(16\) &   \(0\) &   \(0\) &   \(1\) & \(6\) & \((\times,-,-,-)\)   & \(\gamma\)      &                 \\
  23 & *\(31\) &                     & 1b &        \(5^2 31^4\) &  \(1\) &   \(2\) &   \(3\) &   \(5\) & \(2\) & \((-,-,\otimes,-)\)  & \(\alpha_1\)    & \(\mathfrak{K}_1,\mathfrak{K}_2\) \\
  24 & *\(33\) &       \(3\cdot 11\) & 1b &    \(5^2 3^4 11^4\) &  \(3\) &   \(2\) &   \(2\) &   \(4\) & \(1\) & \((-,-,\otimes,-)\)  & \(\alpha_2\)    & \(\mathcal{K},\mathfrak{K}_2\) \\
  25 &  \(34\) &       \(2\cdot 17\) & 1b &    \(5^2 2^4 17^4\) &  \(3\) &   \(0\) &   \(0\) &   \(1\) & \(6\) & \((\times,-,-,-)\)   & \(\gamma\)      &                 \\
  26 & *\(35\) &        \(5\cdot 7\) & 1a &         \(5^6 7^4\) &  \(4\) &   \(0\) &   \(0\) &   \(1\) & \(6\) & \((-,-,-,\otimes)\)  & \(\eta\)        &                 \\
  27 &  \(37\) &                     & 1b &        \(5^2 37^4\) &  \(1\) &   \(0\) &   \(0\) &   \(0\) & \(5\) & \((\times,-,-,-)\)   & \(\varepsilon\) &                 \\
  28 & *\(38\) &       \(2\cdot 19\) & 1b &    \(5^2 2^4 19^4\) &  \(3\) &   \(1\) &   \(1\) &   \(3\) & \(4\) & \((-,\otimes,-,-)\)  & \(\beta_2\)     & \(5,\mathcal{K}\) \\
  29 &  \(39\) &       \(3\cdot 13\) & 1b &    \(5^2 3^4 13^4\) &  \(3\) &   \(0\) &   \(0\) &   \(1\) & \(6\) & \((\times,-,-,-)\)   & \(\gamma\)      &                 \\
  30 &  \(40\) &      \(2^3\cdot 5\) & 1a &         \(5^6 2^4\) &  \(4\) &   \(0\) &   \(0\) &   \(0\) & \(5\) & \((\times,-,-,-)\)   & \(\varepsilon\) &                 \\
  31 &  \(41\) &                     & 1b &        \(5^2 41^4\) &  \(1\) &   \(1\) &   \(1\) &   \(2\) & \(3\) & \((-,-,\otimes,-)\)  & \(\alpha_2\)    & \(\mathcal{L},\mathfrak{K}_2\) \\
  32 & *\(42\) & \(2\cdot 3\cdot 7\) & 1b & \(5^2 2^4 3^4 7^4\) & \(12\) &   \(1\) &   \(2\) &   \(5\) & \(6\) & \((\times,-,-,-)\)   & \(\gamma\)      & \(2\cdot 5,3\cdot 5^2\) \\
  33 &  \(43\) &                     & 2  &            \(43^4\) &  \(1\) &   \(0\) &   \(0\) &   \(0\) & \(5\) & \((-,-,-,\otimes)\)  & \(\vartheta\)   &                 \\
  34 &  \(44\) &     \(2^2\cdot 11\) & 1b &    \(5^2 2^4 11^4\) &  \(3\) &   \(1\) &   \(1\) &   \(3\) & \(4\) & \((-,-,(\times),-)\) & \(\beta_2\)     & \(2\cdot 5,\mathcal{K}\) \\
  35 &  \(45\) &      \(3^2\cdot 5\) & 1a &         \(5^6 3^4\) &  \(4\) &   \(0\) &   \(0\) &   \(0\) & \(5\) & \((\times,-,-,-)\)   & \(\varepsilon\) &                 \\
  36 &  \(46\) &       \(2\cdot 23\) & 1b &    \(5^2 2^4 23^4\) &  \(3\) &   \(0\) &   \(0\) &   \(1\) & \(6\) & \((\times,-,-,-)\)   & \(\gamma\)      &                 \\
  37 &  \(47\) &                     & 1b &        \(5^2 47^4\) &  \(1\) &   \(0\) &   \(0\) &   \(0\) & \(5\) & \((\times,-,-,-)\)   & \(\varepsilon\) &                 \\
  38 &  \(48\) &      \(2^4\cdot 3\) & 1b &     \(5^2 2^4 3^4\) &  \(3\) &   \(0\) &   \(0\) &   \(1\) & \(6\) & \((\times,-,-,-)\)   & \(\gamma\)      &                 \\
\hline
\end{tabular}
\end{center}
\end{table}

\newpage

In these cases, \(\zeta\) can occur as a norm \(N_{N/K}(Z)\) of some unit in \(Z\in U_N\).
If it actually does, the \(\times\) is replaced by \(\otimes\).
Here, we only present the first three tables of
\cite{Ma2},
Tables
\ref{tbl:PureQuinticFields50},
\ref{tbl:PureQuinticFields100},
and
\ref{tbl:PureQuinticFields150}.

\renewcommand{\arraystretch}{1.1}

\begin{table}[hb]
\caption{\(43\) pure metacyclic fields with normalized radicands \(50<D<100\)}
\label{tbl:PureQuinticFields100}
\begin{center}
\begin{tabular}{|r|rc|ccr|cccc|ccc|}
\hline
 No. &   \(D\) &               Factors & S  &              \(f^4\) &  \(m\) & \(V_L\) & \(V_M\) & \(V_N\) & \(E\) & \((1,2,4,5)\)        & T               & P               \\
\hline
  39 &  \(51\) &         \(3\cdot 17\) & 2  &         \(3^4 17^4\) &  \(1\) &   \(0\) &   \(0\) &   \(0\) & \(5\) & \((\times,-,-,-)\)   & \(\varepsilon\) &                 \\
  40 &  \(52\) &       \(2^2\cdot 13\) & 1b &     \(5^2 2^4 13^4\) &  \(3\) &   \(0\) &   \(0\) &   \(1\) & \(6\) & \((\times,-,-,-)\)   & \(\gamma\)      &                 \\
  41 &  \(53\) &                       & 1b &         \(5^2 53^4\) &  \(1\) &   \(0\) &   \(0\) &   \(0\) & \(5\) & \((\times,-,-,-)\)   & \(\varepsilon\) &                 \\
  42 & *\(55\) &         \(5\cdot 11\) & 1a &         \(5^6 11^4\) &  \(4\) &   \(1\) &   \(1\) &   \(2\) & \(3\) & \((-,-,\otimes,-)\)  & \(\alpha_2\)    & \(\mathcal{K},\mathfrak{K}_1\) \\
  43 &  \(56\) &        \(2^3\cdot 7\) & 1b &      \(5^2 2^4 7^4\) &  \(4\) &   \(0\) &   \(0\) &   \(1\) & \(6\) & \((\times,-,-,-)\)   & \(\gamma\)      &                 \\
  44 & *\(57\) &         \(3\cdot 19\) & 2  &         \(3^4 19^4\) &  \(1\) &   \(1\) &   \(1\) &   \(2\) & \(3\) & \((-,\otimes,-,-)\)  & \(\delta_2\)    & \(\mathcal{K}\) \\
  45 &  \(58\) &         \(2\cdot 29\) & 1b &     \(5^2 2^4 29^4\) &  \(3\) &   \(1\) &   \(1\) &   \(3\) & \(4\) & \((-,\otimes,-,-)\)  & \(\beta_2\)     & \(29\cdot 5^2,\mathcal{K}\) \\
  46 &  \(59\) &                       & 1b &         \(5^2 59^4\) &  \(1\) &   \(1\) &   \(1\) &   \(2\) & \(3\) & \((-,\otimes,-,-)\)  & \(\delta_2\)    & \(\mathcal{L}\) \\
  47 &  \(60\) & \(2^2\cdot 3\cdot 5\) & 1a &      \(5^6 2^4 3^4\) & \(16\) &   \(0\) &   \(0\) &   \(1\) & \(6\) & \((\times,-,-,-)\)   & \(\gamma\)      &                 \\
  48 &  \(61\) &                       & 1b &         \(5^2 61^4\) &  \(1\) &   \(1\) &   \(1\) &   \(2\) & \(3\) & \((-,-,\otimes,-)\)  & \(\alpha_2\)    & \(\mathcal{L},\mathfrak{K}_2\) \\
  49 &  \(62\) &         \(2\cdot 31\) & 1b &     \(5^2 2^4 31^4\) &  \(3\) &   \(1\) &   \(1\) &   \(3\) & \(4\) & \((-,-,(\times),-)\) & \(\beta_2\)     & \(5,\mathcal{K}\) \\
  50 &  \(63\) &        \(3^2\cdot 7\) & 1b &      \(5^2 3^4 7^4\) &  \(4\) &   \(0\) &   \(0\) &   \(1\) & \(6\) & \((\times,-,-,-)\)   & \(\gamma\)      &                 \\
  51 &  \(65\) &         \(5\cdot 13\) & 1a &         \(5^6 13^4\) &  \(4\) &   \(0\) &   \(0\) &   \(0\) & \(5\) & \((\times,-,-,-)\)   & \(\varepsilon\) &                 \\
  52 & *\(66\) &  \(2\cdot 3\cdot 11\) & 1b & \(5^2 2^4 3^4 11^4\) & \(13\) &   \(1\) &   \(2\) &   \(5\) & \(6\) & \((-,-,\times,-)\)   & \(\gamma\)      & \(2\cdot 5,3\cdot 5^3\) \\
  53 &  \(67\) &                       & 1b &         \(5^2 67^4\) &  \(1\) &   \(0\) &   \(0\) &   \(0\) & \(5\) & \((\times,-,-,-)\)   & \(\varepsilon\) &                 \\
  54 &  \(68\) &       \(2^2\cdot 17\) & 2  &         \(2^4 17^4\) &  \(1\) &   \(0\) &   \(0\) &   \(0\) & \(5\) & \((\times,-,-,-)\)   & \(\varepsilon\) &                 \\
  55 &  \(69\) &         \(3\cdot 23\) & 1b &     \(5^2 3^4 23^4\) &  \(3\) &   \(0\) &   \(0\) &   \(1\) & \(6\) & \((\times,-,-,-)\)   & \(\gamma\)      &                 \\
  56 & *\(70\) &   \(2\cdot 5\cdot 7\) & 1a &      \(5^6 2^4 7^4\) & \(16\) &   \(0\) &   \(0\) &   \(1\) & \(6\) & \((\times,-,-,-)\)   & \(\gamma\)      &                 \\
  57 &  \(71\) &                       & 1b &         \(5^2 71^4\) &  \(1\) &   \(1\) &   \(1\) &   \(2\) & \(3\) & \((-,-,\otimes,-)\)  & \(\alpha_2\)    & \(\mathcal{L},\mathfrak{K}_1\) \\
  58 &  \(73\) &                       & 1b &         \(5^2 73^4\) &  \(1\) &   \(0\) &   \(0\) &   \(0\) & \(5\) & \((\times,-,-,-)\)   & \(\varepsilon\) &                 \\
  59 &  \(74\) &         \(2\cdot 37\) & 2  &         \(2^4 37^4\) &  \(1\) &   \(0\) &   \(0\) &   \(0\) & \(5\) & \((\times,-,-,-)\)   & \(\varepsilon\) &                 \\
  60 &  \(75\) &        \(3\cdot 5^2\) & 1a &          \(5^6 3^4\) &  \(4\) &   \(0\) &   \(0\) &   \(0\) & \(5\) & \((\times,-,-,-)\)   & \(\varepsilon\) &                 \\
  61 &  \(76\) &       \(2^2\cdot 19\) & 2  &         \(2^4 19^4\) &  \(1\) &   \(1\) &   \(1\) &   \(2\) & \(3\) & \((-,\otimes,-,-)\)  & \(\delta_2\)    & \(\mathcal{K}\) \\
  62 & *\(77\) &         \(7\cdot 11\) & 1b &     \(5^2 7^4 11^4\) &  \(4\) &   \(1\) &   \(1\) &   \(3\) & \(4\) & \((-,-,(\times),-)\) & \(\beta_2\)     & \(11\cdot 5^3,\mathcal{K}\) \\
  63 & *\(78\) &  \(2\cdot 3\cdot 13\) & 1b & \(5^2 2^4 3^4 13^4\) & \(13\) &   \(1\) &   \(2\) &   \(5\) & \(6\) & \((\times,-,-,-)\)   & \(\gamma\)      & \(3,2\cdot 5^3\) \\
  64 &  \(79\) &                       & 1b &         \(5^2 79^4\) &  \(1\) &   \(1\) &   \(1\) &   \(2\) & \(3\) & \((-,\otimes,-,-)\)  & \(\delta_2\)    & \(\mathcal{L}\) \\
  65 &  \(80\) &        \(2^4\cdot 5\) & 1a &          \(5^6 2^4\) &  \(4\) &   \(0\) &   \(0\) &   \(0\) & \(5\) & \((\times,-,-,-)\)   & \(\varepsilon\) &                 \\
  66 & *\(82\) &         \(2\cdot 41\) & 2  &         \(2^4 41^4\) &  \(1\) &   \(1\) &   \(1\) &   \(2\) & \(3\) & \((-,-,\otimes,-)\)  & \(\alpha_2\)    & \(\mathcal{K},\mathfrak{K}_2\) \\
  67 &  \(83\) &                       & 1b &         \(5^2 83^4\) &  \(1\) &   \(0\) &   \(0\) &   \(0\) & \(5\) & \((\times,-,-,-)\)   & \(\varepsilon\) &                 \\
  68 &  \(84\) & \(2^2\cdot 3\cdot 7\) & 1b &  \(5^2 2^4 3^4 7^4\) & \(12\) &   \(1\) &   \(2\) &   \(5\) & \(6\) & \((\times,-,-,-)\)   & \(\gamma\)      & \(2\cdot 5,3\cdot 5\) \\
  69 &  \(85\) &         \(5\cdot 17\) & 1a &         \(5^6 17^4\) &  \(4\) &   \(0\) &   \(0\) &   \(0\) & \(5\) & \((\times,-,-,-)\)   & \(\varepsilon\) &                 \\
  70 &  \(86\) &         \(2\cdot 43\) & 1b &     \(5^2 2^4 43^4\) &  \(4\) &   \(0\) &   \(0\) &   \(1\) & \(6\) & \((\times,-,-,-)\)   & \(\gamma\)      &                 \\
  71 &  \(87\) &         \(3\cdot 29\) & 1b &     \(5^2 3^4 29^4\) &  \(3\) &   \(1\) &   \(1\) &   \(3\) & \(4\) & \((-,\otimes,-,-)\)  & \(\beta_2\)     & \(29,\mathcal{L}\) \\
  72 &  \(88\) &       \(2^3\cdot 11\) & 1b &     \(5^2 2^4 11^4\) &  \(3\) &   \(2\) &   \(2\) &   \(4\) & \(1\) & \((-,-,\otimes,-)\)  & \(\alpha_2\)    & \(\mathcal{K},\mathfrak{K}_2\) \\
  73 &  \(89\) &                       & 1b &         \(5^2 89^4\) &  \(1\) &   \(1\) &   \(1\) &   \(2\) & \(3\) & \((-,\otimes,-,-)\)  & \(\delta_2\)    & \(\mathcal{L}\) \\
  74 &  \(90\) & \(2\cdot 3^2\cdot 5\) & 1a &      \(5^6 2^4 3^4\) & \(16\) &   \(0\) &   \(0\) &   \(1\) & \(6\) & \((\times,-,-,-)\)   & \(\gamma\)      &                 \\
  75 &  \(91\) &         \(7\cdot 13\) & 1b &     \(5^2 7^4 13^4\) &  \(4\) &   \(0\) &   \(0\) &   \(1\) & \(6\) & \((\times,-,-,-)\)   & \(\gamma\)      &                 \\
  76 &  \(92\) &       \(2^2\cdot 23\) & 1b &     \(5^2 2^4 23^4\) &  \(3\) &   \(0\) &   \(0\) &   \(1\) & \(6\) & \((\times,-,-,-)\)   & \(\gamma\)      &                 \\
  77 &  \(93\) &         \(3\cdot 31\) & 2  &         \(3^4 31^4\) &  \(1\) &   \(1\) &   \(1\) &   \(2\) & \(3\) & \((-,-,\otimes,-)\)  & \(\alpha_2\)    & \(\mathcal{K},\mathfrak{K}_1\) \\
  78 &  \(94\) &         \(2\cdot 47\) & 1b &     \(5^2 2^4 47^4\) &  \(3\) &   \(0\) &   \(0\) &   \(1\) & \(6\) & \((\times,-,-,-)\)   & \(\gamma\)      &                 \\
  79 & *\(95\) &         \(5\cdot 19\) & 1a &         \(5^6 19^4\) &  \(4\) &   \(1\) &   \(1\) &   \(2\) & \(3\) & \((-,\otimes,-,-)\)  & \(\delta_2\)    & \(\mathcal{K}\) \\
  80 &  \(97\) &                       & 1b &         \(5^2 97^4\) &  \(1\) &   \(0\) &   \(0\) &   \(0\) & \(5\) & \((\times,-,-,-)\)   & \(\varepsilon\) &                 \\
  81 &  \(99\) &       \(3^2\cdot 11\) & 2  &         \(3^4 11^4\) &  \(1\) &   \(1\) &   \(1\) &   \(2\) & \(3\) & \((-,-,\otimes,-)\)  & \(\alpha_2\)    & \(\mathcal{K},\mathfrak{K}_1\) \\
\hline
\end{tabular}
\end{center}
\end{table}

\newpage

\renewcommand{\arraystretch}{1.0}

\begin{table}[ht]
\caption{\(44\) pure metacyclic fields with normalized radicands \(100<D\le 150\)}
\label{tbl:PureQuinticFields150}
\begin{center}
\begin{tabular}{|r|rc|ccr|cccc|ccc|}
\hline
 No. &    \(D\) &                Factors & S  &              \(f^4\) &  \(m\) & \(V_L\) & \(V_M\) & \(V_N\) & \(E\) & \((1,2,4,5)\)        & T               & P               \\
\hline
  82 & *\(101\) &                        & 2  &            \(101^4\) &  \(1\) &   \(1\) &   \(2\) &   \(4\) & \(5\) & \((-,-,\otimes,\otimes)\) & \(\zeta_1\) & \(\mathfrak{K}_2\) \\
  83 &  \(102\) &   \(2\cdot 3\cdot 17\) & 1b & \(5^2 2^4 3^4 17^4\) & \(13\) &   \(1\) &   \(2\) &   \(5\) & \(6\) & \((\times,-,-,-)\)   & \(\gamma\)      & \(2,3\cdot 5\)  \\
  84 &  \(103\) &                        & 1b &        \(5^2 103^4\) &  \(1\) &   \(0\) &   \(0\) &   \(0\) & \(5\) & \((\times,-,-,-)\)   & \(\varepsilon\) &                 \\
  85 &  \(104\) &        \(2^3\cdot 13\) & 1b &     \(5^2 2^4 13^4\) &  \(3\) &   \(0\) &   \(0\) &   \(1\) & \(6\) & \((\times,-,-,-)\)   & \(\gamma\)      &                 \\
  86 &  \(105\) &    \(3\cdot 5\cdot 7\) & 1a &      \(5^6 3^4 7^4\) & \(16\) &   \(0\) &   \(0\) &   \(1\) & \(6\) & \((\times,-,-,-)\)   & \(\gamma\)      &                 \\
  87 &  \(106\) &          \(2\cdot 53\) & 1b &     \(5^2 2^4 53^4\) &  \(3\) &   \(0\) &   \(0\) &   \(1\) & \(6\) & \((\times,-,-,-)\)   & \(\gamma\)      &                 \\
  88 &  \(107\) &                        & 2  &            \(107^4\) &  \(1\) &   \(0\) &   \(0\) &   \(0\) & \(5\) & \((-,-,-,\otimes)\)  & \(\vartheta\)   &                 \\
  89 &  \(109\) &                        & 1b &        \(5^2 109^4\) &  \(1\) &   \(1\) &   \(1\) &   \(2\) & \(3\) & \((-,\otimes,-,-)\)  & \(\delta_2\)    & \(\mathcal{L}\) \\
  90 & *\(110\) &   \(2\cdot 5\cdot 11\) & 1a &     \(5^6 2^4 11^4\) & \(16\) &   \(1\) &   \(1\) &   \(3\) & \(4\) & \((-,-,(\times),-)\) & \(\beta_2\)     & \(11,\mathcal{L}\) \\
  91 &  \(111\) &          \(3\cdot 37\) & 1b &     \(5^2 3^4 37^4\) &  \(3\) &   \(0\) &   \(0\) &   \(1\) & \(6\) & \((\times,-,-,-)\)   & \(\gamma\)      &                 \\
  92 &  \(112\) &         \(2^4\cdot 7\) & 1b &      \(5^2 2^4 7^4\) &  \(4\) &   \(0\) &   \(0\) &   \(1\) & \(6\) & \((\times,-,-,-)\)   & \(\gamma\)      &                 \\
  93 &  \(113\) &                        & 1b &        \(5^2 113^4\) &  \(1\) &   \(0\) &   \(0\) &   \(0\) & \(5\) & \((\times,-,-,-)\)   & \(\varepsilon\) &                 \\
  94 & *\(114\) &   \(2\cdot 3\cdot 19\) & 1b & \(5^2 2^4 3^4 19^4\) & \(13\) &   \(1\) &   \(2\) &   \(5\) & \(6\) & \((-,\times,-,-)\)   & \(\gamma\)      & \(2\cdot 5^3,3\cdot 5^3\) \\
  95 &  \(115\) &          \(5\cdot 23\) & 1a &         \(5^6 23^4\) &  \(4\) &   \(0\) &   \(0\) &   \(0\) & \(5\) & \((\times,-,-,-)\)   & \(\varepsilon\) &                 \\
  96 &  \(116\) &        \(2^2\cdot 29\) & 1b &     \(5^2 2^4 29^4\) &  \(3\) &   \(1\) &   \(1\) &   \(3\) & \(4\) & \((-,\otimes,-,-)\)  & \(\beta_2\)     & \(29\cdot 5,\mathcal{K}\) \\
  97 &  \(117\) &        \(3^2\cdot 13\) & 1b &     \(5^2 3^4 13^4\) &  \(3\) &   \(0\) &   \(0\) &   \(1\) & \(6\) & \((\times,-,-,-)\)   & \(\gamma\)      &                 \\
  98 &  \(118\) &          \(2\cdot 59\) & 2  &         \(2^4 59^4\) &  \(1\) &   \(1\) &   \(1\) &   \(2\) & \(3\) & \((-,\otimes,-,-)\)  & \(\delta_2\)    & \(\mathcal{K}\) \\
  99 &  \(119\) &          \(7\cdot 17\) & 1b &     \(5^2 7^4 17^4\) &  \(4\) &   \(0\) &   \(0\) &   \(1\) & \(6\) & \((\times,-,-,-)\)   & \(\gamma\)      &                 \\
 100 &  \(120\) &  \(2^3\cdot 3\cdot 5\) & 1a &      \(5^6 2^4 3^4\) & \(16\) &   \(0\) &   \(0\) &   \(1\) & \(6\) & \((\times,-,-,-)\)   & \(\gamma\)      &                 \\
 101 &  \(122\) &          \(2\cdot 61\) & 1b &     \(5^2 2^4 61^4\) &  \(3\) &   \(1\) &   \(1\) &   \(3\) & \(4\) & \((-,-,(\times),-)\) & \(\beta_2\)     & \(61\cdot 5^3,\mathcal{K}\) \\
 102 & *\(123\) &          \(3\cdot 41\) & 1b &     \(5^2 3^4 41^4\) &  \(3\) &   \(2\) &   \(3\) &   \(6\) & \(3\) & \((-,-,\otimes,-)\)  & \(\alpha_2\)    & \(\mathcal{K},\mathfrak{K}_2\) \\
 103 &  \(124\) &        \(2^2\cdot 31\) & 2  &         \(2^4 31^4\) &  \(1\) &   \(1\) &   \(1\) &   \(2\) & \(3\) & \((-,-,\otimes,-)\)  & \(\alpha_2\)    & \(\mathcal{K},\mathfrak{K}_1\) \\
 104 & *\(126\) &  \(2\cdot 3^2\cdot 7\) & 2  &      \(2^4 3^4 7^4\) &  \(4\) &   \(0\) &   \(0\) &   \(1\) & \(6\) & \((\times,-,-,-)\)   & \(\gamma\)      &                 \\
 105 &  \(127\) &                        & 1b &        \(5^2 127^4\) &  \(1\) &   \(0\) &   \(0\) &   \(0\) & \(5\) & \((\times,-,-,-)\)   & \(\varepsilon\) &                 \\
 106 &  \(129\) &          \(3\cdot 43\) & 1b &     \(5^2 3^4 43^4\) &  \(4\) &   \(0\) &   \(0\) &   \(1\) & \(6\) & \((\times,-,-,-)\)   & \(\gamma\)      &                 \\
 107 &  \(130\) &   \(2\cdot 5\cdot 13\) & 1a &     \(5^6 2^4 13^4\) & \(16\) &   \(0\) &   \(0\) &   \(1\) & \(6\) & \((\times,-,-,-)\)   & \(\gamma\)      &                 \\
 108 & *\(131\) &                        & 1b &        \(5^2 131^4\) &  \(1\) &   \(2\) &   \(2\) &   \(4\) & \(1\) & \((-,-,\otimes,-)\)  & \(\alpha_2\)    & \(\mathcal{L},\mathfrak{K}_2\) \\
 109 & *\(132\) & \(2^2\cdot 3\cdot 11\) & 2  &     \(2^4 3^4 11^4\) &  \(3\) &   \(1\) &   \(1\) &   \(3\) & \(4\) & \((-,-,(\times),-)\) & \(\beta_2\)     & \(11,\mathcal{L}\) \\
 110 & *\(133\) &          \(7\cdot 19\) & 1b &     \(5^2 7^4 19^4\) &  \(4\) &   \(1\) &   \(1\) &   \(3\) & \(4\) & \((-,\otimes,-,-)\)  & \(\beta_2\)     & \(7,\mathcal{L}\) \\
 111 &  \(134\) &          \(2\cdot 67\) & 1b &     \(5^2 2^4 67^4\) &  \(3\) &   \(0\) &   \(0\) &   \(1\) & \(6\) & \((\times,-,-,-)\)   & \(\gamma\)      &                 \\
 112 &  \(136\) &        \(2^3\cdot 17\) & 1b &     \(5^2 2^4 17^4\) &  \(3\) &   \(0\) &   \(0\) &   \(1\) & \(6\) & \((\times,-,-,-)\)   & \(\gamma\)      &                 \\
 113 &  \(137\) &                        & 1b &        \(5^2 137^4\) &  \(1\) &   \(0\) &   \(0\) &   \(0\) & \(5\) & \((\times,-,-,-)\)   & \(\varepsilon\) &                 \\
 114 &  \(138\) &   \(2\cdot 3\cdot 23\) & 1b & \(5^2 2^4 3^4 23^4\) & \(13\) &   \(1\) &   \(2\) &   \(5\) & \(6\) & \((\times,-,-,-)\)   & \(\gamma\)      & \(2\cdot 5^2,3\) \\
 115 & *\(139\) &                        & 1b &        \(5^2 139^4\) &  \(1\) &   \(1\) &   \(1\) &   \(3\) & \(4\) & \((-,\otimes,-,-)\)  & \(\beta_2\)     & \(5,\mathcal{L}\) \\
 116 & *\(140\) &  \(2^2\cdot 5\cdot 7\) & 1a &      \(5^6 2^4 7^4\) & \(16\) &   \(1\) &   \(2\) &   \(4\) & \(5\) & \((\times,-,-,-)\)   & \(\varepsilon\) & \(7\)           \\
 117 & *\(141\) &          \(3\cdot 47\) & 1b &     \(5^2 3^4 47^4\) &  \(3\) &   \(1\) &   \(2\) &   \(4\) & \(5\) & \((\times,-,-,-)\)   & \(\varepsilon\) & \(47\cdot 5\)   \\
 118 &  \(142\) &          \(2\cdot 71\) & 1b &     \(5^2 2^4 71^4\) &  \(3\) &   \(1\) &   \(1\) &   \(3\) & \(4\) & \((-,-,(\times),-)\) & \(\beta_2\)     & \(71\cdot 5^2,\mathcal{K}\) \\
 119 &  \(143\) &         \(11\cdot 13\) & 2  &        \(11^4 13^4\) &  \(1\) &   \(1\) &   \(1\) &   \(2\) & \(3\) & \((-,-,\otimes,-)\)  & \(\alpha_2\)    & \(\mathcal{K},\mathfrak{K}_1\) \\
 120 &  \(145\) &          \(5\cdot 29\) & 1a &         \(5^6 29^4\) &  \(4\) &   \(1\) &   \(1\) &   \(2\) & \(3\) & \((-,\otimes,-,-)\)  & \(\delta_2\)    & \(\mathcal{K}\) \\
 121 &  \(146\) &          \(2\cdot 73\) & 1b &     \(5^2 2^4 73^4\) &  \(3\) &   \(0\) &   \(0\) &   \(1\) & \(6\) & \((\times,-,-,-)\)   & \(\gamma\)      &                 \\
 122 &  \(147\) &         \(3\cdot 7^2\) & 1b &      \(5^2 3^4 7^4\) &  \(4\) &   \(0\) &   \(0\) &   \(1\) & \(6\) & \((\times,-,-,-)\)   & \(\gamma\)      &                 \\
 123 &  \(148\) &        \(2^2\cdot 37\) & 1b &     \(5^2 2^4 37^4\) &  \(3\) &   \(0\) &   \(0\) &   \(1\) & \(6\) & \((\times,-,-,-)\)   & \(\gamma\)      &                 \\
 124 & *\(149\) &                        & 2  &            \(149^4\) &  \(1\) &   \(1\) &   \(1\) &   \(2\) & \(3\) & \((-,\otimes,-,\times)\) & \(\delta_2\) & \(\mathcal{L}\) \\
 125 &  \(150\) &  \(2\cdot 3\cdot 5^2\) & 1a &      \(5^6 2^4 3^4\) & \(16\) &   \(0\) &   \(0\) &   \(1\) & \(6\) & \((\times,-,-,-)\)   & \(\gamma\)      &                 \\
\hline
\end{tabular}
\end{center}
\end{table}

\newpage

\section{Acknowledgements}
\label{s:Thanks}

\noindent
We gratefully acknowledge that our research was supported by the Austrian Science Fund (FWF):
projects J 0497-PHY and P 26008-N25.





\end{document}